\documentclass[11pt]{amsart}

\usepackage{amssymb,mathabx}

\usepackage{hyperref}

\usepackage{extarrows}

\usepackage[usenames, dvipsnames, svgnames]{xcolor}

\usepackage{enumerate} % for using different counters
\usepackage{helvet} % for Sans Serif fonts
\usepackage{courier} % default ttdefault
\usepackage{mathrsfs} % for fancy calligraphic fonts

\reversemarginpar %marginal notes appear on the left side margin
\normalmarginpar %switches back to right side margin

\let\oldmarginpar\marginpar %changes the font of margin text
\renewcommand\marginpar[1]{\-\oldmarginpar{\raggedright\small\sf #1}}

\title[Ramification of Galois representations]{Quantitative level lowering for Galois representations}

\author{Najmuddin Fakhruddin} 
\author{Chandrashekhar Khare}
\author{Ravi Ramakrishna}
 
 \address{School of Mathematics, Tata Institute of Fundamental Research, 
Homi Bhabha Road, Mumbai 400005, India} 
\address{Department of Mathematics, UCLA, Los Angeles, CA 90095-1555,
  USA}
\address{Department of Mathematics, Cornell University, Ithaca, USA}

\email{naf@math.tifr.res.in, shekhar@math.ucla.edu, ravi@math.cornell.edu}

\newcommand{\nc}{\newcommand}

\nc{\rnc}{\renewcommand}

\nc{\bs}{\backslash}
\nc{\te}{\otimes}
\nc{\lf}{\lfloor} %for round down
\nc{\rf}{\rfloor}
\nc{\lc}{\lceil}  %for round up
\nc{\rc}{\rceil}
\nc{\lr}{\longrightarrow}
\nc{\sr}{\stackrel}
\nc{\dar}{\dashrightarrow}
\nc{\thra}{\twoheadrightarrow}

\nc{\la}{\langle}
\nc{\ra}{\rangle} 

\nc{\ms}{\mathscr}
\nc{\mc}{\mathcal}
\nc{\mb}{\mathbb}
\nc{\mf}{\mathbf}
\nc{\mr}{\mathrm}
\nc{\mg}{\mathfrak}

\nc{\bP}{\mathbb{P}}
\rnc{\P}{\mathbb{P}}
\nc{\Q}{\mathbb{Q}}
\nc{\Z}{\mathbb{Z}}
\nc{\C}{\mathbb{C}}
\nc{\R}{\mathbb{R}}
\nc{\A}{\mathbb{A}}
\nc{\V}{\mathbb{V}}
\nc{\W}{\mathbb{W}}
\nc{\N}{\mathbb{N}}
\nc{\D}{\mathbb{D}}
\nc{\G}{\mathbb{G}}
\nc{\F}{\mathbb{F}}
\nc{\qb}{\overline{\mathbb{Q}}}

\def\aQ{\overline{{\mathbb Q}}}

\def\aQp{\overline{\mathbb Q}_p}
\def\ra{\rightarrow}

\def\rhobar{ {\bar {\rho} } }

\def\Q{{\mathbb Q}}

\def\Z{{\mathbb Z}}

\def\T{{\mathbb T}}
\def\D{ {\mathcal D}}
\def\L{ {\mathcal L}}

\def\cO{{\mathcal O}}
\def\spec{{\rm Spec}}
\def\LP{ \left(\begin{array}}
\def\RP{ \end{array}\right)}
\def\LB{ \left[\begin{array}}
\def\RB{ \end{array}\right]}
\def\la{ \Lambda }

\def\ad{Ad^0\bar{\rho}}

\def\Ann{{\rm Ann}}
\def\Oc{\mathcal O}

\DeclareFontFamily{U}{wncy}{}
\DeclareFontShape{U}{wncy}{m}{n}{<->wncyr10}{}
\DeclareSymbolFont{mcy}{U}{wncy}{m}{n}
\DeclareMathSymbol{\Sh}{\mathord}{mcy}{"58} 

\nc{\del}{\partial}

\nc{\wt}{\widetilde}
\nc{\wh}{\widehat}
\nc{\ov}{\overline}
\nc{\un}{\underline}

\nc{\aff}{{\A}^1}
\nc{\naive}{\!\sim_n}

\nc{\Spec}{\mr{Spec}}

\nc{\omx}{\omega_X}

\nc{\ep}{\epsilon}
\nc{\ve}{\varepsilon}
\nc{\vt}{\vartheta}

\rnc{\l}{\lambda}
\rnc{\k}{\kappa}

\nc{\ovl}{\ov{\lambda}}
\nc{\vl}{\mb{V}_{\ovl}}
\nc{\dl}{\mb{D}_{\ovl}}
\nc{\mnb}{\ov{\mr{M}}_{0,n}}
\nc{\mn}{\mr{M}_{0,n}}
\nc{\mel}{\ov{\mr{M}}_{1,1}}
\nc{\mfb}{\ov{\mr{M}}_{0,4}}
\nc{\mof}{\mr{M}_{0,4}}
\nc{\mgnb}{\ov{\mr{M}}_{g,n}}
\nc{\mgn}{\ov{\mr{M}}_{g,n}}
\nc{\omc}{\ov{\mr{M}}}

\nc{\im}{\operatorname{Im}}
\nc{\ann}{\operatorname{Ann}}

\rnc{\ker}{\operatorname{Ker}}

\nc{\Ext}{\operatorname{Ext}}
\nc{\Hom}{\operatorname{Hom}}
\nc{\Fitt}{\operatorname{Fitt}}
\nc{\depth}{\operatorname{depth}}

\rnc{\sl}{\shoveleft}

\nc{\pic}{\operatorname{Pic}}
\nc{\gal}{\operatorname{Gal}}
\nc{\fr}{\operatorname{Fr}}
\nc{\ed}{\operatorname{ed}}
\nc{\rank}{\operatorname{rank}}
\nc{\h}{\operatorname{H}}
\nc{\ch}{\operatorname{char}}
\nc{\sw}{\operatorname{sw}}
\nc{\rsw}{\operatorname{rsw}}
\nc{\supp}{\operatorname{supp}}

\rnc{\ker}{\operatorname{Ker}}

\nc{\Mor}{\operatorname{Mor}}
\nc{\Per}{\operatorname{Per}}
\nc{\prep}{\operatorname{Prep}}
\nc{\End}{\operatorname{End}}
\nc{\Orb}{\operatorname{Orb}}

\nc{\rbar}{\bar{\rho}}
\nc{\rbarg}{\bar{\rho}(\mg{g}^{\mr{der}})}

\DeclareFontFamily{U}{wncy}{}
\DeclareFontShape{U}{wncy}{m}{n}{<->wncyr10}{}
\DeclareSymbolFont{mcy}{U}{wncy}{m}{n}
\DeclareMathSymbol{\Sh}{\mathord}{mcy}{"58}

\newtheorem{thm}{Theorem}[section]
\newtheorem{prop}[thm]{Proposition}

\newtheorem{cor}[thm]{Corollary}
\newtheorem{lem}[thm]{Lemma}

\theoremstyle{definition}
\newtheorem{defn}[thm]{Definition}
\newtheorem{ques}[thm]{Question}
\newtheorem{rem}[thm]{Remark}

\newtheorem{claim}[thm]{Claim}

\newtheorem{ass}[thm]{Assumption}

\newtheorem{ex}[thm]{Example}

\newtheorem{thmalph}{Theorem}

\numberwithin{equation}{section}

\begin{document}

\begin{abstract}
  We use Galois cohomology methods to produce optimal mod $p^d$ level
  lowering congruences to a $p$-adic Galois representation that we
  construct as a well chosen lift of a given residual mod $p$
  representation. Using our explicit Galois cohomology methods,  for $F$ a number field,  
  $\Gamma_F$  its absolute Galois group and $G$ a reductive group,  $k$ a finite field,   and a suitable 
  representation $\rhobar: \Gamma_F \ra G(k)$, ramified  at a
  finite set of primes $S$, we construct  under favorable conditions 
   lifts $\rho$, $\{\rho^q\}$ of $\rhobar$ to $G(W(k))$
  for $q \in Q$ with $Q$ a finite set of places of $F$. The lifts $\{\rho^q\}$ have  the following properties: $\rho: \Gamma_F \to G(W(k))$ is
  ramified precisely at  $S \cup Q$; for $q \in Q$,  $\rho^q:G_F \ra G(W(k))$ is unramified outside
  $S \cup Q \backslash \{q\}$ and $\rho$ and $\rho^q$ are congruent
  mod $p^d$  if  $\rho$ mod $p^d$ is   unramified
  at $q$. Furthermore, the Galois representations $\{\rho^q\}$ are
  ``independent''.

\end{abstract}
\maketitle
% \tableofcontents

\section{Introduction}

The study of congruences between modular forms is an important ingredient in studying
the relationship between deformation rings of Galois representations and Hecke algebras.  The work of Ribet on level raising and level lowering congruences between modular forms, cf. \cite{Ribet}, \cite{Ribet1}, used cohomological and  arithmetic-geometric properties of Jacobians of modular curves.  More recently the lifting method of \cite{KW}, and its generalizations in   \cite{BLGGT},  produces level raising and lowering congruences between Galois representations  (and {\it a fortiori}  automorphic forms) using automorphy lifting methods which have their origin in \cite{W}.  

 \subsection{Our results}
 In this paper we use Galois cohomology to produce congruences between
 Galois representations which do not seem accessible by the usual
 methods. %(Once these congruences are shown to exist between Galois representations they should in the situations we study result in congruences between ($p$-adic) automorphic forms using the recent progress on automorphy lifting theorems as in \cite{ACC+}.)
%We hope these might shed light on structure of deformation rings eventually. 
For instance, in the classical case of irreducible odd representations
$\rhobar:\Gamma_{\Q} \ra GL_2(k)$ with $k$ a finite field of
characteristic $p>3$, we prove the following theorem which does not
seem amenable to geometric (cf.  \cite{Ribet}, \cite{Ribet1})  or automorphy lifting  (cf. \cite{KW})  methods.  Here for a field $F$ we denote by $\Gamma_F$ its absolute Galois group.

\begin{thmalph}[see Theorem \ref{thm:augmentations}]\label{thm:example}
{\it Let 
$\rhobar:\Gamma_{\Q} \ra GL_2(\F_p)$  be an odd, irreducible modular mod $p>3$ surjective representation of  square-free conductor $N(\rhobar)=N$ and finite flat at $p$ with determinant the mod $p$ cyclotomic character.  Assume that the minimal Selmer group is non-zero. Then  there are  newforms
$f \in S_2(\Gamma_0(N\Pi_{q \in Q}q))$, with $Q=\{q_1,\cdots,q_r\}$  a finite ordered   set of primes that are  coprime to $Np$, such that the corresponding Galois
representation associated to an embedding
$\iota: \aQ \hookrightarrow \aQp$,
$\rho_{f,\iota}:\Gamma_{\Q} \rightarrow GL_2(\Z_p)$ lifts $\rhobar$ and has the following  properties:
\begin{itemize}
\item 
$\rho_{f,\iota}(\tau_q)$, for $\tau_q$ a generator of the $\Z_p$-quotient of the inertia group $I_q$ at $q$,   is of the form
$\begin{pmatrix}
  1 &  p^d\\
  0 & 1 \\
\end{pmatrix},$ for all $q \in Q$ and for some integer $d \geq 1$.
\item For each $1 \leq i \leq r$, there
is a subset $Q_i$ of $Q$ that omits $q_i$ and contains
$\{q_1,\cdots,q_{i-1}\}$, and there is a newform $f_i$ in
$S_2(\Gamma_0(N\Pi_{q \in Q_i}q))$, new at $Q_i$, with $f_i$ congruent to $f$
modulo $p^d$.
\end{itemize}  }
\end{thmalph}

We regard our result as a {\it quantitative} level
lowering result for the $p$-adic Galois representation
$\rho_{f,\iota}$. As we show in \S \ref{sec:R=T}, the web of  ``independent''
level-lowering congruences we produce in such a result can be used in
proving automorphy lifting results.   Note that the terminology of
quantitative level-lowering has been used earlier in \cite{PW}.  (The
technical condition of the initial dual Selmer being non-trivial for
our applications is not too burdensome, otherwise the minimal
deformation ring is smooth and hence proven to be isomorphic to the
corresponding Hecke algebra by the results of \cite{DT}.)
  
Theorem \ref{thm:example} is a particular case of the main result of
this paper, Theorem \ref{thm:main}, which we now state. We refer to \S
\ref{sec:n} for the basic notation used in the theorem, and the main
body of the paper for the precise assumptions.
 
\begin{thmalph} [see Theorem \ref{thm:main}]\label{thm:main-intro}
{\it   Let $F$ be any number field, $S$ a finite set of finite primes of
  $F$ and $\rbar: \Gamma_S \to G(k)$ a continuous representation.  We
  assume $\rbar$ has ``large image'' (i.e., satisfies Assumption
  \ref{ass:1}) and $p=\ch(k)$ is sufficiently large (i.e., the
  hypotheses of Proposition \ref{prop:nice} hold).  We assume further
  that we are given smooth local conditions
  $\mc{N} = \{\mc{N}_v\}_{v \in S}$ for $\rbar$ which are
  ``balanced'', i.e., the dimensions of the corresponding Selmer and
  dual Selmer groups are equal (see Assumption \ref{ass:2}) to an
  integer $n$ which we assume is nonzero.  Then there exists an ordered  set
  $Q = \{q_1,q_2,\dots,q_m\}$ of nice primes of $F$ (see Definition
  \ref{def:nice}), $m \in \{n+1,n+3\}$, and an integer $d \geq 1$ such
  that
\begin{itemize}
\item $Q$ is auxiliary (see Definition \ref{def:aux}) and the versal
  deformation
  $\rho_{S \cup Q}^{Q-new}: \Gamma_{S \cup Q} \to G(R_{S \cup
    Q}^{Q-new}) = G(W(k))$ is ramified mod $p^{d+1}$ at all $q \in Q$.
\item for each $1 \leq i \leq m$ there is an auxiliary set $Q_i
  \subset Q$ satisfying
  \begin{itemize}
  \item $\{q_1,q_2,\dots,q_{i-1}\} \subset Q_i$,
  \item $q_i \notin Q_i$,
  \item $\rho_{S \cup Q_i}^{Q_i-new} \equiv \rho_{S \cup Q}^{Q-new}$
    mod $p^{d}$, $\rho_{S \cup Q_i}^{Q_i-new}$ mod $p^{d}$ is
    special (see Definition \ref{def:special}) at $q_i$ but
    $\rho_{S \cup Q_i}^{Q_i-new}$ mod $p^{d+1}$ is not special at $q_i$.
  \end{itemize}
\end{itemize}}
\end{thmalph}
%The proof of Theorem \ref{thm:main} uses only Galois cohomology.

\smallskip

Our Galois cohomological method to produce congruences has its origin
in the method of lifting mod $p$ Galois representations to
characteristic $0$ introduced in \cite{Ravi}.  The last author showed
in loc.~cit.~that given an odd irreducible representation
$\rhobar: \Gamma_{\Q} \rightarrow GL_2(k)$ satisfying mild technical conditions, with $k$ a finite field of
characteristic $p$, there is a geometric lift that is uniquely
determined by suitably chosen ramification conditions (e.g., the
condition of being Steinberg at the finite set $Q$ of auxiliary
primes). This method has been generalized by Patrikis in \cite{pat-ex}
to the setting of {\it odd} representations
$\rhobar: \Gamma_F \rightarrow G(k)$, with $G$ a reductive algebraic
group over $W(k)$ and $F$ a CM field. (The results in loc.~cit.~are proved under several
technical hypotheses which we will need to assume in our work as well,
see Theorem \ref{thm:main} below: for example, the image of $\rhobar$
is assumed to be large in a suitable sense. We note that the recent
paper \cite{FKP} addresses lifting residual representations without
assuming they have large image.)
 
In the lifting method of \cite{Ravi} in the classical case (and its
generalization in \cite{pat-ex}) one considers deformations of $\rhobar$ (to complete Noetherian local $W(k)$-algebras with residue field $k$)
that are unramified outside $S \cup Q$ with $Q$ a suitable auxiliary
finite set of places. Here $S$ consists of places at which $ \rhobar$
is ramified and the places above $p,\infty$ with balanced (a condition that we explain below), smooth,
minimal deformation conditions at these places.  The set $Q$ consists
of residually Steinberg places $v$, i.e., $v$ such that (in the
2-dimensional case) $\rhobar({{\rm Frob}}_v)$ has eigenvalues
$\alpha_v,\beta_v$ with ratio $\alpha_v/\beta_v=q_v \neq \pm 1 $ mod
$p$. The deformation condition for $v \in Q$ is  the Steinberg condition at
$v$, i.e., deformations such that the image of (a lift of)
${{\rm Frob}}_v$ has eigenvalues $\tilde{\alpha_v},\tilde{\beta_v}$,
which lift $\alpha_v,\beta_v$ and their ratio
$\tilde{\alpha_v}/\tilde{\beta_v}=q_v$. This condition is smooth, is
{\it non-minimal},  has  Zariski tangent space of  dimension $h^0(\Gamma_v,\ad)$, and cuts out a subspace of the local cohomology
that intersects the unramified cohomology in a codimension one
subspace.
The corresponding deformation problem is balanced and thus  the resulting
deformation ring $R_{S \cup Q}^{Q-new}$ has a presentation of the form
$W(k)[[X_1,\cdots,X_r]]/(f_1,\cdots, f_r)$ with
$r=h^1_\L=h^1_{\L^\perp}$, the dimension of the  corresponding Selmer and dual Selmer groups.

The main innovation of \cite{Ravi}, as formulated in \cite{T03}, is a method of killing the dual
Selmer group for a residual representation $\rhobar$ when deformation
conditions considered at auxiliary places are smooth.  (Killing dual
Selmer for $\rhobar$ with smooth conditions is essential to lift
$\rhobar$ to characteristic 0. The method  in \cite{W}  
to kill dual Selmer allows  unrestricted  ramification
at auxiliary primes considered there, and the vanishing of the  dual Selmer group  achieved in \cite{W}   does not imply the existence of liftings of
$\rhobar$, as the unrestricted local conditions at the auxiliary primes considered in \cite{W} are not smooth.) It produces a set $Q$ consisting of finitely many
residually Steinberg places for $\rhobar$ as above with the Steinberg
deformation condition at these places, such that the global
deformation problem is smooth. The corresponding balanced smooth
deformation ring $R_{S\cup Q}^{Q-new}$ is $W(k)$.

The work of \cite{KR1} in the classical case, and its generalization
in \cite{Z} and \cite{Patrikis-Ann} to the context of $G$-valued representations considered in
\cite{pat-ex}, shows that we can choose $Q$ with the further property
that the universal representation $\rho^{Q-new}_{S\cup Q}$ is ramified at primes
in $Q$. This gives that the augmentation
$\pi:R_{S\cup Q} \rightarrow R_{S\cup Q}^{Q-new}$,
where $R_{S\cup Q}$ parametrizes
deformations 
with  arbitrary ramification allowed at $Q$ but the same balanced  deformation conditions at $v \in S$,  
 has finite
cotangent space $\ker (\pi)/\ker(\pi)^2$, and thus the $Q$-new quotient  $R_{S \cup Q}^{Q-new}=W(k)$ is an irreducible
component of $R_{S\cup Q}$. The main work in this paper is to construct such
$Q$ with the additional property that there are congruences of $\pi$
to other augmentations which can be thought of as level lowering
congruences.

Building on \cite{Ravi} and its generalization in  \cite{pat-ex} to $G$-valued representations, and upon
the ``doubling method'' of \cite{klr},  we prove Theorem \ref{thm:main}  below in \S \ref{sec:Main}.   
The second author had conceived of  a result like Theorem \ref{thm:example}   at the time of writing \cite{K}, but was unable to prove it at the time. Our proof draws upon developments of the method of the third author \cite{Ravi} since that time (cf. \cite{klr}).

\subsection{Motivation}

One of the motivations of our work is an attempt to find a technique
to prove that certain deformation rings are of the expected
dimension. Let $\rhobar:\Gamma_F \rightarrow G(k)$ be a continuous,
irreducible representation with $G$ a reductive algebraic group over
$W(k)$.  Let $S$ be a finite set of places $S$ of $F$, where $S$
includes the places of $F$ above $p$, all the infinite places of $F$, and places at which $\rhobar$ is
ramified.  We consider deformations $\rho:\Gamma_F \rightarrow G(R)$
of $\rhobar$ to complete noetherian local $W(k)$-algebras $R$ with
residue field $k$. At places outside $S$ the deformation condition is
that of being unramified, and at places in $S$ the deformations $\D_v$
satisfy an infinitesimal lifting property (i.e., smoothness), and the
corresponding tangent spaces $\L_v$ satisfy one of the following:
\begin{itemize}
\item[--] {\it favorable} condition
  $\sum_{v \in S} {\rm dim}_k \L_v \geq \sum_{v \in S}
  h^0(\Gamma_v,\rbarg)$, or the more restrictive,
\item[--] {\it balanced} condition
  $\sum_{v \in S} {\rm dim}_k \L_v =\sum_{v \in S}
  h^0(\Gamma_v,\rbarg)$.
\end{itemize}
Here $\rbarg$ denotes the Lie algebra of the derived group of $G$ with
the action of $\Gamma_F$ induced by composing $\rbar$ with the adjoint
representation of $G$.  Typically this is achieved by imposing
deformation conditions at places of $S$ not above $p$ (and $\infty$)
such that ${\rm dim}_k \L_v=h^0(\Gamma_v,\rbarg)$, and at places above
$p$ that
$\sum_{v \in S_p} {\rm dim}_k \L_v \geq \sum_{v \in
  S_p}h^0(\Gamma_v,\rbarg) + \sum_{v \in S_\infty}
h^0(\Gamma_v,\rbarg)$.  (The spaces $\L_v$ are trivial for $v$ the
infinite places of $F$ (for $p>2$).)  We have the corresponding Selmer
and dual Selmer groups $H^1_\L(\Gamma_{F}, \rbarg)$ and
$H^1_{\L^\perp}(\Gamma_{F},\rbarg)$ whose dimension over $k$ is
denoted by $h^1_\L$ and $ h^1_{\L^\perp}$.  When deformation
conditions are favorable we have the inequality
$h^1_\L \geq h^1_{\L^\perp}$, and in the balanced case the equality
$h^1_\L = h^1_{\L^\perp}$.  We remark that known methods to produce
geometric  liftings succeed only when these dimensions satisfy
the favorable inequality $h^1_\L \geq h^1_{\L^\perp}$.

As we are assuming the deformation conditions are smooth  and balanced at places in
$S$ (e.g., minimal deformation conditions), the corresponding
deformation ring $R_S$ has a presentation as a quotient
${W(k)}[[X_1,\cdots, X_r]]/(f_1,\cdots,f_r)$, with
$r=h_\L^1=h_{\L^\perp}^1$. In favorable situations one expects that
if in addition the local conditions are natural (unrestricted,
ordinary,\ldots) then the corresponding deformation ring is a finite
flat complete intersection of relative dimension
$h_\L^1-h_{\L^\perp}^1$ over $W(k)$. Thus  in balanced situations one
might expect that in many cases the ring $R_S$ is a finite, and flat,
$W(k)$-module. (This may not always be the case as was pointed out to
us by G.~Boxer and F.~Calegari. \footnote{Here is an  example
  suggested by Boxer and Calegari which  illustrates  that favorable
  deformation conditions may not always imply that the generic fiber
  of the corresponding deformation ring has the expected relative  dimension
  $h_\L^1-h_{\L^\perp}^1$.  Consider 2-dimensional ordinary
  representations for a CM field $K$ with totally real subfield $F$,
  $\rhobar:\Gamma_K \ra GL_2(k)$, and impose the favorable deformation
  condition that the weights are equal at complex conjugate places of
  $K$ above $p$.  Then the dimension of the ordinary deformation ring
  of $\rhobar$ will be bigger than the expected dimension ($=0$) if
  $\rhobar$ is restricted from the totally real subfield. The locus of
  lifts which come by restriction from $F$ gives a component of
  dimension $[F:\Q]$.
%(and maybe also if it is dihedral?). 
 % (ii) Another example arises with $F$ a totally real field and a
 % prime $p$ which splits completely in $F$. We consider deformations
 % of a suitable dihedral representation
 % $\rhobar:\Gamma_F \ra GL_2(k)$, induced from a CM extension in which
  %the primes of $F$ above $p$ split, which is unramified at $s$ of the
  %primes above $p$, and ordinary (but not fixed weight) at the rest.
  %If $[F:\Q]=d$, and $s$ is the number of places above $p$ at which we
  %impose on deformations the unramified condition, then the expected
  %dimension of the generic fiber of the ordinary deformation ring is
  %$d-2s$.  However, the dihedral locus could be as large as $d-s$
  %dimensional.
The ``discrepancy'' of the actual dimension from the expected  one   is ``explained by functoriality''.})

Let $R_{S\cup Q}^{Q-new}$ be the deformation ring   that parametrizes deformations $\rho:\Gamma_{S \cup Q} \ra G(A)$ of $\rhobar$ (with fixed multiplier $\nu$),  with $A$ a complete Noetherian local $W(k)$-algebra with residue field $k$,    that are unramified outside $S \cup Q$ and satisfy  suitable balanced  deformation conditions $\mc{D}_v$ at $v \in S \cup Q$.   Further let $R_{S \cup Q}$ be the deformation ring that parametrizes deformations   $\rho:\Gamma_{S \cup Q} \ra G(A)$ of $\rhobar$  with the same balanced conditions $\mc{D}_v$ at $v \in S$,  but  the conditions at $ v \in Q$ are relaxed  (cf. \S 3).
Using Theorem \ref{thm:main} we get an auxiliary set
$Q$ such that $R_{S\cup Q}^{Q-new}$ is $W(k)$.   We thus get  an
augmentation $\pi:R_{S \cup Q} \ra R_{S\cup Q}^{Q-new}=W(k)$ which gives rise to the
Galois representation
$\rho_{S \cup Q}^{Q-new}: \Gamma_{S \cup Q} \to G(R_{S \cup
  Q}^{Q-new}) = G(W(k))$ such that  
  that $\rho_{S \cup Q}^{Q-new}$ is ramified at all the primes in $Q$. Furthermore there is an integer $d$ such that $\rho_{S \cup Q}^{Q-new}$ mod $p^d$  is unramified at all the primes in $Q$, and  $\rho_{S \cup Q}^{Q-new}$ mod $p^{d+1}$  is ramified at all the primes in $Q$. 
  
  This 
allows one to prove that the cotangent space at $\pi$ is finite, and
in fact we can compute it to be exactly $\oplus_{q \in Q} {W(k)/p^{d}}$ (Theorem
\ref{thm:selmer}). Wiles's numerical criterion, which we attempt to
enhance in \S \ref{sec:app} (Proposition \ref{prop:ci}), suggests
that if we can verify his numerical coincidence
$\# \Phi_\pi=\# {W(k)}/\eta_\pi$, with $\Phi_\pi$ the cotangent space of
$\pi$ and $\eta_\pi = \pi( \Ann_{R_{S\cup Q}} (\ker \pi))$, in this situation we
could get a grip on $R_{S\cup Q}$ (and hence $R_S$). The depth of
$ \pi( \Ann_{R_{S\cup Q}} (\ker \pi))$ is related to congruences between $\pi$ and
other augmentations, and we produce in a sense optimal level lowering
congruences to $\rho_{S \cup Q}^{Q-new}$. This still falls short of
getting the Wiles numerical coincidence, but is suggestive of it.  
\smallskip

    Here is a description of the contents of the paper. In the key \S \ref{sec:Main} we prove our main theorem Theorem \ref{thm:main}.  In \S \ref{deformation} we give an exact computation of Selmer groups associated to the Galois representation we construct and prove that it is a smooth point of the relevant deformation ring.  In \S \ref{sec:R=T}  we illustrate the  use of  our main theorem, in conjuction with  Wiles's numerical isomorphism criterion,  to  prove automorphy lifting in the classical case. In \S \ref{sec:app} we refine Wiles's numerical isomorphism criterion, which was prompted by the motivation we mention above.

%One
%consequence of the theorem may be illustrated by the following particular classical  illustrative  case. 
%A new congruence to $\rho_{f,\iota}$  is proved everytime we relax ( in  a graduated manner) the local conditions at primes in $Q$,  as we move  outwards in the {\it matryoshka} doll formed by  the expanding   deformation rings with increasingly less constraints put at primes in $Q$.

\subsection{Notation} \label{sec:n}

For a field $K$, $\Gamma_K$ will denote the absolute Galois group of
$K$. If $F$ is a number field and $q$ is a prime of $\mc{O}_F$, the
ring of integers of $F$, then the absolute Galois group of the
completion of $F$ at $q$ will be denoted by $\Gamma_q$ with
$I_q \subset \Gamma_q$ denoting the inertia subgroup; we will
implicitly identify, by making a choice, $\Gamma_q$ with a
decomposition subgroup of $\Gamma_F$. For $X$ a finite set of (finite)
primes of $\mc{O}_F$, $\Gamma_X$ denotes the group $\gal(F_X/F)$,
where $F_X$ is the maximal extension of $F$ (in a fixed algebraic
closure) unramified outside primes in $X$ (and the infinite
primes). Throughout this paper we fix a rational prime $p$ and denote
let $\kappa: \Gamma_F \to \Z_p^{\times}$ be the $p$-adic cyclotomic
character; we also use $\kappa$ to denote the characters arising from
$\kappa$ by restriction to subgroups of $\Gamma_F$, on quotients of
$\Gamma_F$ by subgroups containing $\ker(\kappa)$, or by extension of
scalars to any $\Z_p$-algebra $R$. We let $k$ be a finite field of
characteristic $p$ and $W(k)$ the ring of Witt vectors of $k$.

For $M$ a finitely generated free module over a $\Z_p$-algebra $R$ we
let $M^{\vee}$ denote the module $\Hom_R(M,R)$. If $M$ also has a
Galois action we denote by $M^*$ the Galois module
$\Hom_R(M, R(\kappa))$. For $M$ a finite $W(k)$-module,
$\ell(M)$ denotes the length of $M$.

Let $G$ be a split (connected) reductive group scheme over $W(k)$ and
let $G^{\mr{sc}}$ be the simply connected cover of the derived group
scheme $G^{\mr{der}} =[G,G]$; we assume that $G^{\mr{der}}$ is almost
simple. (We could work with somewhat more general $G$, but we have made
the above assumptions in order to avoid technicalities.)
Let $\mu: G \to C$ be the maximal torus quotient of $G$. Let $\mg{g}$
be the Lie algebra of $G$ and $\mg{g}^{\mr{der}}$ the Lie algebra of
$G^{\mr{der}}$; it is also the kernel of the map on Lie algebras
induced by $\mu$. We assume throughout that $p$ is greater than the
Coxeter number of $G$: this implies that the Killing form on
$\mg{g}^{\mr{der}}$ is nondegenerate (by, e.g.,
\cite[I.4.8]{springer-steinberg}).

\subsection{Acknowledgements}

N.F. was supported by the DAE, Government of India, PIC
12-R\&D-TFR-5.01-0500. C.K.  was partially supported by NSF grant
DMS-1601692 and a Humboldt Research Award, and thanks TIFR for its
hospitality in periods during which the work was carried out.  R.R.
was partially supported by Simons Collaboration grant \#524863. We
thank Mohammed Zuhair for helpful conversations. 

We thank the referee for carefully reading the manuscript and making many helpful suggestions.

\section{Quantitative level lowering for Galois representations}\label{sec:Main}

The main result of this section is Theorem \ref{thm:main}. We begin by
introducing nice primes and prove several local results related to
them.  After this, the main theorem is stated precisely in \S
\ref{sec:main}. We then prove some global lemmas and the main steps of
the proof are given in \S \ref{sec:mat}, \S \ref{sec:ker} and
\S \ref{sec:proof}.

\smallskip

Let $F$ be a number field and and $\rbar: \Gamma_F \to G(k)$ a
continuous homomorphism. Let $S$ be a finite set of places of
$\mc{O}_F$ including all primes ramified in $\rbar$, all primes
lying over $p$ and all infinite places; henceforth we view $\rbar$ as a map
$\Gamma_S \to G(k)$. Fix $\nu: \Gamma_S \to C(W(k))$ a continuous lift
of the map $\mu \circ \rbar: \Gamma_S \to C(k)$ and use the same
notation for all maps induced by $\nu$ (in the same way as for
$\kappa$). All (local and global) deformations that we consider will
have ``fixed determinant'', i.e., give $\nu$ after composing with
$\mu$.

Denote by $\rbarg$ the composition of $\rbar$ with the adjoint
representation of $G(k)$ on $\mg{g}_{k}^{\mr{der}}$.  The tangent
space of the fixed determinant deformation functor corresponding to
$\rbar$ is given by $H^1(\Gamma, \rbarg)$, where $\Gamma$ is a local
or global Galois group.

From now on, unless explicitly stated otherwise, we make the following
assumption on $\rbar$:
\begin{ass} \label{ass:1} %
$\rbar(\Gamma_S)$ contains the image of $G^{\mr{sc}}(k)$ in
  $G(k)$.
\end{ass}

\begin{rem} \label{rem:abs} %
  Since $p$ is greater than the Coxeter number of $G$, $p$ is a ``very
  good'' prime for $G$, hence $\mg{g}^{\mr{der}}_k$ is an irreducible
  representation of $G$; this is a consequence of the simplicity of
  the Lie algebra $\mg{g}^{\mr{der}}_k$ (e.g., \cite[\S
  2.6]{steinberg-auto}). It follows from \cite[Theorem
  43]{steinb-chev} that $\mg{g}_k^{\mr{der}}$ is an absolutely
  irreducible representation of $G^{\mr{sc}}(k)$; moreover, the only
  $\F_p[G^{\mr{sc}}(k)]$--submodules of $\mg{g}_k^{\mr{der}}$ are $0$
  and $\mg{g}_k^{\mr{der}}$. To see the ``moreover'', we observe that
  by Steinberg's theorem the $\F_p[G^{\mr{sc}}(\F_p)]$--submodules of
  $\mg{g}^{\mr{der}}_{\F_p}$ are $0$ and $\mg{g}^{\mr{der}}_{\F_p}$,
  and $\mg{g}^{\mr{der}}_{\F_p}$ generates $\mg{g}_k^{\mr{der}}$ as an
  $\F_p[G^{\mr{sc}}(k)]$--module.

  It then follows from the first condition that
  $h^0(\Gamma_S, \rbarg) = h^0(\Gamma_S, \rbarg^*) = 0$, so the global
  deformation functor is representable.
\end{rem}

\subsection{Nice primes}

Under our assumptions on $\rbar$ we prove the existence of (a
Chebotarev set of) primes $q$ that will make up the auxiliary set $Q$
of Theorem \ref{thm:main}.

\begin{defn}\label{def:nice}
A prime $q$ of $\mc{O}_F$ is called a \emph{nice} prime for $\rbar$ if:
\begin{enumerate}
\item $\mr{Norm}(q)$ is not congruent to $\pm 1 \mod p$;
\item $\rbar(I_q) = \{1\}$ and $\rbar({\rm Frob}_q)$ is a regular semisimple 
element of $G(k)$;
\item there is a \emph{unique} root $\alpha$ of $\Phi(G,T)$, with $T$
  the identity component of the centralizer of $\rbar({\rm Frob}_q)$ in $G$
  (assumed to be a maximal split torus of $G$), such that $\Gamma_q$
  acts on $\mg{g}_{\alpha}$ (the corresponding root space) by
  $\kappa$.
\end{enumerate}
\end{defn}

The uniqueness of the root $\alpha$ plays an essential role in our
proofs.

\begin{prop} \label{prop:nice} %
Let $\rbar$ be as above and assume in addition that:
\begin{itemize}
\item $[F(\mu_p): F] = p-1$;  %and $F(\mu_p) \nsubset F(\rbarg)$
\item $p -1$ is greater than the maximum of $8 \#Z_{G^{\mr{sc}}}$ and
  \[
  \begin{cases}
    (2h-2)  \#Z_{G^{\mr{sc}}} & \mbox{ if }  \#Z_{G^{\mr{sc}}} \mbox{ is even,
      or} \\
    (4h-4)  \#Z_{G^{\mr{sc}}} & \mbox{ if }  \#Z_{G^{\mr{sc}}} \mbox{ is odd,}
    \end{cases}
\]
where $h$ is the Coxeter number of $G$ and $Z_{G^{\mr{sc}}}$ is the
(finite) centre of $G^{\mr{sc}}$.
 \end{itemize}
Then
\begin{enumerate}
\item $h^1(\gal(K/F), \rbarg) = h^1(\gal(K/F), \rbarg^*) = 0$, where
  $K = F(\rbarg, \mu_p)$;
\item  there exists a non-empty Chebotarev set of nice primes for
  $\rbar$.
\end{enumerate}
\end{prop}

\begin{proof}
  We follow the proof of \cite[Theorem 6.4]{pat-ex}.

\smallskip

  Set $L = F(\mu_p) \cap F(\rbarg)$.  It is shown after the proof of
  \cite[Lemma 6.6]{pat-ex} that $[L:F] \divides \# Z_{G^{\mr{sc}}}$, so the
  two assumptions together imply that $F(\mu_p) \nsubset
  F(\rbarg)$. Then (1) follows from \cite[Lemma 6.6]{pat-ex}.

\smallskip

  The differences in our setup as compared to that of \cite[Theorem
  6.4]{pat-ex} are the extra conditions (1) and (3) in the definition of
  nice primes, so we explain the necessary modifications.

  \begin{claim} \label{claim:x} %
    There is a regular semisimple element
    $x \in \rbar(\Gamma_L) \cap G^{\mr{der}}(k)$ contained in a split
    maximal torus $T$ of $G$ and a root $\alpha \in \Phi(G,T)$, such that
  \begin{itemize}
  \item $\beta(x) \neq \alpha(x)$ for any root $\beta \neq \alpha$,
  \item $\alpha(x)$ lies in $\gal(F(\mu_p)/L)$, where we use the
    canonical identification $\gal(F(\mu_p)/F) \cong \F_p^{\times}$,
  \item $\alpha(x) \neq \pm 1$.
  \end{itemize}
\end{claim}      

    Assuming the claim, choose $\tau \in \Gamma_L$ so that
    $\rbar(\tau) = x$. Then there exists $\sigma \in \gal(K/L)$ such
    that
    \begin{itemize}
    \item $\sigma \mapsto \tau|_{F(\rbarg)} \in \gal(F(\rbarg)/L)$,
    \item $\sigma \mapsto \alpha(x) \in \gal(F(\mu_p)/F)$.
    \end{itemize}
    The desired Chebotarev set of nice primes is the set of primes $q$
    of $\mc{O}_F$ whose Frobenius element in $\gal(K/F) $ is
    (conjugate to) $\sigma$. The image of $\sigma$ in
    $G^{\mr{ad}}(k)$ is equal to the image of $x$, so
    the image of $\sigma$ in $G(k)$ lies in $T(k)$ and is regular
    semisimple.

\smallskip

We now prove the claim:

For any generator $g \in k^{\times}$, we let $t = g^{\frac{\#
    Z_{G^{\mr{sc}}}}{2}}$ or $t = g^{\# Z_{G^{\mr{sc}}}}$ depending on whether
$\# Z_{G^{\mr{sc}}}$ is even or odd. The element $2 \rho^{\vee}(g)$ (where
$2\rho^{\vee}$ is the sum of the positive coroots) is in the image of
$G^{\mr{sc}}(k)$, hence in $\rbar(\Gamma_F)$, so by the discussion at
the beginning of the proof $2 \rho^{\vee}(t) \in \rbar(\Gamma_L)$. It
is regular, since for all positive roots $\beta$, $\beta(x) = t^{2
  \mr{ht} \beta}$, the maximum height of a root is $h - 1$, and none
of $t^2,t^4, \dots, t^{2h-2}$ equals $1$ (or even $-1$) by our
assumption on $p$. Also, for any positive root $\beta$, $\beta(x) \in
\gal(F(\mu_p)/L)$.

We then take $\alpha$ to be the highest root (usually denoted by
$\delta$), i.e., the unique positive root of height $h -1$. All the
elements $t^2,t^4, \dots, t^{2h-2}, t^{-2},t^{-4}, \dots, t^{-(2h-2)}$
of $k^{\times}$ are distinct (and not equal to $\pm 1$) by our
assumptions on $p$, so $\alpha(x) \neq \beta(x)$ for any root
$\beta \neq \alpha$, proving the claim.
\end{proof}

\begin{rem} %\marginpar{rem:gl2}
  \label{rem:gl2} The bounds for $p$ in
  Proposition \ref{prop:nice} are not the best possible and can be
  improved for particular choices of $G$.  If $G = GL_2$ and
  $p \geq 5$ then $[L:F] \leq 2$, so the condition
  $F(\mu_p) \nsubset F(\rbarg)$ (used in the proof) is satisfied
  whenever $[F(\mu_p):F] > 2$. Also, from \cite[Lemma 2.48]{DDT} one
  sees that the vanishing in (1) holds whenever $|k| > 5$ or
  $k = \F_5$ and $\rbar(\Gamma_F) = GL_2(\F_5)$.

  Let $\alpha$ be the standard positive root and set
  $d := [F(\mu_p):L]$. Since $\bar{\rho}(\Gamma_L) \supset SL_2(k)$,
  if $[L:F] = 2$ we may take $x$ to be 
  $\Bigl [ \begin{smallmatrix} b & 0 \\ 0 & b^{-1} \end{smallmatrix}
  \Bigr ]$, with $ b = a^{(p-1)/2d}$ for $a$ any generator of
  $\F_p^{\times}$.  Then $\alpha(x) = a^{(p-1)/d} \neq \pm 1$ if
  $d> 2$.  If $F = L$ we may take $x$ to be 
  $\Bigl [ \begin{smallmatrix} b & 0 \\ 0 & b^{-1} \end{smallmatrix}
  \Bigr ]$, with $ b = a^{(p-1)/d}$ for $a$ any generator of
  $\F_p^{\times}$.  Then $\alpha(x) = a^{2(p-1)/d} \neq \pm 1$ if
  $d> 4$.  In either case, we see that the bulleted conditions in
  Claim \ref{claim:x} are satisfied if $[F(\mu_p): F] > 4$ and so (2)
  also holds in this case.  \smallskip

  Now suppose $[F(\mu_p):F] = 4$ and
  $\bar{\rho}(\Gamma_F) \supset GL_2(\F_p)$.  If $F = L$ we may take
  $x$ to be
  $\Bigl [ \begin{smallmatrix} b & 0 \\ 0 & 1\end{smallmatrix} \Bigr
  ]$, with $b = a^{(p-1)/4}$ and $a$ any generator of $\F_p^{\times}$.
  Then $\alpha(x) = b \neq \pm 1$ so again (2) holds.  If $F \neq L$,
  suppose that $\mr{det}(\bar{\rho})$ is $\kappa$
  (resp.~$\kappa^{-1}$), so $p=5$ (since
  $\bar{\rho}(\Gamma_F) \supset GL_2(\F_p)$). Without appealing to
  (the proof of) Proposition \ref{prop:nice}, one sees in this case
  that we may take any prime whose Frobenius in $\gal(F(\rbar)/F)$ is
  $\Bigl [ \begin{smallmatrix} a & 0 \\ 0 & 1\end{smallmatrix} \Bigr
  ]$
  (resp.~$\Bigl [ \begin{smallmatrix} 1 & 0 \\ 0 & a\end{smallmatrix}
  \Bigr ]$).

\end{rem}

\begin{lem} \label{lem:nonsplit} %\marginpar{lem:nonsplit}%
  Let $p\geq
  5$ be a very good prime for $G$ and let $H$ be a subgroup of
  $G^{\mr{ad}}(W(k)/p^nW(k))$ whose image in $G^{\mr{ad}}(k)$ contains the image
  of $G^{\mr{sc}}(k)$. Then
\begin{enumerate}
\item $H$ contains the kernel of the reduction map $G^{\mr{ad}}(W(k)/p^nW(k))
  \to G(k)$. 
\item For $t \leq n$, let $H_t$ be the image of $H$ in
  $G^{\mr{ad}}(W(k)/p^tW(k))$.  The sequence
  \[
  0 \to K_t \to H_t \to H_{t-1} \to 1
  \]
  is nonsplit for all $t \geq 2$. Here $K_t$ is the kernel of the
  reduction map $H_t \to H_{t-1}$ which, using (1), can be seen to be
  isomorphic to $\mg{g}_k^{\mr{der}}$.
\end{enumerate}
\end{lem}

\begin{proof}
  This is a simple consequence of results from \cite{vasiu-surj}.

  \smallskip

  To prove (1), let $H'$ be the inverse image of $H$ in
  $G^{\mr{sc}}(W(k))$. Since $H$ contains the image of
  $G^{\mr{sc}}(k)$, it follows that $H'$ surjects onto
  $G^{\mr{sc}}(k)$ so by \cite[Theorem 1.3]{vasiu-surj}
  $H' = G^{\mr{sc}}(W(k))$. Since $p$ is very good, the map on Lie
  algebras induced by the homomorphism $G^{\mr{sc}} \to G^{\mr{ad}}$
  is an isomorphism, so $\ker(G^{\mr{sc}}(W(k)) \to G^{\mr{sc}}(k))$
  maps onto $\ker(G^{\mr{ad}}(W(k)) \to G^{\mr{ad}}(k))$, giving (1).

  The statement in (2) for $t=2$ is Propostion 4.4.1 of
  \cite{vasiu-surj}, where it is the main ingredient in the proof of the
  main theorem. In fact, in our setting of a split group with $p \geq
  5$ a very good prime, (2) can be reduced to the case of $SL_2$ (see
  the first paragraph of loc.~cit.). In this case it follows from the
  fact that the matrix $A = \bigl [ \begin{smallmatrix} 1 & p^{t-2} \\ 0 &
    1 \end{smallmatrix} \bigr ]$, viewed as an element of
  $SL_2(W(k)/p^{t-1}W(k))$, is of order $p$ but there is no matrix of
  order $p$ in $SL_2(W(k)/p^{t}W(k))$ reducing to $A$.
\end{proof}

The following analogue of Fact 5 from \cite{klr} plays a crucial role in
allowing us to find nice primes with various desired properties.
\begin{lem}  \label{lem:kill} %
  Let $\rho_n$ be a deformation of $\rbar$ to
  $W(k)/p^nW(k)$. % unramified outside $S$
  Let $\{f_1,f_2,\dots,f_r\}$ be $k$-linearly independent in
  $H^1(\Gamma_F, \rbarg)$ and $\{\phi_1, \phi_2,\dots,\phi_s\}$ be
  $k$-linearly independent in $H^1(\Gamma_F, \rbarg^*)$. Let
  $K = F(\rbarg, \mu_p)$ as before, and let $K_{f_i}$
  (resp.~$K_{\phi_j}$) be the fixed field of the kernel of the
  restriction of $f_i$ (resp.~$\phi_j$) to $\Gamma_K$. Then as an
  $\F_p[\gal(K/F)]$-module $\gal(K_{f_i}/K)$
  (resp.~$\gal(K_{\phi_j}/K)$) is isomorphic to $\rbarg$
  (resp.~$\rbarg^*$). Let $K_{P_n}$ be the fixed field of the kernel
  of the composite of $\rho_n$ with the adjoint representation
  $G(W(k)/p^nW(k)) \to GL(\mg{g}^{\mr{der}}_{W(k)/p^nW(k)})$. Then
  each of the fields $K_{f_i}$, $i=1,\dots,r$, $K_{\phi_j}$,
  $j =1,\dots,s$, $K_{P_n}$, $K(\mu_{p^n})$ is linearly disjoint over
  $K$ with the compositum of the others.

  Let $L = F(\mu_p) \cap F(\rbarg)$, and let $x \in \rbar(\Gamma_L)
  \cap G^{\mr{der}}(k)$, $T$ and $\alpha$ be as in Claim
  \ref{claim:x}. For $i=1,2,\dots,r$, let $a_i \in \mg{g}_k^{\mr{der}}$
  be any element fixed by $x$ and for $j =1,2,\dots,s$, let $b_j \in
  (\mg{g}_k^{\mr{der}})^{\vee}$ be any element on which $x$ acts by
  $\alpha^{-1}(x)$. Furthermore, let $y$ be any element of the image
  of $\rho_n$ reducing to $x$. There exists a Chebotarev set of nice
  primes $q$ such that
  \begin{enumerate}
  \item $\rbar({\rm Frob}_q) = x$,
  \item for $i=1,2,\dots,r$, $f_i|_{\Gamma_q}$ (which is an element of
    $H^1_{nr}(\Gamma_q, \rbarg)$ so can be viewed as a homomorphism
    $\Gamma_q/I_q \to (\mg{g}_k^{\mr{der}})^{{\rm Frob}_q}$) is given by ${\rm Frob}_q
    \mapsto a_i$,
  \item for $j=1,2,\dots,s$, $\phi_j|_{\Gamma_q}$ (which is an element
    of $H^1_{nr}(\Gamma_q, \rbarg^*)$ so can be viewed as a
    homomorphism $\Gamma_q/I_q \to ((\mg{g}_k^{\mr{der}})^*)^{{\rm Frob}_q}$) is
    given by ${\rm Frob}_q \mapsto b_j$,
  \item $\rho_n({\rm Frob}_q) = y$.
  \end{enumerate}

\end{lem}

\begin{proof}
  Our assumptions on $\rbar$ and $p$ and Remark \ref{rem:abs} imply
  that $\rbarg$ and $\rbarg^*$ are non-isomorphic irreducible
  $\F_p[\Gamma_F]$-modules (with endomorphism ring $k$), so the
  statements about the images of the cocycles and linear disjointness
  of $K_{f_i}$ and $K_{\phi_j}$ are immediate consequences of
  Proposition \ref{prop:nice}, (1).

  We have to do a little more work in order to include $K_{P_n}$. For
  this we consider the exact sequences
  \[
  1 \to \gal(K_{f_i}/K) \to \gal(K_{f_i}/F) \to \gal(K/F) \to 1
  \]
  and
  \[
  1 \to \gal(K_{\phi_j}/K) \to \gal(K_{\phi_j}/F) \to \gal(K/F) \to 1\ .
  \]
  The first short exact sequence, viewed as an element of
  $H^2(\gal(K/F), \rbarg)$, is obtained by first restricting $f_i$ to
  $\Gamma_K$ and then applying the map $H^1(\Gamma_K,
  \rbarg)^{\gal(K/F)} \to H^2(\gal(K/F), \rbarg)$ from the
  inflation-restriction sequence, the exactness of which then implying
  that this element is zero.  Thus, the sequence splits and the same
  argument implies that the second sequence also splits.

  For any $t \leq n$, let $\rho_t$ be the reduction of $\rho_n$ modulo
  $p^t$.  By Lemma \ref{lem:nonsplit} (1) the image of
  $\tilde{\rho}_t$, the composite of $\rho_t$ with the adjoint
  representation as above, contains the image of $G^{\mr{sc}}(W(k)/p^nW(k))$.
  Furthermore, it follows from Lemma \ref{lem:nonsplit} (2) that the
  induced extensions
  \[
  1 \to \mg{g}_k^{\mr{der}} \to \im(\tilde{\rho}_t) \to
  \im(\tilde{\rho}_{t-1}) \to 1 \] are non-split for $t \geq 2$.

Given the above, we get (strong) linear disjointness of all the
$K_{f_i}$, $K_{\phi_j}$ and $K_{P_n}$: the fields $K_{f_i}$ for
varying $i$ are strongly linearly disjoint because the $f_i$ are
$k$-linearly independent and similarly for the $K_{\phi_j}$ and The
fact that $K(\mu_{p^n})$ is linearly disjoint from the compositum of
all these fields follows from the fact that $\gal(K/F)$ acts trivially
on $\gal(K(\mu_{p^n})/K)$ while the Galois groups of the other
extensions have no such non-trivial quotient by Remark \ref{rem:abs}.

  To construct the desired Chebotarev set of nice primes we proceed as
  follows:

  We apply (the proof) of Proposition \ref{prop:nice} to get a
  Chebotarev set of nice primes $q$ for $\rbarg$ in $\mc{O}_F$
  determined by the choice of a suitable Frobenius element $\sigma \in
  \gal(K/F)$ (with $K = F(\rbarg, \mu_p)$). 

  The conditions in (2), (3) and (4) are all Chebotarev conditions on
  the field extensions $K_{f_i}/K$, $K_{\phi_j}/K$ and $K_{P_n}/K$ and
  then the linear disjointness proved above implies that all these
  conditions can be realised simultaneously by choosing a suitable
  element of $\gal(\wt{K}/F)$, with $\wt{K}$ the compositum of all the
  $K_{f_i}$, $K_{\phi_j}$ and $K_{P_n}$.
 % Since a semisimple element of
 %  $G(k)$ has order prime to $p$, restriction from $\Gamma_q$ to
 %  $\Gamma_{q'}$, where $q'$ is a prime of $\mc{O}_K$ lying over $q$,
 %  does not affect the (non)triviality of a cohomology class with
 %  coefficients in an $k$-vector space. It then suffices to choose
 %  $q$ so that its Frobenius has the desired behaviour in
 %  $\gal(K_{f_i}/F)$.%\footnote{Should be improved!}
\end{proof}

% \marginpar{sec:loc}%
\subsection{Local computations I} \label{sec:loc}

\subsubsection{}
Let $\Gamma$ be the absolute Galois group of any local field with
residue characteristic prime to $p$.  We recall that for any
$W(k)[\Gamma]$-module $M$ which is finite as a $W(k)$-module Tate's
Euler characteristic formula states (\cite[Theorem 2.8]{milne-adt})
:%\marginpar{eq:euler}
\begin{equation} \label{eq:euler} 
  \ell(H^0(\Gamma, M)) - \ell(H^1(\Gamma, M))
  + \ell(H^2(\Gamma,M)) = 0.
\end{equation}
Furthermore, if $M$ is a free $W(k)/p^n$-module for some $n$, then
local Tate duality gives a perfect pairing (\cite[Corollary
2.3]{milne-adt}): %\marginpar{eq:ltd}%
\begin{equation} \label{eq:ltd}
H^i(\Gamma, M) \times H^{2-i}(\Gamma, M^*) \sr{\cup}{\to} H^2(\Gamma,
(W(k)/p^n)(\kappa)) \sr{inv_q}{\cong} W(k)/p^n.
\end{equation}

% Let $q$ be a nice prime for $\rbar$ with associated root
% $\alpha$. Then $Frob_q$ acts trivially on $\mg{l}_{\alpha}$ and
% $\mg{g}_{\alpha}^*$.

\subsubsection{The local condition at nice primes} \label{sec:lc} %\marginpar{sec:lc}%

For a root $\alpha \in \Phi(G,T)$ we let $T_{\alpha}$ be the identity
component of $\ker(\alpha)$, $\mg{t}_{\alpha}$ its Lie algebra and we
let $\mg{l}_{\alpha}$ be the Lie algebra of the image of the coroot
associated to $\alpha$, so
$\mg{t} = \mg{t}_{\alpha} \oplus \mg{l}_{\alpha}$.  We denote by
$H_{\alpha}$ the subgroup of $G$ generated by $T$ and $U_{\alpha}$
(the corresponding root subgroup of $G$), so $\alpha$ extends to a
homomorphism $H_{\alpha} \to \G_m$.

Associated to a nice prime $q$ there is a smooth deformation condition
$\mc{D}_q$ introduced in \cite{pat-ex}, generalizing the one
introduced for $GL_2$ in \cite{Ravi}. Lifts $\rho: \Gamma_q \to G(A)$
(of $\rbar|_{\Gamma_q} \to G(k)$) to Artin local $W(k)$-algebras $A$
with residue field $k$ are in $\mc{D}_q$ if they are $\wh{G}(A)$
conjugate to lifts factoring through $H_{\alpha}(A)$ and such that
$\alpha \circ \rho: \Gamma_q \to A^{\times}$ is $\kappa$; this
condition forces $\rho(I_q)$ to be contained in $U_{\alpha}(A)$. Its
tangent space is
$\mc{N}_q = H^1(\Gamma_q, W) \subset H^1(\Gamma_q, \rbarg)$, where
$W = \mg{t}_{\alpha} \oplus \mg{g}_{\alpha}$.
%, and this hasdimension $h^0(\Gamma_q, \rbarg)$.
We refer to \cite[\S 4.2]{pat-ex}
for the proofs.

\subsubsection{}
From now on we will assume that all nice primes $q$ we consider
correspond to some fixed element of $\rbar(\Gamma_S) \cap T(k)$ (for a
fixed split maximal torus $T$) and a fixed root
$\alpha \in \Phi(G,T)$.  That such primes suffice for our needs is due
to Remark \ref{rem:abs}.  To simplify notation in what follows, we fix
isomorphisms of $\mg{l}_{\alpha}$ and $\mg{g}_{\alpha}^*$ with $k$;
for $\mg{l}_{\alpha}$ we use the natural one corresponding to the map
on Lie algebras induced by $\alpha^{\vee}$ but for $\mg{g}_{\alpha}^*$
there is no canonical
choice.
For nice primes $q$ as above these isomorphisms are also isomorphisms
as $\Gamma_q$-modules (with the trivial action on $k$).  Using these
identifications, for $f \in H^1(\Gamma_q, \rbarg)$ we will denote by
$f({\rm Frob}_q)$ the element of $k$ obtained by evaluating the
$\mg{l}_{\alpha}$ component of $f$ (which is a homomorphism
$\Gamma_q \to k$) at ${\rm Frob}_q$.  Similarly, for
$\phi \in H^1(\Gamma_q, \rbarg^*)$ we will denote by
$\phi({\rm Frob}_q)$ the element of $k$ corresponding to the
$\mg{g}_{\alpha}^*$ component of $\phi$. (These elements are well
defined since $\mr{Norm}(q) \notequiv 1 \mod p$ implies that
$H^1(\Gamma_q, k) = H^1_{nr}(\Gamma_q, k)$.)

\smallskip

The lemma below is the reason why we work with nice primes rather than
the more general class of primes used in \cite[\S 4.2]{pat-ex}:
without the uniqueness of the root $\alpha$ none of the statements
would be true and we would not have enough control to make our global
arguments work.
\begin{lem} \label{lem:nice} %\marginpar{lem:nice}
  For $q$  a nice
  prime for $\bar{\rho}$ we have:
\begin{enumerate}
\item $h^1(\Gamma_q, \rbarg) - h^0(\Gamma_q, \rbarg) = 1$.
\item $h^1(\Gamma_q, \rbarg) - \dim{\mc{N}_q} = 1$.
\item Fix $f \in H^1(\Gamma_q, \rbarg)$ ramified at $q$. Then for all
  unramified $\phi \in H^1_{nr}(\Gamma_q, \rbarg^*)$,
  $inv_q(f \cup \phi) = \gamma_f\phi({\rm Frob}_q)$, where $\gamma_f$ is a
  nonzero constant depending only on $f$.
\item Fix $\phi \in H^1(\Gamma_q, \rbarg^*)$ ramified at $q$. After
  appropriately scaling $\phi$, for any
  $f \in H^1_{nr}(\Gamma_q, \rbarg)$ we have
  $inv_q(f \cup \phi) = f({\rm Frob}_q)$.
\end{enumerate}
\end{lem}

\begin{proof}
  Let $\chi: \Gamma_q \to k^{\times}$ be any character and denote the
  associated $\Gamma_q$--module by $k(\chi)$. We then have that
  $H^i(\Gamma_q, k(\chi)) = 0$ for all $i$ if $\chi$ is not trivial
  or the cyclotomic character: this well-known fact follows from
  \eqref{eq:ltd} and \eqref{eq:euler} since the computation of $h^0$
  is trivial, $h^2$ is computed using duality, and then $h^1$ can be
  determined using the vanishing of the Euler characteristic. For the
  trivial character we have $h^0(\Gamma_q, k) = h^1(\Gamma_q, k) = 1$
  and $h^2(\Gamma_q, k) = 0$ and for the cyclotomic character $\kappa$ we
  have $h^0(\Gamma_q, k(\kappa)) = 0$ and
  $h^1(\Gamma_q, k(\kappa)) = h^2(\Gamma_q, k(\kappa)) = 1$.

  From the definition of niceness, it follows that $\rbarg^*$ is a sum
  of characters of $\Gamma_q$, precisely one of which is trivial, so
  $h^0(\Gamma_q, \rbarg^*) = h^2(\Gamma_q, \rbarg) = 1$ which gives
  (1) using \eqref{eq:euler}. The claim in (2) follows from (1) since
  (\cite[Lemma 4.11 (2)]{pat-ex})
\[
\dim(\mc{N}_q) = h^1_{nr}(\Gamma_q, \rbarg) = h^0(\Gamma_q, \rbarg).
\]

To prove (3) and (4) we use the fact that $H^1_{nr}(\Gamma_q, \rbarg)$
pairs trivially with $H^1_{nr}(\Gamma_q, \rbarg^*)$. We also have
$H^1(\Gamma_q,\mg{t}) = H^1_{nr}(\Gamma_q, \mg{t})$ since the action
is trivial and $H^1(\Gamma_q, \mg{g}_{\beta}) = H^1_{nr}(\Gamma_q,
\mg{g}_{\beta})$ for any root $\beta \neq \alpha$ (and this is
nonzero iff the action on $\mg{g}_{\beta}$ is trivial).

Thus, $inv_q(f \cup \phi)$ in 3) is determined by the
$\mg{g}_{\alpha}$ component of $f$ and this pairs non-degenerately
with the $\mg{g}_{\alpha}^*$ component of $\phi$, so (3) follows. The
proof of (4) is dual to this, so we skip the details.
\end{proof}
\begin{rem}
While the proofs of parts (3) and (4) of Lemma~\ref{lem:nice} are dual to one another, we will apply them as asymmetrically stated, especially in Case 3) of the proof of
Theorem~\ref{thm:main}.
\end{rem}

\begin{defn} \label{def:special} %
  We will say that a lift of $\rbar$ to $G(W(k)/p^nW(k))$ is
  \emph{special}\footnote{For groups $G$ of rank $> 1$, our usage of
    this term conflicts with more established terminology, but for
    lack of a pithy alternative we have still used it. We apologize to
    the reader for this abuse.} at a nice prime $q$ if its restriction
  to $\Gamma_q$ lies in $\mc{D}_q$.
\end{defn}

For $\alpha$ a root associated to a nice prime as constructed in
Proposition \ref{prop:nice}, let $G_{\alpha}$ be the subgroup of $G$
generated by $U_{\alpha}$ and $U_{-\alpha}$ (which is isomorphic to
$SL_2$ or $PGL_2$).  Let $S_{\alpha}$ be the subgroup of $G$ generated
by $G_{\alpha}$ and $T$; it is a (split) reductive group of semisimple
rank one, so is isomorphic to $GL_2 \times T'$, $SL_2 \times T'$ or
$PGL_2 \times T'$, for some split torus $T'$. Henceforth, we fix such
a decomposition with the property that $\alpha$ corresponds to the
standard positive root of $GL_2$, $SL_2$ or $PGL_2$ and $T = T'D$ with
$D$ the (image of) the diagonal matrices in the first factor.

\begin{lem} \label{lem:torlift} %
  Let $q$ be a nice prime for $\rbar$ and let
  $\rho_n: \Gamma_F \to G(W(k)/p^nW(k))$, $n \geq2$, be a lift of
  $\rbar$.  % and $q$ is a nice prime constructed as in Proposition
  % \ref{prop:nice} with Frobenius element $Frob_q$ mapping to an
  % element $t \in T(k)$ 
  \begin{enumerate}
    \item
    There is a Chebotarev class of nice primes $q'$ such that
    $\rbar({\rm Frob}_q) = \rbar({\rm Frob}_{q'})$,
    $\rho_n({\rm Frob}_{q'}) \in T(W(k)/p^nW(k))$ and $\rho_n$ is special at
    $q'$ (with respect to the root $\alpha$).
  \item There is a Chebotarev class of nice primes $q''$ with
    $\mr{Norm}(q'') \equiv \mr{Norm}(q') \mod p^n$ and such that
    $\rho_n({\rm Frob}_{q''}) \in T(W(k)/p^nW(k))$ is equal to
    $\rho_n({\rm Frob}_{q'}) \cdot (A_n, Id_{T'})$, with $q'$ as above and
    $A_n$ the matrix
    $\bigl [ \begin{smallmatrix} 1 - p^{n-1} & 0 \\ 0 & 1 +
      p^{n-1} \end{smallmatrix} \bigr ]$. Here $A_n$ is viewed as an
    element of $SL_2$ or $GL_2$ and in the $PGL_2$ case we replace it by
    its image in $PGL_2$. Thus, $\rho_n$ is not special at $q''$ but
    $\rho_{n-1}$ is special at $q''$.
    \end{enumerate}
  \end{lem}

\begin{proof}
  An element $g \in G(W(k)/p^nW(k))$ lies in $T(W(k)/p^nW(k))$ iff its
  image in $ G^{\mr{ad}}(W(k)/p^nW(k))$ lies in
  $T^{\mr{ad}}(W(k)/p^nW(k))$, where $T^{\mr{ad}}$ is the image of $T$
  in $G^{\mr{ad}}$. From this observation and Lemma \ref{lem:nonsplit}
  (1), (1) of this lemma follows by considering the extension
  $K_{P_n}K(\mu_{p^n})/F$ and choosing an appropriate element in its
  Galois group. Item (2) follows in the same way once we note that the
  image of $SL_2(W(k)/p^nW(k))$ in $G(W(k)/p^nW(k))$ is contained in
  $\rho_n(\Gamma_F)$ by Lemma \ref{lem:nonsplit} (1).
\end{proof}

\subsection{Main theorem} \label{sec:main} %\marginpar{sec:main}%

Let $X$ be a finite set of primes of $\mc{O}_F$.
Let $M$ be a finite $W(k)[\Gamma_X]$-module and for each $v \in X$ let
$\mc{L}_v$ be a submodule of $H^1(\Gamma_v, M)$. We call
$\mc{L} := \{\mc{L}_v\}_{v \in X}$ a Selmer condition.
\begin{defn} \label{def:selmer}
  The \emph{Selmer group} $H^1_{\mc{L}}(\Gamma_X, M)$ is defined to be
  the kernel of the (global to local) restriction map
\[
  H^1(\Gamma_X, M) \to \bigoplus_{v \in X}
  \frac{H^1(\Gamma_v,M)}{{\mc{L}_v}} .
\]
If $M$ as above is a free $W(k)/p^n$-module, let
$\mc{L}_v^{\perp} \subset H^1(\Gamma_v, M^*)$ be the annihilator of
$\mc{L}_v$ with respect to the pairing \ref{eq:ltd}. Then the
\emph{dual Selmer group} is defined to be
$H^1_{\mc{L}^{\perp}}(\Gamma_X, M^*)$.
\end{defn}

From now on we make the following additional assumption on $\rbar$.

\begin{ass} \label{ass:2} %\marginpar{ass:2}%
  For each prime $v$ in $S$ we are given a smooth local deformation
  condition $\mc{D}_v$ for $\rbar|_{\Gamma_{v}}$ with tangent space
  $\mc{N}_v \subset H^1(\Gamma_v, \rbarg)$ such that the set of Selmer
  conditions $\mc{N}$ is \emph{balanced}, i.e.,
  \[
  \dim H^1_{\mc{N}}(\Gamma_S, \rbarg) = \dim
  H^1_{\mc{N}^{\perp}}(\Gamma_S, \rbarg^*).
\]
\end{ass}

\begin{rem}
  It is not known exactly when this assumption is satisfied,
  especially at primes dividing $p$. For $ v \nmid p$ it is known to
  hold for classical groups, albeit after increasing $k$, and for
  $v \mid p$ in the Fontaine--Laffaille case (see
  \cite{CHT},\cite{booher-minimal}, \cite{booher-geometric}).
\end{rem}

\begin{defn} \label{def:aux} %\marginpar{def:aux}%
  An \emph{auxiliary set} for $\rbar$ is a finite set $Q$ of nice
  primes (with $Q \cap S = \emptyset$) such that the new-at-$Q$ Selmer
  group $H^1_{\mc{N}}(\Gamma_{S \cup Q}, \rbarg) = 0$.
\end{defn}
Here $\mc{N}$ is defined using the given local conditions at primes in
$S$ and the condition described above at primes $q$ in $Q$.  In this
situation the universal deformation ring $R_{S \cup Q}^{Q-new}$, defined
using the given local deformation conditions at $v \in S$ and the
above conditions at $q \in Q$, is isomorphic to $W(k)$.

\begin{thm} \label{thm:main} %\marginpar{thm:main}%
  Let $\rbar: \Gamma_S \to G(k)$ satisfy Assumptions \ref{ass:1} and
  \ref{ass:2} and let $p=\ch(k)$ be such that the hypotheses of
  Proposition \ref{prop:nice} hold.  Assume that the Selmer and dual
  Selmer groups of $\rbar$ with respect to the given local conditions
  at primes in $S$ are nontrivial of (the same) dimension $n$. Then
  there exists an ordered  set $Q = \{q_1,q_2,\dots,q_m\}$ of nice primes,
  $m \in \{n+1,n+3\}$, and an integer $d \geq 2$ such that
\begin{itemize}
\item $Q$ is auxiliary and the versal deformation $\rho_{S \cup
    Q}^{Q-new}: \Gamma_{S \cup Q} \to G(R_{S \cup Q}^{Q-new}) = G(W(k))$ is
  ramified mod $p^d$ at all $q \in Q$.
\item for each $1 \leq i \leq m$ there is an auxiliary set $Q_i
  \subset Q$ satisfying
  \begin{itemize}
  \item $\{1,2,\dots,q_{i-1}\} \subset Q_i$,
  \item $q_i \notin Q_i$,
  \item  $\rho_{S \cup Q_i}^{Q_i-new} \equiv  \rho_{S \cup Q}^{Q-new}$
    mod $p^{d-1}$, $\rho_{S \cup Q_i}^{Q_i-new}$ mod $p^{d-1}$  is
    special at $q_i$ but $\rho_{S \cup Q_i}^{Q_i-new}$ mod $p^d$ is
    not special at $q_i$. 
  \end{itemize}
\end{itemize}
\end{thm}

\begin{rem}
  Implicit in the formulation of the theorem is that
  $\rho_{S \cup Q}^{Q-new}: \Gamma_{S \cup Q} \to G(R_{S \cup
    Q}^{Q-new}) = G(W(k))$ is ramified at all $q \in Q$ mod $p^d$ and
  is unramified at all $q \in Q$ mod $p^{d-1}$. The proof will show
  that we can in fact take $d=2$.
  
  Note that we are not claiming that $\rho_{S \cup Q_i}^{Q_i-new}$ is
  ramified at all $v \in S \cup Q_i$, although it would be nice to get
  this refinement of the theorem.
\end{rem}

\begin{ex} \label{ex:gl2} %\marginpar{ex:gl2}%
  Using Remark
  \ref{rem:gl2} one sees that if $G = GL_2$, sufficient conditions for
  the main theorem to hold are that
  $\bar{\rho}(\Gamma_F) \supset SL_2(k)$ and
\begin{itemize}
  \item $[F(\mu_p):F] > 4$ (so we must have $p> 5$), or 
  \item$[F(\mu_p):F] = 4$, $\rbar(\Gamma_F) \supset GL_2(\F_p)$, and
    \begin{itemize}
    \item $F(\rbarg) \cap F(\mu_p) = F$, or
    \item $\det(\rbar) = \kappa^{\pm 1}$.
    \end{itemize}
  \end{itemize}
\end{ex}
As mentioned in the introduction, the theorem will be proved below
using only Galois cohomology (and Chebotarev's theorem).

\subsection{Global to local restriction maps}

\subsubsection{}

Controlling the dimensions of Selmer and dual Selmer groups is fundamental to our
arguments and the main tool for doing this is the Greenberg--Wiles
formula (\cite[Theorem 2.19]{DDT}): %\marginpar{eq:gw}%
\begin{equation} \label{eq:gw}
  \ell(H^1_{\mc{L}}(\Gamma_X, M)) -
  \ell(H^1_{\mc{L}^{\perp}}(\Gamma_X, M^*)) = \ell(H^0(\Gamma_X, M)) -
  \ell(H^0(\Gamma, M^*)) + \sum_{v \in X(\infty)}(\ell(L_v) -
  \ell(H^0(\Gamma_v,M))),
\end{equation}
where $X(\infty)$ denotes the union of $X$ (which contains all primes dividing $|M|$) and the infinite primes of
$F$ and $\mc{L}_v$ for an infinite prime is taken to be $\{0\}$.

The following result from global duality theory, which also plays an
important role in our proofs, is the ``exactness in the middle'' of
the Poitou--Tate exact sequence (\cite[Theorem 4.10 (b)]{milne-adt}).
\begin{thm} \label{thm:gd} %\marginpar{lem:gd}
  Let $X$ be a finite set of primes of $F$ containing all primes above
  $p$ and $\infty$, and let $M$ be a finitely generated and free
  $W(k)/p^n$-module with a $W(k)$-linear action of
  $\Gamma_X$. Consider the restriction maps
  \[
    Res_X: H^1(\Gamma_X, M) \to \bigoplus_{v \in X} H^1(\Gamma_v,M)
  \]
  and
  \[
    Res_X^*:H^1(\Gamma_X, M^*) \to \bigoplus_{v \in X} H^1(\Gamma_v,M^*)%\bigoplus_{v \in X} H^1(\Gamma_v,M).
  \]
  The sum of the local duality pairings of \eqref{eq:ltd} induces a perfect
  pairing
  \[
    \bigoplus_{v \in X} H^1(\Gamma_v,M) \times \bigoplus_{v \in X}
    H^1(\Gamma_v,M) \to W(k)/p^n
  \]
  with respect to which $\im(Res_X)$ is the exact annihilator of
  $\im(Res_X^*)$.
\end{thm}

\subsubsection{}
Let $X$ be a finite set of primes of $\mc{O}_F$ containing $S$, and
let $(h_v)_{v \in X}$ be a collection of elements of
$H^1(\Gamma_v, \rbarg)$. We call $(h_v)_{v \in X}$ a \emph{local
  condition problem} and we are interested in knowing whether there
exists a global class $h \in H^1(\Gamma_X, \rbarg)$ whose restriction
at $v \in X$ is $h_v$.

In general, such a class need not exist so let us suppose that this is
the case, i.e., $(h_v)_{v \in X}$ is not in the image of the
restriction map
\[
Res_X: H^1(\Gamma_X, \rbarg) \to \bigoplus_{v \in X} H^1(\Gamma_v, 
 \rbarg) \ .
\]
We will show that there exists a Chebotarev set $\mc{Q}$ of nice
primes $q$ such that $(h_v)_{v \in X}$ is in the image of the map
\[
Res_{X \cup \{q\}}^X: H^1(\Gamma_{X \cup \{q\}}, \rbarg) \to
\bigoplus_{v \in X} H^1(\Gamma_v, \rbarg) \ .
\]

\begin{lem} \label{lem:perp} %\marginpar{lem:perp}
Let $(h_v)_{v \in X}$ be a local condition problem such that 
the line $l$ spanned by $(h_v)_{v \in X}$ is not in the image of
\[
Res_X: H^1(\Gamma_X,  \rbarg) \to \bigoplus_{v \in X}
H^1(\Gamma_v, \rbarg) \ .
\]
Then the annihilator of $l$ in $\bigoplus_{v \in
  X}H^1(\Gamma_v, \rbarg^*)$ does not contain the image of the map
\[
Res_X^*: H^1(\Gamma_X, \rbarg^*) \to \bigoplus_{v \in X}
H^1(\Gamma_v, \rbarg^*) \ .
\]
\end{lem}

\begin{proof} 
  By Theorem \ref{thm:gd} the images of $Res_X$ and $Res_X^*$ are exact
  annihilators of one another, so
\[
\ann(l) \supset \im(Res^*_X) \Longleftrightarrow
l  \subset \im(Res_X) \ .
\]
The latter condition is false by hypothesis, so the former is false as
well.
\end{proof}

\begin{prop} \label{prop:hv} %\marginpar{prop:hv}%
  Let $(h_v)_{v \in X}$ be a local condition problem such that the
  line $l$ spanned by it is not in the image of
\[
Res_X : H^1(\Gamma_X,  \rbarg) \to \bigoplus_{v \in X}
H^1(\Gamma_v, \rbarg).
\]
Then there is a basis $\{\zeta_1,\dots,\zeta_s,\zeta\}$ of
$H^1(\Gamma_X, \rbarg^*)$ such that $\{\zeta_1,\dots,\zeta_s\}$
all annihilate $l$. Let $\mc{Q}$ be the Chebotarev set of nice primes
satisfying 
\begin{itemize}
\item $\zeta_i|_{\Gamma_q} = 0$ for $i=1,2,\dots,s$ and
\item the $\mg{g}_{\alpha}^*$ component of $\zeta |_{\Gamma_q}$ is
  nonzero.
\end{itemize}
Then for any $q \in \mc{Q}$, the image of
\[
Res_{X \cup \{q\}}^X: H^1(\Gamma_{X \cup \{q\}}, \rbarg) \to
\bigoplus_{v \in X} H^1(\Gamma_v, \rbarg^*) \
\]
contains $l$.
\end{prop}

\begin{proof}
  By Lemma \ref{lem:perp} we may choose a basis
  $\{\zeta_1,\dots,\zeta_s,\zeta\}$ of $H^1(\Gamma_X, \rbarg^*)$ as
  required and we also have $Res_X^*(\zeta) \notin \ann(l)$. We may
  also assume that the $\zeta_i$ are ordered so that they first
  include a basis of $\Sh^1_X(\rbarg^*)$ ($ =\ker(Res_X^*)$).  The
  assumptions on $\rbar$ and the linear disjointness from Lemma
  \ref{lem:kill} imply that $\mc{Q}$ as in the statement is a set of
  primes of positive density.

  For $v \in X$, let $\mc{L}_v = H^1(\Gamma_v, \rbarg)$ and also let
  $\mc{L}_q = H^1(\Gamma_q, \rbarg)$.  Let
  $\mc{L} = (\mc{L}_v)_{v \in X}$ and let
  $\mc{L}' = (\mc{L}_v)_{v \in X \cup \{q\}}$.  Comparing the
  Greenberg--Wiles formula \eqref{eq:gw} applied to
  $\mc{L}$ and $\mc{L}'$, we see, %
  using the easy fact of local Galois cohomology  that
  $h^1_{nr}(\Gamma_q,M)=h^0(\Gamma_q,M)$, 
  that
  \begin{equation} \label{eq1}%
    \begin{split}
      & \left(h^1_{\mc{L}}(\Gamma_X, \rbarg) - h^1_{\mc{L}'}(\Gamma_{X \cup
        \{q\}}, \rbarg\right) - \left(\ h^1_{\mc{L}^{\perp}}(\Gamma_X, \rbarg^*)
      - h^1_{\mc{L}'^{\perp}}(\Gamma_{X \cup
        \{q\}}, \rbarg^*)\right)  \\
      = &\ h^1_{nr}(\Gamma_q, \rbarg) - h^1(\Gamma_q, \rbarg) \ .
    \end{split}
  \end{equation}
    By definition, $H^1_{\mc{L}}(\Gamma_X, \rbarg)$ is the
    full $H^1$ and $H^1_{\mc{L}^{\perp}}(\Gamma_X, \rbarg^*)$ is the
    $\Sh^1$. By assumption, we have $\zeta_i|_{\Gamma_q} = 0$ for all
    $i$, so 
    $H^1_{\mc{L}^{\perp}}(\Gamma_X, \rbarg^*) =
    H^1_{\mc{L}'^{\perp}}(\Gamma_{X \cup \{q\}},  \rbarg^*)$. Thus,
    \begin{equation} \label{eq2} h^1(\Gamma_{X \cup \{q\}}, \rbarg) =
      h^1(\Gamma_X, \rbarg) + h^1(\Gamma_q, \rbarg) -
      h^1_{nr}(\Gamma_q, \rbarg)  =  h^1(\Gamma_X, \rbarg) +1.
    \end{equation}
    
    We have shown the dimensions of the domains of
        \[
    Res_X: H^1(\Gamma_X, \rbarg) \to \bigoplus_{v \in X}
    H^1(\Gamma_v, \rbarg) 
    \]
    and 
    \[
    Res_{X \cup \{q\}}^X:  H^1(\Gamma_{X \cup \{q\}}, \rbarg) \to  \left(\bigoplus_{v \in X}
    H^1(\Gamma_v, \rbarg)\right) \oplus \frac{H^1(\Gamma_q, \rbarg)}{ H^1(\Gamma_q, \rbarg)}
    \]
    differ by $1$.
    
 We now show  their kernels have the same dimension so the image of $ Res_{X \cup \{q\}}^X$ properly contains that of $Res_X$. 
    For this we let $\mc{L}_v = 0$ for $v \in X$,
    $\mc{L}_q = H^1(\Gamma_q, \rbarg)$ and
    $\mc{L} = (\mc{L}_v)_{v \in X}$,
    $\mc{L}' = (\mc{L}_v)_{v \in X \cup \{q\}}$ and apply the
    Greenberg--Wiles formula again as in \eqref{eq1}.

    Now $H^1_{\mc{L}^{\perp}}(\Gamma_X, \rbarg^*) = H^1(\Gamma_X,
    \rbarg^*)$ which has basis $\{\zeta_1,\dots,\zeta_s,\zeta\}$. As
    $\mc{L}_q^{\perp} = 0$, $\zeta_i|_{\Gamma_q} = 0$ and
    $\zeta|_{\Gamma_q} \neq 0$, we have that
    $H^1_{\mc{L}'^{\perp}}(\Gamma_{X \cup \{q\}}, \rbarg^*))$ is the
    span of $\{\zeta_1, \dots, \zeta_s\}$. Combining \eqref{eq1} and
    \eqref{eq2} we get %
    \begin{equation} \label{eq3} (h^1(\Gamma_{X \cup \{q\}}, \rbarg) -
      \dim \ker(Res_{X \cup \{q\}}^X)) - ((h^1(\Gamma_X, \rbarg) -
      \dim \ker(Res_X)) = 1
\end{equation}
Thus,   the kernels of $Res_{X \cup \{q\}}^X$ and
$Res_X$ are identical so the rank of $Res_{X \cup \{q\}}^X$  is one greater than that of $Res_X$. 

If $l \nsubseteq \im(Res_{X \cup \{q\}}^X)$, then $\im(Res_X)$ is of
codimension at least two in $l + \im(Res_{X \cup \{q\}}^X)$, so
$\ann(l + \im(Res_{X \cup \{q\}}^X))$ is of codimension at least two
in $\ann(\im(Res_X))$. However, by Theorem \ref{thm:gd} the latter
space is precisely the image of $Res_X^*$, so spanned by the (images
of) $\zeta_i$ and $\zeta$. By construction, all the $\zeta_i$
annihilate $l$ and since $\zeta_i|_q = 0$ for all $i$, it follows by
reciprocity that all the $\zeta_i$ also annihilate
$\im(Res_{X \cup \{q\}}^X)$. We conclude that
$l \subseteq \im(Res_{X \cup \{q\}}^X)$.
\end{proof}

% \begin{rem}
%   In the notation of the proposition, suppose
%   $h \in H^1(\Gamma_{X \cup \{q\}}, \rbarg)$ restricts to $h_v$
%   for each $v \in X$. Since
%   $\sum_v \langle \zeta, h_v \rangle_v \neq 0$, it follows by
%   reciprocity that $\langle \zeta, h \rangle_q \neq 0$ as well.
% \end{rem}

\subsection{Infinitesimal adjustment of lifts}

The process of ``adjusting'' a representation
$\rho: \Gamma_F \to G(W(k)/p^mW(k))$, $m>0$, by a cocyle
(representing) $f \in H^1(\Gamma_F, \rbarg)$ to get another
representation $\rho': \Gamma_F \to G(W(k)/p^mW(k))$ plays a key role
in the sequel.  We recall here what this means: the kernel of the
reduction map $G(W(k)/p^mW(k)) \to G(W(k)/p^{m-1}W(k))$ is naturally
identified \cite[\S 3.5]{tilouine-def} with $\mg{g}_{k}$ using the
first order exponential map and the generator $p^{m-1}$ of the kernel
of the reduction map $W(k)/p^{m}W(k) \to W(k)/p^{m-1}W(k)$. We denote
this identification by $x \mapsto \mr{exp}(p^{m-1}x)$,
$x \in \mg{g}_{k}$.

For $\rho, f$ as above, the map $\rho': \Gamma_F \to G(W(k)/p^mW(k))$
given by $\gamma \mapsto \mr{exp}(p^{m-1} f(\gamma))\cdot
\rho(\gamma)$ is a continuous homomorphism, equivalent to $\rho$ mod
$p^{m-1}$. Moreover, since $f$ takes values in $\mg{g}^{\mr{der}}_{k}
\subset \mg{g}_{k}$ and $\mg{g}^{\mr{der}}$ is the kernel of the map on
Lie algebras induced by $\mu: G \to C$, the maps $\mu \circ \rho$ and
$\mu \circ \rho'$ from $\Gamma_F$ to $C(W(k)/p^nW(k))$ are equal.

\begin{defn} \label{def:adjust} We call $\rho'$ the representation
obtained from $\rho$ by \emph{adjusting} it by $f$ and denote it by
$\mr{exp}(p^{m-1}f)\rho$.
\end{defn}
We use similar notation when $\Gamma_F$ is replaced by $\Gamma_v$, for
$v$ a prime of $\mc{O}_F$.

\subsection{A matricial condition for auxiliary sets} \label{sec:mat} %\marginpar{sec:mat}%

Recall that $\mc{N} = (\mc{N}_v)_{v \in S}$ is the set of tangent
spaces to the (balanced) smooth local deformation conditions
$(\mc{D}_v)_{v \in S}$ from Assumption \ref{ass:2}.

Let $\{f_1, f_2,\dots, f_n\}$ be a basis of
$H^1_{\mc{N}}(\Gamma_S, \rbarg)$ and $\{\phi_1,\phi_2,\dots,\phi_n\}$
a basis of $H^1_{\mc{N}^{\perp}}(\Gamma_S, \rbarg^*)$; note that $n
\neq 0$ by the assumptions of Theorem \ref{thm:main}. By Assumption
\ref{ass:1} and Remark \ref{rem:abs}, it follows that
$f_i|_{\Gamma_K}: \Gamma_K \to \mg{g}_k^{\mr{der}}$,
$K = F(\rbarg, \mu_p)$, is a surjective homomorphism and similarly for
$\phi_i|_{\Gamma_K}$. Therefore, using the Chebotarev density theorem
and Lemma \ref{lem:kill}, we may pick a set
$\tilde{Q} = \{q_1,q_2,\dots,q_n\}$ of nice primes for $\rbar$ such
that
\begin{itemize}
\item the $\mg{t}_{\alpha} \oplus (\oplus_{\beta}\mg{g}_{\beta}$)
  component of $f_i$
  restricted to $\Gamma_{q_j}$ is $0$ for all $i, j$,
\item the $\mg{l}_{\alpha}$ component of $f_i$ restricted to
  $\Gamma_{q_j}$ is $0$ for $i\neq j$,
\item the $\mg{l}_{\alpha}$ component of $f_i$ restricted to
  $\Gamma_{q_i}$ is $1$ (using the identification fixed in \S \ref{sec:loc})
\end{itemize}
and
\begin{itemize}
\item the $\mg{t}^* \oplus (\oplus_{\beta \neq \alpha}\mg{g}_{\beta}^*)$
  component of $\phi_i$
  restricted to $\Gamma_{q_j}$ is $0$ for all $i, j$,
\item the $\mg{g}_{\alpha}^*$ component of $\phi_i$ restricted to
  $\Gamma_{q_j}$ is $0$ for $i\neq j$,
\item the $\mg{g}_{\alpha}^*$ component of $\phi_i$ restricted to
  $\Gamma_{q_i}$ is $1$ (using the identification fixed in \S
  \ref{sec:loc}).
\end{itemize}

\begin{lem} \label{lem:sds}
  If we augment $\mc{N}$ by using the condition of \S \ref{sec:lc} at
  the primes in $\tilde{Q}$, then the Selmer groups
  $H^1_{\mc{N}}(\Gamma_{S \cup \tilde{Q}}, \rbarg)$ and
  $H^1_{\mc{N}^{\perp}}(\Gamma_{S \cup \tilde{Q}}, \rbarg^*)$ are both
  trivial, so $\tilde{Q}$ is auxiliary.
\end{lem}
\begin{proof}
  The third condition on $f_i$ implies that $f_i|_{\Gamma_{q_i}} \notin
  \mc{N}_{q_i}$ and the third condition on $\phi_i$ implies that
  $\phi_i|_{\Gamma_{q_i}} \notin \mc{N}_{q_i}^{\perp}$; see \cite[Lemma 4.11]{pat-ex}.
 
  The rest of the argument consists in inductively applying the
  Greenberg--Wiles formula \ref{eq:gw}; see \cite[\S 5]{pat-ex},
  especially the argument after the statement of Lemma
  5.3.
\end{proof}

Note that not all the conditions imposed above are necessary for
$\tilde{Q}$ to be auxiliary, but they will all be used later in the
proof of Theorem \ref{thm:main}.

\begin{lem} \label{lem:aux}
  Let $X = S \cup \tilde{Q}$ with $\tilde{Q}$ auxiliary and let
  $\{q_{n+1},q_{n+2},\dots,q_{n+s}\}$ be a set of nice primes for
  $\rbar$ disjoint from $X$.  The kernel of the restriction maps
\[
Res^{n+i}:H^1(\Gamma_{X \cup \{q_{n+i}\}}, \rbarg) \to \bigoplus_{v
  \in X} \frac{H^1(\Gamma_v, \rbarg)}{\mc{N}_v}
\]
is one dimensional for all $i$ and we let $f_{n+i}$ be any nonzero
element in it.  Then
$\tilde{Q} \cup \{q_{n+1},q_{n+2},\dots,q_{n+s}\}$ is auxiliary iff
the matrix $[f_{n+i}({\rm Frob}_{q_{n+j}})]_{1 \leq i,j \leq s}$ is
invertible.
\end{lem}
\begin{proof}
 Recall that when we evaluate a cohomology class at a Frobenius, we are projecting along the 
$\mg{l}_\alpha$ or
$\mg{g}_{\alpha}^*$ component and evaluating as described prior to Lemma\ref{lem:nice}. 

  Let $\mc{L}_v$ = $\mc{N}_v$ for $v \in X$,
  $\mc{L}_{q_{n+i}} = H^1(\Gamma_{q_{n+i}}, \rbarg)$ and
  $\mc{L} = (\mc{L}_v)_{v \in X}$,
  $\mc{L}' = (\mc{L}_v)_{v \in X \cup \{{q_{n+i}}\}}$. Since
  $\tilde{Q}$ is auxiliary, both
  $H^1_{\mc{L}}(\Gamma_{X}, \rbarg)$ and
  $H^1_{\mc{L}^{\perp}}(\Gamma_{X}, \rbarg^*)$ vanish and so
  $ H^1_{\mc{L}'^{\perp}}(\Gamma_{X \cup \{{q_{n+i}}\}}, \rbarg^*)$
  also vanishes since it is contained in
  $H^1_{\mc{L}^{\perp}}(\Gamma_{X}, \rbarg^*)$. Applying the
  Greenberg--Wiles formula \eqref{eq:gw} as in \eqref{eq1} we get
  %\marginpar{eq7}
\begin{equation} \label{eq7} 
  h^1_{\mc{L}'}(\Gamma_{X
    \cup \{{q_{n+i}}\}}, \rbarg)
  % - h^1_{\mc{L}'^{\perp}}(\Gamma_{S \cup \tilde{Q} \cup \{{q_{n+i}}\}},
  % \rbarg^*)
  = h^1(\Gamma_{q_{n+i}}, \rbarg) - h^1_{nr}(\Gamma_{q_{n+i}}, \rbarg)
  = 1 ,
\end{equation}
where the last equality follows from Lemma \ref{lem:nice}.

% Since $h^1_{nr}(\Gamma_{q_{n+i}}, \rbarg) = h^0(\Gamma_{q_{n+i}}, \rbarg)$, it follows
% from \eqref{eq:euler} that the RHS of \eqref{eq7} equals
% $h^2(\Gamma_{q_{n+i}}, \rbarg)$. By \eqref{eq:ltd}, this is equal to
% $h^0(\Gamma_{q_{n+i}}, \rbarg^*) = 1$ by the uniqueness of the root $\alpha$
% in the definition of nice primes.

  Let $K$ be the kernel of the restriction map
\[
  H^1(\Gamma_{X \cup \{q_{n+1},q_{n+2},\dots,
    q_{n+s}\}}, \rbarg) \to \bigoplus_{v \in X}
  \frac{H^1(\Gamma_v, \rbarg)}{\mc{N}_v}\ .
\]
The set $\{f_{n+1},f_{n+2}, \dots,f_{n+s}\}$ spans an $s$-dimensional
subspace of $K$ since $f_{n+i}$ is ramified at $q_{n+i}$ but not at
$q_{n+j}$ for $i \neq j$. Applying the Greenberg--Wiles formula
\eqref{eq:gw} one sees that $\dim(K) = s$, so we have equality. Thus,
the set is auxiliary iff the map
\[
\mr{Span}(\{f_{n+1},f_{n+2}, \dots,f_{n+s}\}) \to \bigoplus_{j =1}^{s}
\frac{H^1(\Gamma_{q_{n+j}}, \rbarg)}{\mc{N}_{q_{n+j}}}
\]
is an isomorphism. Each quotient on the RHS is one dimensional by (2)
of Lemma \ref{lem:nice} and the image is spanned by the row vectors
$[f_{n+i}({\rm Frob}_{q_{n+j}})]_{1 \leq j \leq s}$, so the lemma
follows.
\end{proof}

\begin{rem} \label{rem:matrix}%
  In the sequel, we will use similar, but not the same, ``matricial
  conditions'' as in the lemma to ensure that various sets of primes
  are auxiliary. They can all be deduced by minor variants of the
  same method.
\end{rem}

\subsection{Controlling the kernel of a restriction
  map} \label{sec:ker} %\marginpar{sec:ker}%

In \S \ref{sec:mat} we have chosen a set $\tilde{Q}$ of nice primes
for $\rbar$ satisfying a list of conditions. By Lemma \ref{lem:sds}
$\tilde{Q}$ is auxiliary, so the ring $R_{S \cup \tilde{Q}}^{\tilde{Q}-new}$
is $W(k)$ and
$\rho_{S \cup \tilde{Q}}^{\tilde{Q}-new}: \Gamma_{S \cup \tilde{Q}} \to
G(W(k))$.

Let $n_i$ be the minimum of the set of integers such that
$\rho_{S \cup \tilde{Q}}^{\tilde{Q}-new}$ is ramified modulo $p^{n_i}$
at $q_i$. Set $d$ to be the minimum of all the $n_i$ if this is not
$\infty$, else set $d=2$. For each $v \in {S \cup \tilde{Q}}$, choose
$g_v \in H^1(\Gamma_v, \rbarg)$ as follows:
\begin{itemize}
\item For $v \in S$, let $g_v = 0$
\item For $q_i \in \tilde{Q}$ choose $0 \neq g_{q_i} \in \mc{N}_{q_i}$
  such that:
  \begin{itemize}
  \item If $\rho_{S \cup \tilde{Q}}^{\tilde{Q}-new}$ is unramified
    modulo $p^{d-1}$ at $q_i$ but is ramified modulo $p^d$ then choose
    $0 \neq g_{q_i} \in \mc{N}_{q_i}$ such that
    $\mr{exp}(p^{d-1}g_{q_i})(\rho_{S \cup
      \tilde{Q}}^{\tilde{Q}-new}|_{\Gamma_{q_i}})$ modulo $p^d$ is
    unramified at $q_i$. This $g_{q_i}$ is ramified at $q_i$.
  \item Otherwise choose $g_{q_i}$ to be any ramified element of
    $\mc{N}_{q_i}$.
      \end{itemize}
\end{itemize}

The point of these choices is that for all $a \in k$, 
$\mr{exp}(p^{d-1}ag_v)(\rho_{S \cup \tilde{Q}}^{\tilde{Q}-new}|_{\Gamma_v})$
is in the smooth local deformation condition $\mc{D}_v$  for $v \in S
\cup \tilde{Q}$ and when $ a \neq 1$ it is ramified at  $v = q_i \in
\tilde{Q}$. However, the triviality of the Selmer group for $S \cup \tilde{Q}$  implies that
the local deformation problem $(g_v)_{v \in {S \cup \tilde{Q}}}$ is
not solvable.

By Proposition \ref{prop:hv}, there exists a Chebotarev set of nice
primes $\mc{Q}_0$ such that for $q \in \mc{Q}_0$ there exists
$f^{(q)} \in H^1(\Gamma_{S \cup \tilde{Q} \cup \{q\}}, \rbarg)$ with
$f^{(q)}|_{\Gamma_v} = g_v$ for all $v \in S \cup \tilde{Q}$. We would like to use
$q \in \mc{Q}_0$ as $q_{n+1}$, so want all the sets
$\tilde{Q} \cup \{q_{n+1}\} - \{q_i\}$ to be auxiliary.

In order for this to hold, we need that the $n \times (n+1)$ matrices
$[f_i({\rm Frob}_{q_j})]$ and $[\phi_i({\rm Frob}_{q_j})]$ are
invertible after deleting any column, see Remark \ref{rem:matrix}. To
ensure the  first condition, using Lemma \ref{lem:kill} we impose the further
condition that $f_i({\rm Frob}_{q_{n+1}}) = 1$ for all $i$. 
This is
independent of the conditions in Proposition \ref{prop:hv} since they
involve cocycles for $\rbarg$ as opposed to $\rbarg^*$.  We also
require that %\marginpar{eq4}
\begin{equation} \label{eq4} 
\rho_{S \cup \tilde{Q}}^{\tilde{Q}-new}({\rm Frob}_{q_{n+1}}) \in
T(W(k)/p^dW(k)) \mbox{ mod } p^d \mbox{ is as in Lemma
  \ref{lem:torlift} (2)}
\end{equation}
(so $\rho_{S \cup \tilde{Q}}^{\tilde{Q}-new}$ is \emph{not} special at
$q_{n+1}$). 
To see
that this is achievable we use that our hypotheses on $p$ and $\rbar$
and Lemma \ref{lem:nonsplit} imply that the kernel of the reduction
map $G(W(k)/p^dW(k)) \to G(W(k)/pW(k))$ is contained in the image of $\rho_{S
  \cup \tilde{Q}}^{\tilde{Q}-new}$ mod $p^d$. This is also independent
from the previously imposed conditions by Lemma \ref{lem:kill}.

Finally, we also require that $\phi_i({\rm Frob}_{q_{n+1}}) \neq 0$ for all
$i$.  %
\begin{lem} \label{lem:conflict}
  These conditions do not conflict with the conditions imposed in
  Proposition \ref{prop:hv}
\end{lem}
This is not obvious since both sets of conditions involve cocycles with
values in $\rbarg^*$.

\begin{proof}
  For $v \in S$ we have $g_v = 0$ and for $q_j \in \tilde{Q}\bs
  \{q_i\}$, we have by assumption that $\phi_i|_{\Gamma_{q_j}} =
  0$. Thus, $inv_v(g_v \cup \phi_i) = 0$ for $v \in S \cup \tilde{Q}\bs
  \{q_i\}$. However, $inv_{q_i}(g_{q_i} \cup \phi_i) \neq 0$ by Lemma
  \ref{lem:nice} (3), so $Res_{S \cup \tilde{Q}}(\phi_i)$ is not in the
  annihilator of the line $l$ spanned by $(g_v)_{v \in S \cup
    \tilde{Q}}$. It follows that writing $\phi_i = a_i \zeta + \sum_j
  b_{i,j}\zeta_j$, where $\{\zeta_1,\dots,\zeta_s,\zeta\}$ is the
  basis of $H^1(\Gamma_{S \cup \tilde{Q}}, \rbarg^*)$ as in
  Proposition \ref{prop:hv}, we have $a_i \neq 0$. By evaluating the
  RHS at ${\rm Frob}_{q_{n+1}}$, bearing in mind the conditions of
  Proposition \ref{prop:hv}, we see that $\phi_i({\rm Frob}_{q_{n+1}}) \neq
  0$.
\end{proof}

Thus, the nonempty Chebotarev set of nice primes $q \in \mc{Q}_0$
satisfying: %\marginpar{eq5}%
\begin{equation} \label{eq5} 
\begin{split}
  & \zeta_i |_{\Gamma_q} = 0 \mbox{ for } i =1,2,\dots s, \\
  & \mbox{ the } \mg{g}_{\alpha}^* \mbox{ component of } \zeta
  |_{\Gamma_q} \mbox{ is
    nonzero}, \\
  & f_i({\rm Frob}_q) = 1 \mbox{ for all } i,\\
  & \rho_{S \cup \tilde{Q}}^{\tilde{Q}-new}({\rm Frob}_{q}) \in
T(W(k)/p^dW(k)) \mbox{ mod } p^d \mbox{ is as in Lemma
  \ref{lem:torlift} (2)},
\end{split}
\end{equation}
also satisfies %\marginpar{eq6}
\begin{equation} \label{eq6} 
\begin{split}
  & \phi_i({\rm Frob}_q) \neq 0 \mbox{ for all } i, \\
  &\exists \ f^{(q)} \in H^1(\Gamma_{{S \cup \tilde{Q}} \cup \{q\}},
  \rbarg) \mbox{ such that } f^{(q)}|_{\Gamma_v} = g_v \mbox{ for all
  } v \in {S \cup \tilde{Q}} \ .
\end{split}
\end{equation}
%\marginpar{lem:onedimker}
\begin{lem} \label{lem:onedimker} For $\tilde{Q}$ as above and any $q
  \in \mc{Q}_0$, $f^{(q)}$ spans the one dimensional kernel of the
  restriction map
\begin{equation}
H^1(\Gamma_{{S \cup \tilde{Q}} \cup \{q\}}, \rbarg) \to \bigoplus_{v \in S
  \cup \tilde{Q}} \frac{H^1(\Gamma_v, \rbarg)}{\mc{N}_v}\ .
\end{equation}
\end{lem}

\begin{proof}
  The fact that the kernel is one dimensional follows from the same
  argument as in the first part of the proof of Lemma \ref{lem:aux}.
  Since $g_v \in \mc{N}_v$ for $v \in S \cup \tilde{Q}$, $f^{(q)}$
  lies in the kernel and is nonzero, so it must span.
\end{proof}

\subsection{Proof of  Theorem \ref{thm:main}} \label{sec:proof} %\marginpar{sec:proof}%

\begin{proof}[Proof of Theorem \ref{thm:main}]

  Denote the  positive  density of the set of nice primes $\mc{Q}_0$ constructed in \S
  \ref{sec:ker} by $t_0$. At least one of the following must be true:
\begin{enumerate}
\item $\exists \ q \in \mc{Q}_0$ such that $f^{(q)}({\rm Frob}_q) \neq
  0,1$\ ,
\item $\exists \ \mc{Q}_1 \subset \mc{Q}_0$ having positive upper
  density such that $q \in \mc{Q}_1 \Rightarrow
  f^{(q)}({\rm Frob}_q) = 1$\ ,
\item  $\exists \  \mc{Q}_1 \subset \mc{Q}_0$ having positive upper
  density such that $q \in \mc{Q}_1 \Rightarrow
  f^{(q)}({\rm Frob}_q) = 0$\ .
\end{enumerate}
We expect that all three cases do actually occur, but this seems
difficult to verify so we prove the theorem by considering all three
cases separately.

%We treat these three cases in succession.
For a prime $q_{n+i}$ we use the more compact notation $f_{n+i}$
instead of $f^{(q_{n+i})}$. At various points in the proof we will,
for auxiliary sets $A$ and $B$, move from $\rho_{S \cup A}^{A-new}$
mod $p^d$ to $\rho_{S \cup B}^{B-new}$ mod $p^d$ by adjusting by one
or two cohomology classes. While our choices may seem arbitrary, since
$A$ and $B$ are auxiliary, there is exactly one class that works and
after adjustment by this class the problem is uniquely liftable.

In each case we must determine the auxiliary sets $Q$ and $Q_i$ as in
Theorem \ref{thm:main}, show that $\rho_{S \cup Q}^{Q-new}$ is
ramified at all $q \in Q$ and also that $\rho_{S \cup Q_i}^{Q_i-new}$
is unramified and \emph{not} special at $q_i$. Note that
$\mr{exp}(p^{d-1}f_{n+1}) \rho_{S \cup \tilde{Q}}^{\tilde{Q}-new}$ mod
$p^d$ is only special at $q_{n+1}$ in Case 2).

\smallskip

\noindent {\bf Proof of Case (1)} Choose any $q_{n+1} \in \mc{Q}_0$
such that $f^{(q_{n+1})}({\rm Frob}_{q_{n+1}}) = a \neq 0,1$ and set
\[
Q = \tilde{Q} \cup \{q_{n+1}\}, \  Q_i = Q \bs \{q_i\}, \ i=1,2,\dots,n+1.
\]
We first show that $Q$ and all $Q_i$ are auxiliary. Recall that for
$v \in S \cup \tilde{Q}$, $f_{n+1}|_{\Gamma_v} \in \mc{N}_v$. The
map
\[
 H^1(\Gamma_{S \cup \tilde{Q}}, \rbarg) \to \bigoplus_{v \in S
  \cup \tilde{Q}} \frac{H^1(\Gamma_v, \rbarg)}{\mc{N}_v}\ 
\]
is an isomorphism since $\tilde{Q}$ is auxiliary and by Lemma
\ref{lem:onedimker} the kernel of the map
\[
H^1(\Gamma_{S  \cup Q}, \rbarg) \to \bigoplus_{v \in S
  \cup \tilde{Q}} \frac{H^1(\Gamma_v, \rbarg)}{\mc{N}_v}
\]
is spanned by $f_{n+1}$. As $f_{n+1}({\rm Frob}_{q_{n+1}}) \neq 0$, so
$f_{n+1}|_{\Gamma_{q_{n+1}}} \notin \mc{N}_{q_{n+1}}$, the restriction
map
\[
 H^1(\Gamma_{S \cup Q}, \rbarg) \to \bigoplus_{v \in S
  \cup Q} \frac{H^1(\Gamma_v, \rbarg)}{\mc{N}_v}\ 
\]
is an injection, so $Q$ is auxiliary.

Using the third part of \eqref{eq5} and the first part of \eqref{eq6}
we see that the $n \times (n + 1)$ matrix
$F = [f_i({\rm Frob}_{q_j})]$
(resp.~$ \Phi= [\phi_i({\rm Frob}_{q_j})]$) is the identity matrix
with an extra column all of whose entries are $1$
(resp.~nonzero). Such matrices are invertible upon the deletion of any
column so by Lemma \ref{lem:sds} and the arguments of Lemma
\ref{lem:aux} (see Remark \ref{rem:matrix}) the sets $Q_i$ are all
auxiliary.

We now show that $\rho_{S \cup Q}^{Q-new}$ is ramified at all
$q_i \in Q$. Recall that for $1 \leq i \leq n$, 
$f_i |_{\Gamma_{q_i}} \notin \mc{N}_{q_i}$ and for
$v \in S \cup \tilde{Q} \bs \{q_i\}$,  $f_i |_{\Gamma_v} \in
\mc{N}_v$.  By \eqref{eq4} we know that
$\rho_{S \cup \tilde{Q}}^{\tilde{Q}-new}$ mod $p^d$ is not special at
$q_{n+1}$ but is special at all $v \in S \cup \tilde{Q}$.  We adjust
$\rho_{S \cup \tilde{Q}}^{\tilde{Q}-new}$ mod $p^d$ by
$\frac{1}{a}f_{n+1}$, i.e., we consider the representation
$\rho' = \mr{exp}(p^{d-1} \frac{1}{a}f_{n+1}) (\rho_{S \cup
  \tilde{Q}}^{\tilde{Q}-new})$ mod
$p^d$. %\footnote{maybe there should be a $-$ sign instead}.
Then $\rho'$ is special at $q_{n+1}$ and is also ramified at $q_{n+1}$
since $f_{n+1}$ is. As $f_{n+1}|_{\Gamma_v} \in \mc{N}_v$ for all
$v \in S \cup \tilde{Q}$, $\rho'$ is $\tilde{Q}$-new
%liftable at all $v \in S \cup \tilde{Q}$
so (by uniqueness) we have
$\rho' = \rho_{S \cup Q}^{Q-new}$ mod $p^d$. As $a \neq 1$,
$\frac{1}{a}f_{n+1}|_{\Gamma_{q_i}} = \frac{1}{a} g_{q_i} \neq
g_{q_i}$, so this class introduces ramification at all $q_i$ that were
unramified and \emph{does not} remove ramification at any $q_i$. Thus
$\rho_{S \cup Q}^{Q-new}$ is ramified at all primes of $Q$ and this
ramification occurs mod $p^d$ for all these primes.

It remains to show that $\rho_{S \cup Q_i}^{Q_i-new}$ mod $p^d$ is unramified and not
special at $q_i$. As $Q_{n+1} = \tilde{Q}$ and
$\rho_{S \cup \tilde{Q}}^{\tilde{Q}-new}$ mod $p^d$ is not special at
$q_{n+1}$ the $i = n+1$ case is settled.

When $1 \leq i \leq n$, to get from
$\rho_{S \cup \tilde{Q}}^{\tilde{Q}-new}$ mod $p^d$ to
$\rho_{S \cup Q_i}^{Q_i-new}$ mod $p^d$ we need to remove (possible)
ramification at $q_i$ and make the representation nonspecial there
while ensuring specialness at all other $q_j \in Q$.  Since
$f_{n+1}|_{\Gamma_{v}} = g_v$, adjusting by $f_{n+1}$ removes any
ramification at $q_i$ (and perhaps some other $q_j$) while preserving
specialness there; if there is no ramification at $q_i$ we do not
perform this step. Our Case 1) assumption implies that this adjustment
also keeps the representation nonspecial at $q_{n+1}$. Since
$f_i({\rm Frob}_{q_{i}}) = f_i({\rm Frob}_{q_{n+1}}) = 1$, adjusting
further by a suitable nonzero multiple of $f_i$ makes the
representation special at $q_{n+1}$ and nonspecial at $q_i$ as
desired. As $f_i$ is in the Selmer group for $S$ and trivial when
restricted to $\Gamma_{q_j}$ for $j \leq n$, $j \neq i$, the
representation remains special at all other primes.

  \smallskip

\noindent {\bf Proof of Case (2)}
The sets $Q$ and $Q_i$ of Case (1) are auxiliary in this case as well,
but we cannot use them as we cannot guarantee ramification of
$\rho_{S \cup Q}^{Q-new}$ at all $q_i \in Q$. The problem is that as
$f_{n+1}({\rm Frob}_{q_{n+1}}) = 1$, adjusting by $f_{n+1}$ makes the
representation special at $q_{n+1}$ and liftable at all other primes,
but will remove ramification at all $q_i$ at which
$\rho_{S \cup \tilde{Q}}^{\tilde{Q}-new}$ mod $p^d$ is ramified. We
will need to add three primes to $\tilde{Q}$ to form $Q$ in this case.

\smallskip

For this case we are assuming that $\mc{Q}_1 = \{q \in \mc{Q}_0| f^{(q)}({\rm Frob}_q) = 1\}$
has positive upper density. Note that $\mc{Q}_1$ is almost certainly
not a Chebotarev set, though this seems difficult to settle one way or
the other.

In this case $Q = \tilde{Q} \cup \{q_{n+1}, q_{n+2}, q_{n+3}\}$ where
$q_i \in \mc{Q}_1$ are carefully chosen; the $\mc{Q}_i$ will be
described later. We will find three primes such that
$[f_{n+i}({\rm Frob}_{q_{n+j}} )]_{1\leq i,j\leq 3} =I_3$. While we could
make do with less, the argument is quite transparent with this choice
of matrix. Also, our method shows that the techniques of \cite{klr} can
be strengthened with more effort.

Choose $q_{n+1} \in \mc{Q}_1$. The $n \times (n+1)$ matrices
$F = [f_i({\rm Frob}_{q_j})]$ and $\Phi = [\phi_i({\rm Frob}_{q_j})]$ are the
$n \times n$ identity matrix augmented by a column with no zero
entries (by the third line of \eqref{eq5} and the first line of
\eqref{eq6}). For $q_{n+2} \in \mc{Q}_1$, the matrix
$[f_{n+i}({\rm Frob}_{q_{n+j}})]_{1 \leq i,j \leq 2} = \bigl
[ \begin{smallmatrix} 1 & x \\ z & 1 \end{smallmatrix} \bigr ]$.
For our eventual $3 \times 3$ matrix to be the identity we need
$x=z=0$.

We now choose an appropriate $q_{n+2}$. We need
$f_{n+1}({\rm Frob}_{q_{n+2}}) = 0$ and also $f_{n+2}({\rm Frob}_{q_{n+1}}) = 0$.
The first of these is a Chebotarev condition disjoint from those
determining $\mc{Q}_0$ and we will see that the second is as well!
Dualizing the proof of Lemma \ref{lem:onedimker} (taking
$\mc{L}_q = 0$ instead of $H^1(\Gamma_q, \rbarg)$) one sees that for
$e \in \{n+1, n+2\}$, the map
\[
H^1(\Gamma_{S \cup \tilde{Q} \cup \{q_e\}}, \rbarg^*) \to \bigoplus_{v \in
  S \cup \tilde{Q}} \frac{H^1(\Gamma_v, \rbarg^*)}{\mc{N}_v^{\perp} }
\]
has one dimensional kernel. Choose $\phi_e$ to span this kernel and
scale it---that this is possible follows from Lemma
\ref{lem:nice}---so that for any $h \in H^1_{nr}(\Gamma_{q_e},
\rbarg)$, $inv_{q_e}(h \cup \phi_e) = h({\rm Frob}_{q_e})$.

The first equality below is the global reciprocity law. The second
follows since for $v \in S \cup \tilde{Q}$,
$f_{n+1}|_{\Gamma_v} \in \mc{N}_v$ and
$\phi_{n+2}|_{\Gamma_v} \in \mc{N}_v^{\perp}$ annihilate each other.
%\marginpar{eq9}
\begin{equation} \label{eq9}
\begin{split}
0 &= \sum_{v \in S \cup \{q_1,q_2,\dots, q_{n+2}\}} inv_v(f_{n+1} \cup
\phi_{n+2}) \\
&= inv_{q_{n+1}}(f_{n+1} \cup
\phi_{n+2}) + inv_{q_{n+2}}(f_{n+1} \cup
\phi_{n+2}).
\end{split}
\end{equation}
We then have
\begin{multline*}
f_{n+1}({\rm Frob}_{q_{n+2}}) = 0 \Leftrightarrow inv_{q_{n+2}}(f_{n+1}
\cup \phi_{n+2}) = 0 \\ 
\Leftrightarrow  inv_{q_{n+1}}(f_{n+1} \cup \phi_{n+2}) = 0  
\Leftrightarrow \phi_{n+2}({\rm Frob}_{q_{n+1}}) = 0,
\end{multline*}
where the middle ``iff'' follows from \eqref{eq9} and the outer ones
follow from the choice of $\phi_{n+2}$.

We will need the set of $q_{n+2} \in \mc{Q}_1$ satisfying
%\marginpar{eq9b}
\begin{equation} \label{eq9b}
(f_{n+1}, \phi_{n+1})({\rm Frob}_{q_{n+2}}) = (0,0)
\end{equation}
to have positive upper density. Note that the complement
 of these two
Chebotarev conditions on $f_{n+1}$ and $\phi_{n+1}$ is a set of
density $t_0\left(1 - \frac{1}{|k|^2}\right)$ (where $|k|$ is the order of the
field $k$) within $\mc{Q}_0$. Suppose that for our choice of prime
$q_{n+1}$ the set of desired second primes in $\mc{Q}_1$ has upper
density $0$. Rename $q_{n+1}$ to $\ell_1$ and let
$S_1 \subset \mc{Q}_1$ be the density $0$ set of primes satisfying the
above equation. Then $\mc{Q}_1 \bs (\{\ell_1\} \cup S_1)$ lies in a
subset of $\mc{Q}_0$ that is the complement of the conditions of
\eqref{eq9b} and which has density $t_0(1 - \frac{1}{|k|^2})$. Now
choose $\ell_2 \in \mc{Q}_1 \bs \{\ell_1\}$. If the set of second
primes, $S_2$, that gives us the desired Frobenius matrix has density
$0$, then $\mc{Q}_1 \bs (\{\ell_1,\ell_2\} \cup S_1 \cup S_2)$ is
contained in a subset of $\mc{Q}_0$ that is the complement of two sets
of independent conditions \eqref{eq9}. This set has density
$t_0(1 - \frac{1}{|k|^2})^2$.  Continuing in this way, we get that
$\mc{Q}_1 \bs (\{\ell_1,\ell_2,\dots,\ell_r\} \cup S_1 \cup S_2 \cup
\dots \cup S_r)$ is contained in a set of density
$t_0(1 - \frac{1}{|k|^2})^r$.  Since $\mc{Q}_1$ has positive upper
density, if each $S_i$ has density $0$ then we get a contradiction for
large $r$. Thus there is a $q_{n+1} \in \mc{Q}_1$ such that for a set
$\mc{Q}_2 \subset \mc{Q}_1$ of positive density
\[
(f_{n+1}, \phi_{n+1})({\rm Frob}_{q_{n+2}}) = (0,0) \Leftrightarrow
f_{n+1}({\rm Frob}_{q_{n+2}}) =0, \ f_{n+2}({\rm Frob}_{q_{n+1}}) = 0\ .
\]

Choose $q_{n+2} \in \mc{Q}_{n+2}$. For any $q_{n+3} \in \mc{Q}_2$ we
then have
$[f_{n+i}({\rm Frob}_{q_{n+j}})]_{1 \leq i,j \leq 3} = \Bigl
[ \begin{smallmatrix} 1 & 0 &0 \\ 0 & 1 & a \\ 0 & b &
  1 \end{smallmatrix} \Bigr ]$ and we need $a = b =0$.  As before,
\[
f_{n+3}({\rm Frob}_{q_{n+2}}) = 0 \Leftrightarrow \phi_{n+2}({\rm Frob}_{q_{n+3}})
= 0
\]
so we need the independent Chebotarev conditions
%\marginpar{eq10}
\begin{equation} \label{eq10}
(f_{n+2}, \phi_{n+2})({\rm Frob}_{q_{n+3}}) = (0,0)
\end{equation}
to hold. Suppose for our choice of $q_{n+2}$ there is no
$q_{n+3} \in \mc{Q}_2$ satisfying \eqref{eq10}. Repeating the same
limiting argument as above, we see that we can find some other
$q_{n+2} \in \mc{Q}_2$ so that there is a prime $q_{n+3} \in \mc{Q}_2$
satisfying \eqref{eq10}.

We have thus found primes $q_{n+1}, q_{n+2}, q_{n+3}$ so that
$[f_{n+i}({\rm Frob}_{q_{n+j}})]_{1 \leq i,j \leq 3} = \Bigl
[ \begin{smallmatrix} 1 & 0 &0 \\ 0 & 1 & 0 \\ 0 & 0 &
  1 \end{smallmatrix} \Bigr ]$.
Set 
\[Q = \tilde{Q} \cup \{ q_{n+1}, q_{n+2}, q_{n+3}\};
\]
that $Q$ is auxiliary then follows from Lemma \ref{lem:aux}. 

We now establish that $\rho_{S \cup Q}^{Q-new}$ is ramified at all
$q \in Q$. To get from $\rho_{S \cup \tilde{Q}}^{\tilde{Q}-new}$ mod
$p^d$ to $\rho_{S \cup Q}^{Q-new}$ mod $p^d$ requires making the
representation special at $q_{n+1}, q_{n+2}, q_{n+3}$, so we adjust by
$h = f_{n+1} + f_{n+2} + f_{n+3}$. Observe that
$h({\rm Frob}_{q_{n+i}}) = 1$ for $i=1,2,3$, and for
$v \in S \cup \tilde{Q}$, $h|_{\Gamma_v} = 3g_v \in \mc{N}_v$ and
$h|_{\Gamma_v} \neq g_v$, so if
$\rho_{S \cup \tilde{Q}}^{\tilde{Q}-new}$ is unramified at
$v \in \tilde{Q}$ we are introducing ramification and if it is
ramified at $v$ we are not removing the ramification. Also, adjusting
by $h$ introduces ramification at each of $q_{n+1}, q_{n+2}, q_{n+3}$
and makes the representation special at each of these primes so this
adjusted representation must be equal to $\rho_{S \cup Q}^{Q-new}$ mod
$p^d$ (which is then ramified at all $q \in Q$).

For $i \leq n$ our choice of $Q_i$ depends on whether or not $\rho_{S
  \cup \tilde{Q}}^{\tilde{Q}-new}$ mod $p^d$ is ramified at $q_i$. We
set
\[
Q_i = \begin{cases}
  Q\bs\{q_i\} & i=n+2,n+3\\
  \{q_1, q_2,\dots, q_n\}   & i = n+1 \\
  \{q_1,q_2,\dots,q_{n+1}\} \bs \{q_i\} & i \leq n, \ \rho_{S \cup
    \tilde{Q}}^{\tilde{Q}-new} \mbox{ mod } p^d \mbox{ is
    unramified at } q_i \\
  Q \bs \{q_i, q_{n+3}\} & i \leq n, \ \rho_{S \cup
    \tilde{Q}}^{\tilde{Q}-new} \mbox{ mod } p^d \mbox{ is ramified at
  } q_i\ .
\end{cases}
\]

We first show that each $Q_i$ is auxiliary and then establish that $\rho^{Q_i-new}_{S\cup Q_i}$ is nonpecial at $q_i$. 

When $i = n+2, n+3$, that $Q_i$ is auxiliary follows from Lemma
\ref{lem:aux} since the relevant principal minors of
$[f_{n+i}({\rm Frob}_{q_{n+j}})]_{1 \leq i,j \leq 3}$ are invertible.  That
$Q_{n+1}$ is auxiliary is clear since $Q_{n+1} = \tilde{Q}$.

For $1 \leq i \leq n$ we will show both choices of $Q_i$ are
auxiliary; we need both to guarantee nonspecialness at $q_i$. That the
first choice of $Q_i$ is auxiliary was demonstrated in Case (1) when
we saw that the $n \times (n+1)$ matrices $F = [f_i({\rm Frob}_{q_j})]$ and
$\Phi = [\phi_i({\rm Frob}_{q_j})]$ are the $n \times n$ identity matrix
augmented by a column with no zero entries.  We now show our second
choice of $Q_i$ is auxiliary. The set
$\{q_1,q_2,\dots,q_n,q_{n+1}\}\bs \{q_i\}$ is auxiliary and the
map
\[
H^1(\Gamma_{S \cup \tilde{Q} \bs \{q_i\}}, \rbarg) \to \bigoplus_{v
  \in S \cup \tilde{Q} \bs \{q_i\}} \frac{H^1(\Gamma_v, \rbarg)}{\mc{N}_v}
\]
has kernel spanned by $f_i$. We claim that the map
%\marginpar{eq11}
\begin{equation} \label{eq11}
H^1(\Gamma_{S \cup \tilde{Q} \cup \{q_{n+1}\} \bs \{q_i\}}, \rbarg) \to \bigoplus_{v
  \in S \cup \tilde{Q} \bs \{q_i\}} \frac{H^1(\Gamma_v, \rbarg)}{\mc{N}_v}
\end{equation}
also has kernel spanned by $f_i$. Clearly $f_i$ is in the kernel and
the only other possibility is that the kernel is two dimensional. If
it were, then changing the target direct sum as below
\[
H^1(\Gamma_{S \cup \tilde{Q} \cup \{q_{n+1}\} \bs \{q_i\}}, \rbarg)
\to \bigoplus_{v \in S \cup \tilde{Q} \cup \{q_{n+1}\} \bs \{q_i\}}
\frac{H^1(\Gamma_v, \rbarg)}{\mc{N}_v}
\]
would yield at least a one dimensional kernel which contradicts that
$\tilde{Q} \cup \{q_{n+1}\} \bs \{q_i\}$ is
auxiliary.

Thus the kernel of \eqref{eq11} is spanned by $f_i$ and so the kernel
of
\begin{equation} \label{eq13}
H^1(\Gamma_{S \cup \tilde{Q} \cup \{q_{n+1}, q_{n+2}\} \bs \{q_i\}},
\rbarg) \to \bigoplus_{v \in S \cup \tilde{Q} \bs \{q_i\}}
\frac{H^1(\Gamma_v, \rbarg)}{\mc{N}_v}
\end{equation}
is two dimensional. Set $g_{n+1} = f_i$ and
$g_{n+2} = f_{n+1} - f_{n+2}$. 
 By construction $g_{n+2}$ is in the kernel of \eqref{eq13} so $\{g_{n+1},g_{n+2} \}$ forms a basis for this kernel. 
The $2 \times 2$ matrix
$[g_{n+i}({\rm Frob}_{q_{n+j}})] = \bigl [ \begin{smallmatrix} 1 & 1 \\ 1 &
  -1 \end{smallmatrix} \bigr ]$
is invertible, so by (the proof of) Lemma \ref{lem:aux}
\[
H^1(\Gamma_{S \cup \tilde{Q} \cup \{q_{n+1}, q_{n+2}\} \bs \{q_i\}},
\rbarg) \to \bigoplus_{v \in S \cup \tilde{Q} \cup \{q_{n+1}, q_{n+2}\}
  \bs \{q_i\}} \frac{H^1(\Gamma_v, \rbarg)}{\mc{N}_v}
\]
is an isomorphism and $\mc{Q}_i$ is auxiliary.

Finally, we show that $\rho_{S \cup Q_i}^{Q_i-new}$ mod $p^d$ is
\emph{not} special at $q_i$ in all cases:
\begin{itemize}
\item $ i = n+2$: to get from
  $\rho_{S \cup \tilde{Q}}^{\tilde{Q}-new}$ mod $p^d$ to
  $\rho_{S \cup Q_{n+2}}^{Q_{n+2}-new}$ mod $p^d$ we adjust by
  $h = f_{n+1} + f_{n+3}$ to make the representation special at
  $q_{n+1}$ and $q_{n+3}$. As $h({\rm Frob}_{q_{n+2}}) = 0$
  % the action of
  % ${\rm Frob}_{q_{n+2}}$ on the $\mg{g}_{\alpha}$ component of
  % $\mr{exp}( p^{d-1}h) \rho_{S \cup \tilde{Q}}^{\tilde{Q}-new}$ mod $p^d$
  % is given by multiplication by $\mr{Norm}(q_{n+2}) - p^{d-1}$
  %  so is not special.
  and $\rho_{S \cup \tilde{Q}}^{\tilde{Q}-new}$ is not special at
  $q_{n+2}$ (by construction, cf.~\eqref{eq6}), the representations
  remains nonspecial at $q_{n+2}$.  For $v \in S \cup \tilde{Q}$,
  $h|_{\Gamma_v} = 2g_v \in \mc{N}_v$, so adjusting by $h$ keeps the
  representation special at $v$ without removing ramification. We
  conclude as before by uniqueness.
\item $i = n+3$: set $h = f_{n+1} + f_{n+2}$ and proceed as in the
  $i = n+2$ case.
\item $i = n+1$: as $Q_{n+1} = \tilde{Q}$ and
  $\rho_{S \cup \tilde{Q}}^{\tilde{Q}-new}$ mod $p^d$ is not special
  at $q_{n+1}$ by construction we are done.
\item $ i \leq n$: we go from
  $\rho_{S \cup \tilde{Q}}^{\tilde{Q}-new}$ mod $p^d$ to
  $\rho_{S \cup Q_{i}}^{Q_{i}-new}$ mod $p^d$. Our starting point
  is not special at $q_{n+1}$ and $q_{n+2}$.
  \begin{itemize}
  \item If $\rho_{S \cup \tilde{Q}}^{\tilde{Q}-new}$ mod $p^d$ is
    unramified at $q_i$ then adjust by a (nonzero) multiple of $f_i$
    to make the representation special at $q_{n+1}$. This makes the
    representation nonspecial at $q_i$ as desired (but preserves
    unramifiedness) and preserves liftability at the other primes in
    $S \cup \tilde{Q}$, so we are done. (If
    $\rho_{S \cup \tilde{Q}}^{\tilde{Q}-new}$ mod $p^d$ were ramified
    at $q_i$ and we used this $Q_i$, we would need to adjust by
    $f_{n+1}$ to remove ramification at $q_i$. But then the
    representation would be special at both $q_i$ and $q_{n+i}$ and
    liftable elsewhere. We need nonspecialness at $q_i$, which is why
    we need the $Q_i$ below in the ramified case.)
  \item If $\rho_{S \cup \tilde{Q}}^{\tilde{Q}-new}$ mod $p^d$ is
    ramified at $q_i$, then when going from
    $\rho_{S \cup \tilde{Q}}^{\tilde{Q}-new}$ mod $p^d$ to
    $\rho_{S \cup Q_{i}}^{Q_{i}-new}$ mod $p^d$ we need to make
    the representation special at $q_{n+1}$ and $q_{n+2}$ and unramified and
    nonspecial at $q_i$. The cohomology class $h$ that we use must
    satisfy:
    \begin{itemize}
    \item $h({\rm Frob}_{q_{n+i}}) = 1$ for $i=1,2$ (specialness at
      $q_{n+1}, q_{n+2}$),
    \item $h|_{\Gamma_{q_i}} = g_{q_i} +$ a nonspecial unramified
      class (removes ramification at $q_i$ and makes it nonspecial),
    \item $h$ is unramified at $q_{n+3}$,
    \item $h|_{\Gamma_v} \in \mc{N}_v$ for
      $v \in S \cup \tilde{Q} \bs \{q_i\}$ (preserves liftability at
      these places).
    \end{itemize}
    The class $h = \frac{1}{2}(f_i + f_{n+1} + f_{n+2})$ satisfies all
    these conditions, so we are done by uniqueness.  Note that in this
    case the representation may become unramified at some
    $v \in S \cup \tilde{Q}$ (besides $q_i$).
  \end{itemize}
\end{itemize}

\smallskip
\noindent{\bf Proof of Case 3)} This is the most delicate of the three
cases and makes important use of parts 3) and 4) of Lemma
\ref{lem:nice}.

For this case we are assuming $\mc{Q}_1 = \{q \in \mc{Q}_0 | f_{n+1}({\rm Frob}_q)=0\}$ has
positive upper density. Note that this is almost certainly \emph{not}
a Chebotarev set, but again this seems difficult to prove. Our set $Q$
will be $\tilde{Q} \cup \{q_{n+1},q_{n+2},q_{n+3}\}$ for three
carefully chosen primes in $\mc{Q}_1$ such that the matrix
\[
[f_{n+i}({\rm Frob}_{q_{n+j}})]_{1 \leq i,j \leq 3} = \Bigl
[ \begin{smallmatrix} 0 & 1 & 1 \\ 1 & 0 & 1 \\ 1 & 1 &
  0\end{smallmatrix} \Bigr ]\ .
\]
After constructing $Q$ we will describe the sets $Q_i$.

Choose $q_{n+1} \in \mc{Q}_1$. Then the map
\[
H^1(\Gamma_{S \cup \tilde{Q} \cup \{q_{n+1}\} }, \rbarg) \to \bigoplus_{v
  \in S \cup \tilde{Q}} \frac{H^1(\Gamma_v, \rbarg)}{\mc{N}_v}
\]
has kernel spanned by $f_{n+1}$ and since
$q_{n+1} \in \mc{Q}_1 \Rightarrow f_{n+1}|_{\Gamma_{q_{n+1}}} \in
\mc{N}_{q_{n+1}}$, we see that $f_{n+1}$ spans the kernel of
\[
H^1(\Gamma_{S \cup \tilde{Q} \cup \{q_{n+1}\} }, \rbarg) \to \bigoplus_{v
  \in S \cup \tilde{Q} \cup \{q_{n+1}\}} \frac{H^1(\Gamma_v,
  \rbarg)}{\mc{N}_v}
\]
as well.

For $q_{n+1}$ fixed, consider another prime $q_{n+2} \in \mc{Q}_1$.
The matrix
$[f_{n+i}({\rm Frob}_{q_{n+j}})]_{1 \leq i,j \leq 2} = \bigl
[ \begin{smallmatrix} 0 & x \\ z & 0 \end{smallmatrix} \bigr ]$
and we want $x = z = 1$. As the kernel of 
\[
H^1(\Gamma_{S \cup \tilde{Q} \cup \{q_{n+1}, q_{n+2}\} }, \rbarg) \to
\bigoplus_{v \in S \cup \tilde{Q} } \frac{H^1(\Gamma_v,
  \rbarg)}{\mc{N}_v}
\]
is spanned by $\{f_{n+1}, f_{n+2}\}$, Lemma \ref{lem:aux} implies that
$S \cup \tilde{Q} \cup \{q_{n+1}, q_{n+2}\}$ is auxiliary if this
holds.

For a potential second prime $q_{n+2} \in \mc{Q}_1$, we need that
$f_{n+1}({\rm Frob}_{q_{n+2}}) = 1$ and this is a Chebotarev condition
disjoint from the one determining $\mc{Q}_0$. We also need
$f_{n+2}({\rm Frob}_{q_{n+1}}) = 1$. For $e \in \{n+1, n+2\}$, let $\phi_e$
as before span the kernel of
\[
H^1(\Gamma_{S \cup \tilde{Q} \cup \{q_e\}}, \rbarg^*) \to \bigoplus_{v \in
  S \cup \tilde{Q}} \frac{H^1(\Gamma_v, \rbarg^*)}{\mc{N}_v^{\perp} }
\]
and scale it as in (4) of Lemma \ref{lem:nice}. The first equality
below is the global reciprocity law and the second holds as, for $v
\in S \cup \tilde{Q}$, $f_{n+2}|_{\Gamma_v} \in \mc{N}_v$ and
$\phi_{n+1}|_{\Gamma_v} \in \mc{N}_v^{\perp}$ annihilate each other.
The third uses parts (3) and (4) of Lemma \ref{lem:nice}; we finally
see the reason for the asymmetry in our statements for these parts
since we have the freedom to scale $\phi_{n+1}$ but not $f_{n+2}$ as
it comes from Proposition \ref{prop:hv}.  %\marginpar{eq12}
\begin{equation} \label{eq12}
\begin{split}
0 &= \sum_{v \in S \cup \{q_1,q_2,\dots, q_{n+2}\}} inv_v(f_{n+2} \cup
\phi_{n+1}) \\
&= inv_{q_{n+1}}(f_{n+2} \cup
\phi_{n+1}) + inv_{q_{n+2}}(f_{n+2} \cup
\phi_{n+1}) \\
& = f_{n+2}({\rm Frob}_{q_{n+1}}) + \gamma_{q_{n+2}}
\phi_{n+1}({\rm Frob}_{q_{n+2}}).
\end{split}
\end{equation}
Note that $\gamma_{q_{n+2}}$ only depends on $f_{n+2}$ and thus on
$q_{n+2}$ but not $q_{n+1}$. Since $\mc{Q}_1$ has positive upper
density and $k^{\times}$ is finite, there is a $\gamma_0 \in k^{\times}$ and a
subset $\mc{Q}_2 \subset \mc{Q}_1$ of positive upper density such that
$q \in \mc{Q}_2 \Rightarrow \gamma_q = \gamma_0$. So for $q_{n+1},
q_{n+2} \in \mc{Q}_2$,
\[
f_{n+2}({\rm Frob}_{q_{n+1}}) + \gamma_0 \phi_{n+1}({\rm Frob}_{q_n+2}) = 0,
\]
or equivalently
\[
f_{n+2}({\rm Frob}_{q_{n+1}}) = - \gamma_0 \phi_{n+1}({\rm Frob}_{q_n+2}).
\]
Henceforth, we will only choose primes from $\mc{Q}_2$. Having chosen
$q_{n+1} \in \mc{Q}_2$, we want the set of $q_{n+2} \in \mc{Q}_2$
satisfying
\[
(f_{n+1}({\rm Frob}_{q_{n+2}}) = 1 \mbox{ and } f_{n+2}({\rm Frob}_{q_{n+1}}) = 1)
\Leftrightarrow (f_{n+1}, \phi_{n+1})({\rm Frob}_{q_{n+2}}) = (1,
-1/\gamma_0)
\]
to have positive upper density.  Note that the complement of these two
Chebotarev conditions associated to $f_{n+1}$ and $\phi_{n+1}$ on
primes of $\mc{Q}_2$ forms a set of density $t_0(1 - \frac{1}{|k|^2})$
within $Q_0$. Suppose that for our choice of $q_{n+1}$ the set of
desired second primes $q_{n+2} \in \mc{Q}_2$ has density $0$. Using
the limiting argument as before, we see that by changing $q_{n+1}$
there is a set $\mc{Q}_3 \subset \mc{Q}_2$ of positive upper density
such that for $q_{n+2} \in \mc{Q}_3$,
$f_{n+1}({\rm Frob}_{q_{n+2}}) = 1 = f_{n+2}({\rm Frob}_{q_{n+1}})$.

Having chosen $q_{n+1}$, for $q_{n+2}, q_{n+3} \in \mc{Q}_3$ we have
$[f_{n+i}({\rm Frob}_{q_{n+j}})]_{1 \leq i,j \leq 3} = \Bigl
[ \begin{smallmatrix} 0 & 1 & 1 \\ 1 & 0 & a \\ 1 & b &
  0\end{smallmatrix} \Bigr ]$. We want $a = b = 1$, so we fix $q_{n+2}$
and vary $q_{n+3}$. Choosing
\[
(f_{n+2}, \phi_{n+2})({\rm Frob}_{q_{n+3}}) = (1,
-1/\gamma_0)
\]
is a pair of Chebotarev conditions independent of those determining
$\mc{Q}_0$. If a $q_{n+3}$ exists satisfying these conditions we are
done. If not, using the same limiting argument as before we see that
we can change $q_{n+2}$ so that a $q_{n+3}$ satisfying these
conditions will indeed exist. Thus, we now have primes
$q_{n+1}, q_{n+2}, q_{n+3}$ such that
$[f_{n+i}({\rm Frob}_{q_{n+j}})]_{1 \leq i,j \leq 3} = \Bigl
[ \begin{smallmatrix} 0 & 1 & 1 \\ 1 & 0 & 1 \\ 1 & 1 &
  0\end{smallmatrix} \Bigr ]$.

We are now ready to settle Case (3). Set
\[
\begin{split}
Q = & \{q_1,q_2,\dots, q_{n+1},q_{n+2}, q_{n+3}\} \\
Q_i = &
\begin{cases}
\{q_1,q_2,\dots, q_n,q_{n+1}\} \bs \{q_i\} & i \leq n+1 \\
\{q_1, q_2,\dots, q_{n+1},q_{n+2}, q_{n+3}\} \bs \{q_i\} & i= n+2,
n+3\ .
\end{cases}
\end{split}
\]
That $Q$ is auxiliary follows from Lemma \ref{lem:aux}. For
$i\leq n+1$, that $Q_i$ is auxiliary follows as before from the fact
that the $n \times (n+1)$ matrices $F = [f_i({\rm Frob}_{q_j})]$ and
$\Phi = [\phi_i({\rm Frob}_{q_j})]$ are the $n \times n$ identity matrix
augmented by a column with no zero entries. That $Q_i$ is auxiliary
for $i = n+2, n+3$ follows from Lemma \ref{lem:aux} since the relevant
principal minors of the matrix
$[f_{n+i}({\rm Frob}_{q_{n+j}})]_{1 \leq i,j \leq 3}$ are invertible.

We show that $\rho_{S \cup Q}^{Q-new}$ mod $p^d$ is ramified at all
primes in $Q$. To get from $\rho_{S \cup \tilde{Q}}^{\tilde{Q}-new}$
mod $p^d$ to $\rho_{S \cup Q}^{Q-new}$ mod $p^d$ requires making the
representation special at $q_{n+1}, q_{n+2}, q_{n+3}$. The class $h =
\frac{1}{2}(f_{n+1} + f_{n+2} + f_{n+3})$ satisfies $h({\rm Frob}_{q_{n+i}})
= 1$, so adjusting by this class yields specialness \emph{and}
ramification at $q_{n+1}, q_{n+2}, q_{n+3}$. Note that $h|_{\Gamma_v}
= \frac{3}{2} g_v$ for $v \in S \cup \tilde{Q}$, so if $\rho_{S \cup
  \tilde{Q}}^{\tilde{Q}-new}$ mod $p^d$ is unramified at $v$ we are
introducing ramification and if it is ramified at $v$ we are not
removing ramification. By uniqueness as before, the adjusted
representation must be equal to $\rho_{S \cup Q}^{Q-new}$ mod $p^d$
which is therefore ramified at all primes $q \in Q$.

It remains to show that $\rho_{S \cup Q_i}^{Q_i-new}$ is not special
at $q_i$.  Recall that $\rho_{S \cup Q_i}^{Q_i-new}$ is unramified at
$q_i$ and $\rho_{S \cup \tilde{Q}}^{\tilde{Q}-new}$ mod $p^{d-1}$ is
unramified at $q_i$.
\begin{itemize}
\item $ i = n+2$: to get from
  $\rho_{S \cup \tilde{Q}}^{\tilde{Q}-new}$ mod $p^d$ to
  $\rho_{S \cup Q_{n+2}}^{Q_{n+2}-new}$ mod $p^d$ we adjust by
  $h = f_{n+1} + f_{n+3}$ to make the representation ramified and
  special at $q_{n+1}$ and $q_{n+3}$. As
  $h({\rm Frob}_{q_{n+2}}) = 2$, the adjusted representation is not
  special at $q_{n+2}$. For $v \in S \cup \tilde{Q}$,
  $h|_{\Gamma_v} = 2g_v$ so the representation remains special for
  such a $v$ (and becomes ramified at all $q_i \in \tilde{Q}$).
\item $i = n+3$: the same argument as above works with
  $h = f_{n+1} + f_{n+2}$.
\item $i = n+1$: this follows since $Q_{n+1} = \tilde{Q}$ and
  $\rho_{S \cup \tilde{Q}}^{\tilde{Q}-new}$ is not special at
  $q_{n+1}$ by construction.
\item $i \leq n$: we move from
  $\rho_{S \cup \tilde{Q}}^{\tilde{Q}-new}$ mod $p^d$ to
  $\rho_{S \cup Q_i}^{Q_i-new}$ mod $p^d$. This breaks into two cases.
  \begin{itemize}
  \item if $\rho_{S \cup \tilde{Q}}^{\tilde{Q}-new}$ mod $p^d$ is
    unramified at $q_i$ we use a nonzero multiple of $f_i$ to make the
    representation special at $q_{n+1}$. This keeps the representation
    unramifed at $q_i$ and also makes it nonspecial at $q_i$. 
  \item if $\rho_{S \cup \tilde{Q}}^{\tilde{Q}-new}$ mod $p^d$ is
    ramified at $q_i$, using that $f_{n+1}|_{\Gamma_{q_i}} = g_{q_i}$
    we adjust by $f_{n+1}$ to remove ramification at $q_i$. Then as
    $f_{n+1}({\rm Frob}_{q_{n+1}}) = 0$, the representation at $q_{n+1}$ is
    nonspecial before and after adjustment. Now use a nonzero multiple
    of $f_i$ to make the representation special at $q_{n+1}$. This
    keeps the representation unramified at $q_i$ and also makes the
    representation nonspecial at $q_i$.
    \end{itemize}
  \end{itemize}
  This completes the proof of Case (3) of Theorem \ref{thm:main} and
  hence also the proof of the theorem.

\end{proof}

\section{An exact computation of Selmer groups}\label{deformation}

In \S 2, and Theorem \ref{thm:main}, we  consider suitable
$\rbar:\Gamma_S \to G(k)$ and a set $S$ of primes of $F$ (including
all the ramified primes, the primes of $F$ dividing the residue characteristic $p$ of $k$ and all the infinite places) together with balanced local deformation  $\mc{D}_v$ conditions at $v \in S$.  We also consider nice primes $q$ and balanced  local deformation conditions $\mc{D}_q$ at $q$.  Given a finite set of nice  primes $Q$, there is an associated deformation ring $R_{S \cup Q}^{Q-new}$   that parametrizes deformations $\rho:\Gamma_{S \cup Q}  \to G(A)$ of $\rbar$ (of  fixed   mulitplier $\nu$),  for $A$ a complete Noetherian local $W(k)$-algebra with residue field $k$,   that satisfies the local  conditions $\mc{D}_v$ at $v \in S \cup Q$. We also consider the relaxed deformation ring $R_{S \cup Q}$ which parametrizes  deformations (of  fixed   mulitplier $\nu$)  $\rho:\Gamma_{S \cup Q}  \to G(A)$ of $\rbar$  that satisfy the local  conditions $\mc{D}_v$ at $v \in S$, but with no restriction at $v \in Q$.   The complete Noetherian local $W(k)$-algebras $R_{S \cup Q}$ and $R_{S \cup Q}^{Q-new}$ are in general quite mysterious.

A simple consequence of Theorem \ref{thm:main} is that there exists a finite set $Q$ of nice primes such that
$R_{S \cup Q}^{Q-new} \cong W(k)$ and $\rho_{S \cup Q}^{Q-new}$ is
ramified at all primes $q \in Q$. This was first proved for $G=GL_2$
in \cite{KR1} and similar statements for $GL_n$ appear in
\cite[Theorem 1]{Z} and for general $G$ in \cite[Theorem
3.16]{Patrikis-Ann}. There is a natural surjection
$\pi:R_{S \cup Q} \to R_{S \cup Q}^{Q-new}$. The main result of this
section is the following:

\begin{thm} \label{thm:selmer} %\marginpar{thm:selmer} %
  Let $\rbar: \Gamma_F \to G(k)$ satisfy the assumptions of Theorem
  \ref{thm:main} and let $Q$ be as in the conclusion of the
  theorem. For each $q \in Q$, let $m_q$ be the smallest integer such
  that $\rho_{S \cup Q}^{Q-new}$ is ramified at $q$ modulo
  $p^{m_q}$. Then
  \[
    (\ker(\pi))/(\ker(\pi))^2 \cong \bigoplus_{q \in Q} W(k)/p^{m_q}W(k) \, ;
    \]
  in particular, $(\ker(\pi))/(\ker(\pi))^2$ is a finite $W(k)$-module and $\Spec(R_{S
    \cup Q}^{Q-new})$ is an irreducible component of $\Spec(R_{S \cup Q})$.
\end{thm}

\begin{rem}
The abelian group   $ (\ker(\pi))/(\ker(\pi))^2 $ is the Pontryagin dual of a certain Selmer group (see Lemma \ref{lem:eta} below).
To contextualize the result in Theorem \ref{thm:selmer}, we recall that the automorphy lifting results of  \cite{W} proved finiteness of Selmer groups associated to Galois representations  arising from newforms $f$ and related them to special values of $L$-functions. Our result  proves the finiteness of a Selmer group, and a formula for its order,  associated to a very particular Galois representation that we construct in Theorem \ref{thm:main}, and is thus less general in scope (say in the case $G=GL_2$) than \cite{W}. Our methods are those of Galois cohomology, while those of \cite{W} use additionally (congruence properties of)  modular forms. Theorem \ref{thm:h2} below is again known  as a consequence of automorphy lifting theorems in the case of polarized Galois representations (see \cite{Allen}).  Our result proves the vanishing of $H^2(\Gamma_{S \cup Q}, \rho_{S \cup Q}^{Q-new}(\mg{g}^{\mr{der}})
    \otimes_{W(k)} K) $ for the ``designer'' Galois representation we construct in Theorem \ref{thm:main}. While our result is very  tailored  to the representation  $ \rho_{S \cup Q}^{Q-new}$, it is not covered by \cite{Allen} as $\rho_{S \cup Q}^{Q-new}$ may not be polarized.
\end{rem}

Finiteness of $(\ker(\pi))/(\ker(\pi))^2$, in fact that $(\ker(\pi))/(\ker(\pi))^2$ is isomorphic
to a submodule of $\oplus_{q \in Q} W(k)/p^{m_q}W(k)$, in the situation of the above theorem,    was first proved in \cite{KR1} when 
$G= GL_2$ (where equality was
conjectured).  A  version of the result   for $G=GL_n$ was also
proved in \cite[\S 4]{Z}. 
\subsection{Local lifts over $W(k)/(p^n)$ at nice primes}

The results of this subsection are not used in the sequel, but we
include them here to illustrate that at our nice primes $q$, deformations have ``rank one'' ramification.

The lemma that follows is presumably well-known but we include a proof
since we do not know a reference.
\begin{lem} \label{lem:tori}%
  Let $G$ be a split reductive group over $W(k)$, $T$ a split maximal
  torus of $G$, and $\bar{x} \in T(k)$ a regular semisimple
  element.  Let $x_n$ be an element of $G(W(k)/p^n)$ for some $n>1$
  reducing to $\bar{x}$. Then there is a split maximal torus $T'$ of
  $G$ with $x_n \in T'(W(k)/p^n)$ and $T'_k = T_k$.
\end{lem}

\begin{proof}
  Since $G$ is smooth, $x_n$ can be lifted to $x' \in G(W(k))$ so it
  suffices to show that there exists $T'$ as in the lemma with
  $x' \in T'(W(k))$.

  Let $C^0_G(x')$ be the identity component of the (scheme-theoretic)
  centralizer of $x'$ in $G$. The (scheme-theoretic) fibre of
  $C^0_G(x')$ over $\spec(k)$ is $T(k)$ since $\bar{x} \in T(k)$ is
  regular.  Its generic fibre has dimension at least $\dim(T)$, since 
  the centralizer of any element of $G(K)$ has dimension bounded below
  by $\dim(T)$. By semicontinuity of dimension it follows that the
  generic fibre also has dimension $\dim(T)$ and $C^0_G(x')$ is flat
  over $W(k)$. By Grothendieck's theorem on deformations of tori
  (\cite[Th\'eor\`eme X.8.1]{sga3II}, it follows that $C^0_G(x')$ is a
  torus which is also split (\cite[Lemma X.3.1]{sga3II}) since $T_k$
  is split. We may thus take $T'$ to be $C^0_G(x')$.
\end{proof}

\begin{rem} \label{rem:unique}
  The proof shows that $T'_{W(k)/p^n}$ is uniquely determined as the
  identity component of the centralizer in $G_{W(k)/p^n}$  of $x_n$.
\end{rem}

Let $F_q$ be a local field of residue characteristic not equal to $p$,
set $\Gamma_q$ to be the absolute Galois group of $F_q$ and consider
an unramified representation $\rbar:\Gamma_q \to G(k)$. We assume that
$\bar{x} = \rbar({\mr Frob}_q)$ is a regular semisimple element in
$T(k)$ and there exists a \emph{unique} root $\alpha \in \Phi(G,T)$
such that $\Gamma_q$ acts on $\mg{g}_{\alpha}$ via $\kappa$. (In
particular, this implies that $N(q)$, the order of the residue field
of $F_q$, is not $1$ mod $p$.)
\begin{lem} \label{lem:nicelifts}%
  Any lift $\rho_n: \Gamma_q \to G(W(k)/p^n)$ of $\rbar$ is tamely
  ramified and is determined as follows: Let $\sigma_q \in \Gamma_q$
  be a lift of ${\mr Frob}_q$ and let $\tau_q \in \Gamma_q$ be a lift
  of a generator of tame inertia. Then there is a split torus
  $T' \subset G$ such that $T_k = T'_k$,
  $\rho_n(\sigma_q) \in T'(W(k)/p^n)$ and
  $\rho_n(\tau_q) \in U_{\alpha'}(W(k)/p^n)$, where
  $\alpha' \in \Phi(G,T')$ is the unique root having the same
  reduction as $\alpha$ modulo $p$.
\end{lem}
 
\begin{proof}
  We will prove the lemma by induction on $n$, the statement for $n=1$
  being obvious.  The statement about tame ramification is clear (for
  all $n$) since $\rbar$ is unramified and the residue characteristic
  of $F_q$ is not $p$. Thus any lift $\rho_n$ is determined by
  $\rho_n(\sigma_q)$ and $\rho_n(\tau_q)$.

  We now assume $n>1$ and the statement known for $n-1$. So letting
  $\rho_{n-1}$ be the reduction of $\rho_n$ modulo $p^{n-1}$, we have
  that there exists a torus $T'$ of $G$ such that $x_{n-1}$, the
  reduction of $\rho_n(\sigma_q)$ modulo $p^{n-1}$, is in
  $T'(W(k)/p^{n-1})$ and
  $\rho_{n-1}(\tau_q) \in U_{\alpha'}(W(k)/p^{n-1})$. By Lemma
  \ref{lem:tori}, $x_n := \rho_n(\sigma_q)$ lies in a split maximal
  torus, which we may assume to be $T'$ (cf. Remark \ref{rem:unique}),
  and we need to show that $\rho_n(\tau_q) \in U_{\alpha'}(W(k)/p^n)$.

  Let $y_n$ be any lift of $\rho_{n-1}(\tau_q)$ in
  $U_{\alpha'}(W(k)/p^n)$ so we may write $\rho_n(\tau_q) = y_n. g$
  where $g \in \ker(G(W(k)/p^n) \to G(W(k)/p^{n-1})) \cong
  \mg{g}_k$. Write $g = t + \sum_{\beta'} g_{\beta'} $, where
  $t \in \mg{t}_k$, $\beta'$ runs over the roots of $T'$, and
  $g_{\beta'} \in \mg{g}_{\beta', k}$.  By the structure of tame
  inertia, we must have
  $x_n \rho_n(\tau_q) x_n^{-1} = \rho_n(\tau)^{N(q)} = (y_n\cdot
  g)^{N(q)}$. But the reduction of $y_n$ in $G(k)$ is the identity, so
  $y_n$ commutes with $g$, hence
  $(y_n\cdot g)^{N(q)} = y_n^{N(q)} \cdot g^{N(q)}$.

  On the other, hand we also have
  $x_n \rho_n(\tau_q)x_n^{-1} = x_n y_n x_n^{-1}\cdot x_n g x_n^{-1}=z_n \cdot (t + \sum_{\beta'} \beta'(x_n)g_{\beta'})$, for some
  $z_n \in U_{\alpha'}(W(k)/p^n)$ (since
  $y_n \in U_{\alpha'}(W(k)/p^n)$). Thus,
  \[
    y_n^{N(q)} \cdot g^{N(q)} =  z_n \cdot (t + \sum_{\beta'}
    \beta'(x_n)g_{\beta'})
  \]
  so
  \[
    z_n^{-1}y_n^{N(q)} = (1-N(q))t + \sum_{\beta'}(
    \beta'(x_n) - N(q)) g_{\beta'}.
  \]
  Since the LHS is in $\mg{g}_{\alpha'}$, using the uniqueness of the
  root $\alpha$ we see that $t = 0$ and $g_{\beta'} = 0$ for all
  $\beta' \neq \alpha'$. This implies that $\rho_n(\tau_q) \in
  U_{\alpha'}(W(k)/p^n)$ as desired.
  \end{proof}

\subsection{Selmer and dual Selmer groups modulo $p^n$}

Let $K$ be the quotient field of $W(k)$. For any finitely generated
free $W(k)$-module $M$ let $M_{\infty}$ denote
$M \otimes_{W(k)}(K/W(k))$ and for $n>0$ let $M_n \cong M/p^nM$ denote
the $p^n$-torsion in $M_{\infty}$. For each $v \in S \cup Q$ and
$n>0$, the choice of smooth local condition $\mc{D}_v$ at each such
prime, the given condition at primes in $S$ and the one defined in \S
\ref{sec:lc} at primes in $Q$, gives rise to a $W(k)$-submodule
${\mc{N}}_{v,n} \subset H^1(\Gamma_v,\rho_{S \cup
  Q}^{Q-new}(\mg{g}^{\mr{der}})_n)$; see, for example, \cite[\S
4]{FKP}. The module
$\mc{N}_{v,1}$ is the tangent space of the deformation condition,
denoted $\mc{N}_v$ in \S \ref{sec:main}. We let
${\mc{N}}_{v,n}^{\perp} \subset H^1(\Gamma_v,\rho_{S \cup
  Q}^{Q-new}(\mg{g}^{\mr{der}})^{*}_n)$ denote the annihilator of
${\mc{N}}_{v,n}$ with respect to the Tate duality pairing
(generalising the one in \ref{eq:ltd}). We let
${\mc{N}}_{v,\infty} \subset H^1(\Gamma_v,\rho_{S \cup
  Q}^{Q-new}(\mg{g}^{\mr{der}})_{\infty})$ be the direct limit of all
the ${\mc{N}}_{v,n}$.

We let ${\mc{N}}_{v,n}' = {\mc{N}}_v$ for $v \in S$ and let
${\mc{N}}_{v,n}' = H^1(\Gamma_v,\rho_{S \cup Q}^{Q-new}(\mg{g}^{\mr{der}})_n)$ for $v \in Q$.
\begin{itemize}
  \item[--]
For $\bullet$ being $n$ or $\infty$, we define the Selmer groups
$H^1_{\mc{N}_{\bullet}}(\Gamma_{S \cup Q}, \rho_{S \cup
  Q}^{Q-new}(\mg{g}^{\mr{der}})_{\bullet})$ to be, as in
Definition \ref{def:selmer}, the kernel of the map
  \[
    H^1(\Gamma_{S \cup Q}, \rho_{S \cup Q}^{Q-new}(\mg{g}^{\mr{der}})_{\bullet}) \to \bigoplus_{v \in S \cup Q} \frac{H^1(\Gamma_v,\rho_{S
        \cup Q}^{Q-new}(\mg{g}^{\mr{der}})_{\bullet})}{{\mc{N}}_{v,\bullet}}
  \]
  and define
  $H^1_{\mc{N}_{\bullet}'}(\Gamma_{S \cup Q}, \rho_{S \cup
    Q}^{Q-new}(\mg{g}^{\mr{der}})_{\bullet} )$ in the same way with
  ${\mc{N}}_{v,\bullet}$ replaced by ${\mc{N}}_{v,\bullet}'$. 
\item[--]
 We also define dual Selmer groups   $H^1_{\mc{N}^{\perp}_n}(\Gamma_{S \cup Q}, \rho_{S \cup
    Q}^{Q-new}(\mg{g}^{\mr{der}})_{n}^*)$
    to be the kernel of the map
  \[
    H^1(\Gamma_{S \cup Q}, \rho_{S \cup Q}^{Q-new}(\mg{g}^{\mr{der}})_{n}^*) \to \bigoplus_{v \in S \cup Q} \frac{H^1(\Gamma_v,\rho_{S
        \cup Q}^{Q-new}(\mg{g}^{\mr{der}})_{n}^*)}{{\mc{N}}_{v,n}^{\perp}}
  \]
and define  $H^1_{\mc{N}'^{\perp}_n}(\Gamma_{S \cup Q}, \rho_{S \cup
  Q}^{Q-new}(\mg{g}^{\mr{der}})_{n}^*)$ similarly.
% \item[--]
% Tate duality gives a perfect pairing of $W(k)$-modules:
% \[
%   H^1(\Gamma_v,\rho_{S \cup Q}^{Q-new}(\mg{g}^{\mr{der}})_{\infty})
%   \times   H^1(\Gamma_v,\rho_{S \cup Q}^{Q-new}(\mg{g}^{\mr{der}})^*)
%   \to K/W(k)
% \]
% and we let
% $\mc{N}_v^{\infty, \perp} \subset H^1(\Gamma_v,\rho_{S \cup
%   Q}^{Q-new}(\mg{g}^{\mr{der}})^*)$ be the annihilator of
% $\mc{N}_{v,\infty} \subset H^1(\Gamma_v,\rho_{S \cup
%   Q}^{Q-new}(\mg{g}^{\mr{der}})_{\infty})$. We then define the
% integral dual Selmer group
% $H^1_{\mc{N}^{\infty, \perp}}(\Gamma_v,\rho_{S \cup
%   Q}^{Q-new}(\mg{g}^{\mr{der}})^*)$ to be the kernel of the map
% \[
%   H^1(\Gamma_{S \cup Q},\rho_{S \cup Q}^{Q-new}(\mg{g}^{\mr{der}})^*) \to
%   \bigoplus_{v \in S \cup Q} \frac{H^1(\Gamma_v,\rho_{S \cup
%       Q}^{Q-new}(\mg{g}^{\mr{der}})^*)}{{\mc{N}}_v^{\infty,\perp}} \ .
% \]
\end{itemize}

\begin{lem} \label{lem:van} %\marginpar{lem:van}%
  $ $
  \begin{enumerate}
  \item
    $H^1_{\mc{N}_\bullet}(\Gamma_{S \cup Q}, \rho_{S \cup
      Q}^{Q-new}(\mg{g}^{\mr{der}})_{\bullet}) = 0$ and
    $H^1_{\mc{N}^{\perp}_\bullet}(\Gamma_{S \cup Q}, \rho_{S \cup
      Q}^{Q-new}(\mg{g}^{\mr{der}})_{\bullet}^*) = 0$ for all
    $\bullet$ ($= n$ or $\infty$);
\item
   $H^1_{\mc{N}'^{\perp}_{n}}(\Gamma_{S \cup Q}, \rho_{S \cup
     Q}^{Q-new}(\mg{g}^{\mr{der}})_{n}^* ) = 0$ for all $n>0$;
\item The map
  \[
    H^1_{\mc{N}_{n}'}(\Gamma_{S \cup Q}, \rho_{S \cup
      Q}^{Q-new}(\mg{g}^{\mr{der}})_{n} ) \to \bigoplus_{v  \in Q} \frac{H^1(\Gamma_v,\rho_{S
        \cup Q}^{Q-new}(\mg{g}^{\mr{der}})_{n})}{\mc{N}_{v,n}}
  \]
  is an injection for all $n$;
  % \item $H^1_{\mc{N}^{\infty, \perp}}(\Gamma_{S \cup Q},\rho_{S \cup
  % Q}^{Q-new}(\mg{g}^{\mr{der}})^*)$  is annihilated by a power of $p$.
% \item $H^1_{\mc{N}_{n}'}(\Gamma_{S \cup Q}, \rho_{S \cup
%       Q}^{Q-new}(\mg{g}^{\mr{der}})_{n} )$ has length $\sum_{q \in Q} m_q$
%     for all $n \gg 0$.
   \end{enumerate}
\end{lem}

\begin{proof}
  Item (1) for $\bullet = n$ follows from Lemma 6.1 of \cite{FKP} and
  the fact that since $Q$ is auxiliary, the statement holds for
  $n=1$. For $\bullet = \infty$ it follows by taking limits.

  Item (2) follows from (1) since
  $H^1_{\mc{N}'^{\perp}_{n}}(\Gamma_{S \cup Q}, \rho_{S \cup
    Q}^{Q-new}(\mg{g}^{\mr{der}})_{n}^* )$ is a submodule of
  $H^1_{\mc{N}^{\perp}_n}(\Gamma_{S \cup Q}, \rho_{S \cup
    Q}^{Q-new}(\mg{g}^{\mr{der}})_{n}^*)$ (which is immediate from the
  definitions).

Item (3) follows from the vanishing of   $H^1_{\mc{N}_n}(\Gamma_{S \cup Q}, \rho_{S \cup
    Q}^{Q-new}(\mg{g}^{\mr{der}})_{n})$ in (1) and the definition of the
  Selmer groups.

%   For $v \in S \cup Q$, let $\mc{N}_v^{n,\perp}$ be the image of
%   $\mc{N}_v^{\infty,\perp}$ in
%   $H^1(\Gamma_v,\rho_{S \cup Q}^{Q-new}(\mg{g}^{\mr{der}})_n^*)$; it
%   is contained in $\mc{N}_{n,v}^{\perp}$ by the definition of
%   $\mc{N}_v^{\infty,\perp}$. The finite generation of
%   $\mg{g}^{\mr{der}}$ as $W(k)$-module implies that kernel of the map
%   \[
%   H^1(\Gamma_v, \rho_{S \cup
%     Q}^{Q-new}(\mg{g}^{\mr{der}})_{n}^*) \to  H^1(\Gamma_v,
%   \rho_{S \cup Q}^{Q-new}(\mg{g}^{\mr{der}})_{\infty}^*)
% \]
% is killed by a power $p^{a_v}$ of $p$ depending only on $v$, so the
% same holds for $\mc{N}_{n,v}^{\perp}/ \mc{N}_v^{n,\perp}$. Then (1)
% implies that
% $H^1_{\mc{N}^{\infty, \perp}}(\Gamma_{S \cup Q},\rho_{S \cup
%   Q}^{Q-new}(\mg{g}^{\mr{der}})^*)$ is killed by $p^a$, where
% $a = \max \{a_v\}_{v \in S \cup Q}$, so (4) follows.
\end{proof}

\begin{lem} \label{lem:eta} %\marginpar{lem:eta} %
  There is a canonical isomorphism
  \[
    \Hom_{W(k)}((\ker(\pi))/(\ker(\pi))^2, K/W(k)) \cong H^1_{\mc{N}'}(\Gamma_{S
      \cup Q}, \rho_{S \cup Q}^{Q-new}(\mg{g}^{\mr{der}})_{\infty}) \ .
  \]
\end{lem}
\begin{proof}
  See \cite[Lemma 2.40]{DDT} for the proof in the case of $GL_n$
  (which works in general).
\end{proof}

% From (1) of Lemma \ref{lem:eta} we see that to prove the theorem it
% suffices to determine the structure of
% $ H^1_{\mc{N}}(\Gamma_{S \cup Q}, \rho_{S \cup
%   Q}^{Q-new}(\mg{g}^{\mr{der}}) \otimes_{W(k)} K/W(k))$. Moreover, it follows from
% (2) and the definition of the Selmer groups that there is an injection
% \[
%   H^1_{\mc{N}}(\Gamma_{S \cup Q}, \rho_{S \cup
%   Q}^{Q-new}(\mg{g}^{\mr{der}}) \otimes_{W(k)} K/W(k)) \to
%   \bigoplus_{q \in   Q} \frac{H^1(\Gamma_q,\rho_{S \cup
%       Q}^{Q-new}(\mg{g}^{\mr{der}})}{{\mc{N}}_v'}\ .
% \]
% We will prove the theorem by computing the RHS and showing that this
% map is also a surjection.

\subsection{Local computations II}

Let $q$ be a nice prime for $\rbar$ and let $\alpha$ be the root
associated to $q$. Then by definition of the local condition at $q$
(\S \ref{sec:lc}), any deformation of $\rbar|_{\Gamma_q}$ to $W(k)$ in
$\mc{D}_q$ has a representative (up to $\wh{G}(W(k))$-conjugacy)
$\rho:\Gamma_q \to G(W(k))$ factoring through $H_{\alpha}(W(k))$.  We
fix such a $\rho$ and we assume that it is ramified; let $m\geq 1$ be the
maximal integer such that $\rho$ is unramfied modulo $p^m$.

For any root $\beta \in \Phi(G,T)$ such that $\beta \neq \pm \alpha$,
we have that $[\mg{g}_{\alpha}, \mg{g}_{\beta}] \subset \mg{g}_{\alpha+\beta}$ if
$\alpha + \beta$ is a root and it is zero otherwise. It follows from
this and the definition of $H_{\alpha}$, that as an
$H_{\alpha}$-module (via the adjoint action) there is a decomposition
\[
  \mg{g}^{\mr{der}} = V_{\alpha} \oplus (\mg{t}_{\alpha} \cap \mg{g}^{\mr{der}}) \oplus W_{\alpha} \mbox{ with } V_{\alpha} = \mg{g}_{\alpha} \oplus \mg{l}_{\alpha} \oplus \mg{g}_{-\alpha}
 \mbox{ and } W_{\alpha} = \bigoplus_{\beta \neq
  \pm \alpha} \mg{g}_{\beta}\  .
\]
Furthermore, $W_{\alpha}$ has a finite filtration, with each term a
sum of root spaces, such that $U_{\alpha} \subset H_{\alpha}$ acts
trivially on the associated graded and as a $T$-module the associated
graded is simply $\bigoplus_{\beta \neq \pm \alpha} \mg{g}_{\beta}$
with its natural $T$-action. This decomposition and filtration clearly
induce a decomposition and filtration of $\mg{g}^{\mr{der}}$ as a
$\Gamma_q$-module.

\begin{lem} \label{lem:w} %\marginpar{lem:w}%
  $ $
  \begin{enumerate}
  \item $H^1(\Gamma_q, (W_{\alpha})_n) = 0$ for all
    $n>0$;
  \item
    $H^i(\Gamma_q, W_{\alpha}
    \otimes_{W(k)} K) = 0$ for $i=0,1,2$;
  \item $H^i(\Gamma_q, V_{\alpha} \otimes_{W(k)} K) = 0$ for $i=0,1,2$;
  \item For all $n \geq m$,
    \begin{align*}
       H^1(\Gamma_q, (V_{\alpha})_n) 
      = & H^1(\Gamma_q/I_q, (V_{\alpha})_n^{I_q}) \oplus
          \im(H^1(\Gamma_q, (\mg{g}_{\alpha})_n) \to
          H^1(\Gamma_q, (V_{\alpha})_n) \\
      \cong & W(k)/p^m \oplus W(k)/p^m ;
    \end{align*}
    \item
      $\im(H^1(\Gamma_q, (\mg{g}_{\alpha})_{\infty}) \to
      H^1(\Gamma_q, (V_{\alpha})_{\infty}) = 0$;
    \item $H^1(\Gamma_q, (V_{\alpha})_{\infty}) \cong W(k)/p^m$.
  \end{enumerate}
\end{lem}
\begin{proof}
  Items (1) and (2) follow by induction using the filtration discussed
  above and the fact that the analogous vanishing holds for each
  $\mg{g}_{\beta}$, with $\beta \neq \alpha$. This uses the assumption
  that $q$ is a nice prime so the character giving the
  $\Gamma_q$-action on $\mg{g}_{\beta}$ is neither the trivial
  character nor the cyclotomic character $\kappa$.

  For (3) we note that $V_{\alpha}$ also has a filtration as
  $\Gamma_q$-module
  \[
    0 = V_0 \subset V_1 \subset V_2 \subset V_3 = V_{\alpha}
  \]
  with $V_1 = \mg{g}_{\alpha}$, $V_2/V_1 = \mg{l}_{\alpha}$ and
  $V_3/V_2 = \mg{g}_{-\alpha}$. By the niceness assumption we have
  vanishing of all cohomology for $\mg{g}_{-\alpha}$. We conclude by
  using the known cohomology of the trivial (for $\mg{l}_{\alpha}$)
  and cyclotomic (for $\mg{g}_{\alpha}$) characters and noting that
  the extension
  \[
    0 \to \mg{g}_{\alpha} \otimes_{W(k)} K \to V_2 \otimes_{W(k)} K \to
    \mg{l}_{\alpha} \otimes_{W(k)} K \to 0
  \]
  is non-split since $\rho$ is ramified.

  For (4) we first use the vanishing of all cohomology of
  $\mg{g}_{-\alpha}$ as above to replace $V_{\alpha}$ by $V_2$. We
  then use the exact sequences
  \[
    0 \to (\mg{g}_{\alpha})_n \to (V_2)_n \to (\mg{l}_{\alpha})_n\to 0
    \ ,
  \]
  the known cohomology of the first and third (nonzero) terms and the
  structure of the action of $I_q$ (determined by the definition of
  $\mc{D}_q$). This gives that
  $(V_{\alpha})_n^{I_q} = (\mg{g}_{\alpha})_n \oplus
  (\mg{l}_{\alpha})_m$ and
  $H^1(\Gamma_q/I_q, (V_{\alpha})_n^{I_q}) \cong W(k)/p^m$ (again
  using that $q$ is a nice prime). The fact that
  $\im(H^1(\Gamma_q, (\mg{g}_{\alpha})_n) \to H^1(\Gamma_q,
  (V_{\alpha})_n) \cong W(k)/p^m$ follows by considering the long
  exact sequence of cohomology associated to the short exact sequence
  above: the structure of the action of inertia implies that the image
  of the boundary map
  $H^0(\Gamma_q, (\mg{l}_{\alpha})_n)\to H^1(\Gamma_q,
  (\mg{g}_{\alpha})_n)$ has image isomorphic to $W(k)/p^{n-m}$.

  The two submodules $H^1(\Gamma_q/I_q, (V_{\alpha})_n^{I_q})$ and
  $\im(H^1(\Gamma_q, (\mg{g}_{\alpha})_n) \to H^1(\Gamma_q,
  (V_{\alpha})_n)$ are seen to intersect trivially by using the map
  $ (V_2)_n \to (\mg{l}_{\alpha})_n$. We see that they generate all of
  $H^1(\Gamma_q, (V_{\alpha})_n)$ by computing
  $h^2(\Gamma_q, (V_{\alpha})_n)$ using \eqref{eq:ltd} and then using
  \eqref{eq:euler} to compute $h^1(\Gamma_q, (V_{\alpha})_n)$.
  
  Item (5) holds since a divisible $W(k)$-module which is killed by
  $p^m$ must be zero.

  Item (6) follows from (4) and (5) since the inclusions
  $(\mg{l}_{\alpha})_n \to (\mg{l}_{\alpha})_{\infty}$ induce
  injections on (unramified) cohomology.
  
\end{proof}

\subsection{Proof of Theorem \ref{thm:selmer}}

\begin{proof}[Proof of Theorem \ref{thm:selmer}]
  We apply the computations of Lemma \ref{lem:w} with
  $\rho = \rho_{S \cup Q}^{Q-new}$.
  
  The module
  ${\mc{N}}_{q,n} \subset H^1(\Gamma_q,\rho_{S \cup
      Q}^{Q-new}(\mg{g}^{\mr{der}})_{n})$ for $q \in Q$ corresponds to the
  summand
  \[
    \im(H^1(\Gamma_q, (\mg{g}_{\alpha})_n) \to H^1(\Gamma_q,
  (V_{\alpha})_n)) \oplus H^1(\Gamma_q, (\mg{t}_{\alpha} \cap
  \mg{g}^{\mr{der}}), 
\]
so by Lemma \ref{lem:w} (4) it follows that the quotient
$\frac{H^1(\Gamma_v,\rho_{S \cup
    Q}^{Q-new}(\mg{g}^{\mr{der}})_{n})}{{\mc{N}}_{v,n}}$ is isomorphic
to $W(k)/p^{m_q}$ if $n \geq m_q$. To prove the theorem it suffices,
by Lemmas \ref{lem:van} (3) and \ref{lem:eta}, to show that the length of
$H^1_{\mc{N}_{n}'}(\Gamma_{S \cup Q}, \rho_{S \cup
  Q}^{Q-new}(\mg{g}^{\mr{der}})_{n} )$ as $W(k)$-module is
$\sum_{q \in Q} m_q$  if $n \geq \max_q \{m_q\}$.
  
This follows by comparing the Greenberg--Wiles formula
%\cite[Theorem  2.19]{DDT}
for the Selmer conditions given by $\mc{N}_n$ and $\mc{N}_n'$. Since
the dual Selmer group for $\mc{N}_n'$ vanishes by (2) of Lemma
\ref{lem:van}, it follows from (1) of that Lemma and the definitions of
${\mc{N}}_{n,v}$ and ${\mc{N}}_{n,v}'$ for all $v \in S \cup Q$, that
the only contribution to the length of
$H^1_{\mc{N}_{n}'}(\Gamma_{S \cup Q}, \rho_{S \cup
  Q}^{Q-new}(\mg{g}^{\mr{der}})_{n} )$ comes from the primes in $Q$.
This contribution is precisely
  \[
    \ell(H^1(\Gamma_v,\rho_{S \cup
      Q}^{Q-new}(\mg{g}^{\mr{der}})_{n})) - \ell({{\mc{N}}_{v,n}}) = m_q .
  \]
\end{proof}

\subsection{Vanishing of $H^2$}

Let $R$ be the universal ring representing all deformations of
$\rbar:\Gamma_{S \cup Q}\to G(k)$ (with fixed determinant). There is a
natural surjection $\xi: R \to R_{S \cup Q}^{Q-new} \cong W(k)$.

\begin{thm} \label{thm:h2} %\marginpar{thm:h2}%
$ $
  \begin{enumerate}
   \item $H^2(\Gamma_{S \cup Q}, \rho_{S \cup Q}^{Q-new}(\mg{g}^{\mr{der}}) )  \hookrightarrow  \oplus_{v \in S \cup Q}  H^2(\Gamma_v, \rho_{S \cup Q}^{Q-new}(\mg{g}^{\mr{der}}))$;
  \item  $H^2(\Gamma_v, \rho_{S \cup Q}^{Q-new}(\mg{g}^{\mr{der}}) \otimes_{W(k)} K) = 0$ for all $v \in Q$;
  \item Thus, if
    $H^2(\Gamma_v, \rho_{S \cup Q}^{Q-new}(\mg{g}^{\mr{der}})
    \otimes_{W(k)} K) =0$ for all $v \in S$, then it follows that \\
    $H^2(\Gamma_{S \cup Q}, \rho_{S \cup Q}^{Q-new}(\mg{g}^{\mr{der}})
    \otimes_{W(k)} K) = 0$, and so $\xi$ gives rise to a formally
    smooth (closed) point of $\Spec(R[1/p])$.

\end{enumerate}

\end{thm}

\begin{rem} ${ }$
  \begin{itemize}
  \item The condition
    $H^2(\Gamma_v, \rho_{S \cup Q}^{Q-new}(\mg{g}^{\mr{der}})
    \otimes_{W(k)} K) = 0$ is referred to as
    $\rho_{S \cup Q}^{Q-new}(\mg{g}^{\mr{der}})$ being generic at $v$.
  \item Instead of the condition
    $H^2(\Gamma_v, \rho_{S \cup Q}^{Q-new}(\mg{g}^{\mr{der}})
    \otimes_{W(k)} K) = 0$ for all $v \in S$ we could assume the
    stronger condition that $H^2(\Gamma_v, \rbarg) = 0$ for all
    $v \in S$; this is perhaps more intrinsic since it does not depend
    on the set $Q$.
  \end{itemize}
\end{rem}

\begin{proof}
  % By Lemma \ref{lem:van}, (4), we have
  % $H^1_{\mc{N}^{\infty, \perp}}(\Gamma_{S \cup Q}, \rho_{S \cup
  %   Q}^{Q-new}(\mg{g}^{\mr{der}})^* \otimes_{W(k)} K) = 0$.

  For any $n \geq 1$, using the Poitou--Tate exact sequence (see,
  e.g., the proof of Theorem 2.18 of \cite{DDT}) we see that the dual
  of
  $H^1_{\mc{N}_n^{\perp}}(\Gamma_{S \cup Q}, \rho_{S \cup
    Q}^{Q-new}(\mg{g}^{\mr{der}})_n^*) $ surjects onto the kernel of
  $H^2(\Gamma_{S \cup Q}, \rho_{S \cup
    Q}^{Q-new}(\mg{g}^{\mr{der}})_n) \rightarrow \oplus_{v \in S \cup
    Q} H^2(\Gamma_v, \rho_{S \cup Q}^{Q-new} (\mg{g}^{\mr{der}})_n)$,
  so by Lemma \ref{lem:van} (1) it follows that this map is injective
  for all $n$.  The first of the lemma follows from this by taking
  inverse limits.

  For the second part, by Tate duality it suffices to show that
  $H^0(\Gamma_q, \rho_{S \cup Q}^{Q-new}(\mg{g}^{\mr{der}})^*
  \otimes_{W(k)} K ) = 0$ for all $q \in Q$. Since $Q$ consists of
  nice primes, there is a \emph{unique} root $\alpha$ of $\Phi(G,T)$,
  with $T$ the identity component of the centralizer of
  $\rbar({\rm Frob}_q)$ in $G$ (assumed to be a maximal split torus of
  $G$), such that $\Gamma_q$ acts on $\mg{g}_{\alpha}$ (the
  corresponding root space) by $\kappa$. By considering the action of
  a lift of Frobenius, it follows that the $\Gamma_q$-invariants of
  $(\mg{g}^{\mr{der}})^* = \mg{g}^{\mr{der}}(1)$ must be contained in
  $\mg{g}_{-\alpha}(1)$. However, since
  $\rho_{S \cup Q}^{Q-new}(\mg{g}^{\mr{der}})$ is ramified and inertia
  lies in $U_{\alpha}$, we see that the only $\Gamma_q$-invariant
  element in $\mg{g}_{-\alpha}(1)$ is $0$.

% This follows from the fact
% that $\rho_{S \cup Q}^{Q-new}$ is ramified for all $v \in Q$, and by
% choice the primes in $Q$ are nice for $\rhobar$, so in particular:
%   \begin{itemize}
%     \item[--] $\mr{Norm}(q)$ is not 1 mod $p$; 
%     \item[--] there is a \emph{unique} root $\alpha$ of $\Phi(G,T)$,
%       with $T$ the identity component of the centralizer of
%       $\rbar({\rm Frob}_q)$ in $G$ (assumed to be a maximal split torus of
%       $G$), such that $\Gamma_v$ acts on $\mg{g}_{\alpha}$ (the
%       corresponding root space) by $\kappa$.
%      \end{itemize}

The third  part follows directly from the first two parts.

\end{proof}

\section{Level lowering mod $p^n$ (via Theorem \ref{thm:main}) and modularity lifting}\label{sec:R=T}

Using Theorem \ref{thm:main} applied to odd irreducible
representations $\rbar:\Gamma_{\Q} \ra GL_2(k)$ we sketch a different proof
of modularity lifting along the lines of \S 4 of \cite{K}.  The method of  \cite{K} gives an approach to modularity lifting
which does not use patching techniques of Wiles, Taylor-Wiles.

We recall the situation of \cite{K}  and   then sketch an  approach to automorphy lifting  of loc. cit. which uses the quantitative level lowering results of the present paper. We consider an  irreducible, semistable  $\rhobar:G_\Q \ra GL_2(k)$ with a $k$ a finite field of characteristic $p>3$ (which by the remark before \cite[\S 2.1]{K} implies that our main theorem applies to $\rhobar$),  that  arises from a newform $f \in S_2(\Gamma_0(N(\rhobar)p^\delta)$ with $\delta=0,1$ according to whether $\rho$ is finite flat at $p$ or not. We assume that the minimal Selmer group for $\rhobar$ is non-zero (as otherwise we have an easy $R=\T$ theorem in the minimal case). Let $S$ be the set of places of $\Q$  dividing $N(\rhobar)p$ and the infinite place.

Consider a  finite set of primes $Q=\{q_1,\cdots, q_n\}$  such that $q_i$ is not $\pm 1 $ mod $p$ and prime to $N(\rhobar)p$, that are special
for $\rhobar$, i..e $\rhobar({\rm Frob}_{q_i})$ has eigenvalues with ratio $q_i$.  As  in \cite{K} we define rings $R_{S \cup Q}^{Q-new}$ that parametrize semistable  deformations of $\rhobar$ that have (in particular)

--   determinant the $p$-adic cyclotomic character, 

-- unramified outside the primes that divide $N(\rhobar)p$ and primes in $Q$, 

-- finite flat at $p$ if $\rhobar$ is finite flat at $p$, 

-- and Steinberg at $Q$.

Now we define the related Hecke rings  $\T_Q^{Q-new}$. For  a
subset $\alpha$ of $Q$ consider the module
$H^1(X_0(N(\rhobar)p^{\delta}Q),W(k))^{\alpha-{\rm new}}$ which is
defined as the maximal torsion-free quotient of the quotient of
$H^1(X_0(N(\rhobar)p^{\delta}Q),W(k))$ by the $W(k)$-submodule spanned
by the images of $H^1(X_0(N(\rhobar)p^{\delta}{Q \over q}),W(k))^{2}$
in $H^1(X_0(N(\rhobar)p^{\delta}Q),W(k))$, as $q$ runs through the
primes of $\alpha$, under the standard degeneracy maps (here and below
for a finite set of primes $Q$ we abusively denote by $Q$ again the
product of the primes in it).  We consider the standard action of
Hecke operators $T_r$ for all primes $T_r$ (note that we are using
$T_r$ for operators that normally get called $U_r$).

 By \cite{DT}
there is a maximal ideal $\sf m$ of the $W(k)$-algebra generated by
the action of these $T_r$'s such that:

\begin{itemize}

\item  $T_r -a_r \in \sf m$ for $a_r $
a lift to $W(k)$ of the trace of $\rhobar({\rm Frob}_r)$ when 
$(r,N(\rhobar)p^{\delta}Q)=1$;

\item   for $r \in Q$, $T_r -\tilde{\alpha_r} \in \sf m$
where $\alpha_r$ is the unique root of
the characteristic polynomial of $\rhobar({\rm Frob}_r)$ congruent to
$\pm 1$,  and $\tilde{\alpha_r}$ is
any lift of $\alpha_r$ to $W(k)$;

\item for  $r|N(\rhobar)p$:

--  if $r|N(\rhobar)$, or $r=p$ and $\rhobar$ is ordinary at $p$,  it is the scalar by which (the arithmetic
Frobenius) ${\rm Frob}_r$ acts on the unramified quotient of
$\rhobar|_{\Gamma_r}$  and $\tilde{\alpha_r}$ is
any lift of $\alpha_r$ to $W(k)$;

-- if $r=p$ and $\rhobar$ is not
ordinary at $p$, $T_p \in \sf m$.

\end{itemize}

 Then we define
${\T}_Q^{\alpha-{\rm new}}$ to be the localisation at $\sf m$ of
the $W(k)$-algebra generated by the action of these Hecke operators on
the finite flat $W(k)$-module
$H^1(X_0(N(\rhobar)p^{\delta}Q),W(k))^{\alpha-{\rm new}}$.  An analog
of Lemma 1 of \cite{K} gives that we have natural surjective maps
$R_{S \cup Q}^{\alpha-{\rm new}} \rightarrow {\T}_Q^{\alpha-{\rm new}}$
(where we take care of the fact that $T_r$'s are in the
image, including $r|N(\rhobar)p^{\delta}Q$,  using local-global compatibility results of Carayol as in Section 2 of [W]).

We need properties of ${\T}_Q^{\alpha-{\rm new}}$ that are
accessible (see \cite{H} and \cite{DT}) by exploiting another
description of these algebras that we recall for orienting the reader
(although we do not make explicit use of the alternative descriptions
in this sketch). For this fix a subset $\beta$ of the prime divisors
of $N(\rhobar)p^{\delta}$. Denote by $B_{\alpha,\beta}$ the quaternion
algebra over ${\bf Q}$ ramified at the primes in $\alpha \cup \beta$
and further at $\infty$ if the cardinality $n'=|\alpha \cup \beta|$ is
odd.  Denote by ${\bf A}$ the adeles over ${\bf Q}$.  For the standard
open compact subgroup
$U_{\alpha,\beta}:=U_0(N(\rhobar)p^{\delta}Q\alpha^{-1}\beta^{-1})$ of
the ${\bf A}$-valued points of the algebraic group $G_{\alpha,\beta}$
(over ${\bf Q}$) corresponding to $B_{\alpha,\beta}$,
$G_{\alpha,\beta}({\bf A})$, we consider the coset space
${\mathcal X}_{U_{\alpha,\beta}}=G_{\alpha,\beta}({\bf Q})\backslash
G_{\alpha,\beta}({\bf A}) /U_{\alpha,\beta}.$ Depending on whether
$n'$ is odd or even, this double coset space either is merely a finite
set of points, or can be given the structure of a Riemann surface
(that is compact if $n' \neq 0$ and can be compactified by adding
finitely many points if $n'$ is 0).  If $n'$ is odd we consider the
space of functions ${\mathcal S}_{U_{\alpha,\beta}}:=$
%\{f:G_{\alpha,\beta}({\bf Q})\backslash 
%G_{\alpha,\beta}({\bf A}^f)/U_{\alpha, \beta} 
$\{f:{\mathcal X}_{U_{\alpha,\beta}} \rightarrow W(k)\}$ modulo the
functions which factorise through the norm map, and in the case of
$n'$ even we consider the first cohomology of the corresponding
Riemann surface ${\mathcal X}_{U_{\alpha,\beta}}$, i.e.,
${\mathcal S}_{U_{\alpha,\beta}}:=H^1({\mathcal
  X}_{U_{\alpha,\beta}},W(k))$. These $W(k)$-modules have the standard
action of Hecke operators $T_r$.  

By
the results of \cite{DT} and the Jacquet--Langlands correspondence there is
a maximal ideal that we denote by $\sf m$ again in the support of the
$W(k)$-algebra generated by the action of the $T_r$'s on
${\mathcal S}_{U_{\alpha,\beta}}$ characterised as before.  We denote
the localisation at $\sf m$ of this Hecke algebra by
${\T'}_Q^{\alpha \cup \beta-{\rm new}}$.  Then by the
Jacquet-Langlands correspondence, which gives an isomorphism
${\bf T'}_Q^{\alpha \cup \beta-{\rm new}} \otimes {\bf Q}_p \simeq
{\bf T}_Q^{\alpha-{\rm new}} \otimes {\bf Q}_p$ that takes $T_r$ to
$T_r$, and the freeness of ${\T'}_Q^{\alpha \cup \beta-{\rm new}}$
and ${\T}_Q^{\alpha-{\rm new}}$ as $W(k)$-modules (which for the
former is a standard consequence of $\sf m$ being non-Eisenstein), we
have
${\T'}_Q^{\alpha \cup \beta-{\rm new}} \simeq {\T}_Q^{\alpha-{\rm new}}$, an isomorphism of local $W(k)$-algebras.
We consider an auxiliary set $Q=\{q_1,\cdots,q_n\}$ as in  Theorem \ref{thm:example} (see also Example
\ref{ex:gl2}) such that
$R_{S \cup Q}^{Q-new} \simeq {\T}_Q^{Q-new} \simeq W(k)$.

\begin{thm}\label{thm:augmentations}
Let 
$\rhobar:\Gamma_{\Q} \ra GL_2(k)$  be an odd, semistable, irreducible modular mod $p>3$ surjective representation that arises from $S_2(\Gamma_0(N(\rhobar)p^\delta)$.  Assume that the minimal dual Selmer for $\rhobar$ is not zero. Then there is a finite ordered set of primes $Q=\{q_1,\cdots,q_n\}$ that are special for $\rhobar$ such that:

\begin{itemize}

\item $R_{S \cup Q}^{Q-new}=W(k)$ and $R_{S \cup Q}^{Q-new}=\T_Q^{Q-new}$, and the corresponding Galois representation $\rho^{Q-new}:\Gamma_{\Q} \rightarrow GL_2(W(k))$, has the following  properties:

\item 
$\rho^{Q-new} (\tau_q)$, for $\tau_q$ a generator of the $\Z_p$-quotient of the inertia group $I_q$ at $q$,   is of the form
$\begin{pmatrix}
  1 &  p^d\\
  0 & 1 \\
  \end{pmatrix},$
for all $q \in Q$ for an integer $d \geq 1$;

\item For each $1 \leq i \leq r$, there
is a subset $Q_i$ of $Q$ that omits $q_i$ and contains
$\{q_1,\cdots,q_{i-1}\}$, such that $R_{S \cup Q_i}^{Q_i-new}=\T_{Q_i}^{Q_i-new}=W(k)$
and $\rho^{Q-new}$ is congruent 
modulo $p^d$ to $\rho^{Q_i-new}$.

\end{itemize}

\end{thm}

If the minimal dual Selmer for $\rhobar$ is zero, then as  there is a minimal modular lift of $\rhobar$, we have an automorphy lifting theorem in the minimal case. Thus the non-vanishing of the minimal dual Selmer is not too burdensome an assumption.

\begin{proof}

The assertions related to the set of primes $Q$ and subsets $Q_i$ with  $R_Q^{Q-new},R_{Q_i}^{Q_i-new}=W(k)$ and 
the congruences between $\rho^{Q_i-new}$ and $\rho^{Q-new}$ follow from Theorem \ref{thm:main} and its specialisation Theorem \ref{thm:example}. The assertions that $R_{Q_i}^{Q_i-new}=\T_{Q_i}^{Q_i-new}$ follow from this upon  using the level raising results of \cite{DT}.

\end{proof}

We sketch the proof  of the result below to go from restricted modularity lifting theorems to more general ones: we are reproving known results by a new method which exploits the independent level lowering congruences our work produces in Theorem \ref{thm:augmentations}.

\begin{cor}
 Keeping the notation above, we deduce an isomorphism  $R_{S \cup Q}=\T_Q$.
\end{cor}

\begin{proof}

The
argument is very similar to loc.~cit.~and the main ingredients are:
\begin{itemize}
\item[--] Wiles's numerical isomorphism criterion
\item[--] Level raising results of Diamond and Taylor in \cite{DT}
\item[--] Quantitative  level lowering as in Theorem \ref{thm:augmentations}
\item[--] Gorenstein properties of Hecke algebras arising from Shimura curves
(results of  \cite{H}; we may invoke these as for $q \in Q$, $q$ is not 1 mod $p$).
\end{itemize}

In loc.~cit.~we used an idea of level substitution mod $p^n$ to deduce
an $R_Q=\T_Q$ theorem from $R_Q^{Q-new}=\T_Q^{Q-new}$. 
Here we use use Theorem \ref{thm:augmentations}  instead of the level substitution step
of \cite[\S 4]{K}.   The strategy is to drop the newness conditions at   $\{q_n\}, \{q_n,q_{n-1}\},\ldots$, and prove successively starting with $R_{S\cup Q}^{Q-new}=\T_Q^{Q-new}$, that $R_{S\cup Q}^{Q\backslash \{q_n\}-new}=\T_Q^{Q\backslash \{q_n\}-new}$,   $R_{S\cup Q}^{Q\backslash \{q_n,q_{n-1}\}-new}=\T_Q^{Q\backslash \{q_n,q_{n-1}\}-new}, \ldots$, ultimately proving $R_{S \cup Q}=\T_Q$. We will focus on the first step, i.e., deduce $R_{S \cup Q}^{Q_n-new}=\T_Q^{Q_n-new}$ from $W(k)=R_{S\cup Q}^{Q-new}=\T_Q^{Q-new}$, where $Q_n=Q \backslash \{q_n\}$.

Using Theorem \ref{thm:augmentations}, we have  augmentations  $\pi:\T_Q \ra \T_Q^{Q-new}=W(k)$,
$\pi': \T_Q \ra \T_Q^{Q_n-new}=W(k)$ such
that $\pi$ and $\pi'$ are congruent mod $p^d$ and $\pi'$ does not factor through $\T_Q^{Q-new}$. Further it is known
(see Section 8 of \cite{H} and note that we are assuming our primes $q$
are not 1 mod $p$) the Hecke algebra $\T_Q^{Q_n-new}$ is Gorenstein. We denote  below, abusing notation,  by  $\pi$ all the morphisms $\T_Q^{\alpha-new} \ra W(k)$ induced by the augmentation $\pi:\T_Q \ra W(k)$ for all subsets $\alpha$ of $Q$.
  
  We claim that  the ideal
$\pi({\rm Ann}_{{\bf T}_{Q}^{Q_n-{\rm new}}}({\rm
  ker}(\pi)))$ of $W(k)$ is contained
in  $$p^n\pi({\rm Ann}_{{\bf T}_{Q }^{Q-{\rm
    new}}}({\rm ker}(\pi)))=(p^d).$$

To prove the claim 
the two ingredients are the morphism $\pi':{\T}_{Q }^{Q_r-{\rm
    new}} \rightarrow W(k)$ which {\it does not} factor through ${\T}_{Q}^{Q-{\rm
    new}}$ by construction, and the  fact   that
 the Hecke algebra $\T_Q^{Q_n-new}$ is  Gorenstein.  
 For ease of notation set 
   ${\T}'={\T}_{Q}^{Q_n-{new}}$
   and ${\T}={\T}_{Q }^{Q-{new}}$ and  $\beta: {\T}' \rightarrow {\T}$ the natural map. Recall that we are denoting by $\pi$ both the 
    fixed augmentation $\pi: \T \ra W(k)$ and its pullback to  $\T'$

    We justify the claim using these 2 ingredients as in \cite[proof
    of claim, pg. 216]{K} (see also Lemma \ref{gor}), an argument that
    is due to the referee of that paper.  Because ${\T}$ and ${\T}'$
    Gorenstein, we have that
$$\pi({\rm Ann}_{{\T}'}({\rm ker}(\pi)))=\pi({\rm
Ann}_{{\T}'}({\rm ker}(\beta)))\pi({\rm Ann}_{{\T}}({\rm ker}(\pi)))$$ as
ideals of $W(k)$. Choose a $x \in {\rm ker}(\beta)$ such
that $\pi'(x) \neq 0$.
Consider $y \in {\rm Ann}_{{\T}'}({\rm ker}(\beta))$. Then as $xy=0$ this
implies $\pi'(xy)=0$, which implies $\pi'(y)=0$, 
which implies $\pi(y) \in p^{d}W(k)$ as $\pi'$ and $\pi$
are congruent mod $p^{d}$ which proves the claim. 

This combined with the injection of the cotangent space at (the pull-back) $\pi:R_Q^{Q_r-new}$ into $W(k)/p^nW(k)$ (see Lemma \ref{lem:w}  above or \cite[Proposition 2]{K}), and the Wiles numerical criterion gives an isomorphism $R_Q^{Q_n-new} =\T_Q^{Q_n-new}$. Iterating the argument yields that $R_Q=\T_Q$.

\end{proof}

  \appendix
  
  \section{Wiles numerical isomorphism criterion}\label{sec:app}
  
  Wiles, in his work on the modularity of elliptic curves \cite{W},
  proved a numerical criterion for a complete Noetherian local
  $\cO$-algebra $A$ which is:
\begin{itemize}
\item[--]  finite flat over $\cO$,
\item[--]  and  with an {\it augmentation} $\pi_A: A \rightarrow \cO$ with  $\Phi_A:=(\ker \pi_A)/(\ker \pi_A)^2$ a finite abelian group
\end{itemize}
to be a complete intersection of dimension $1$.  Here $\cO$ is a
discrete valuation ring that is finite over $\Z_p$.  He used this
criterion to deduce that isomorphisms of deformation rings and Hecke
algebras at minimal level also imply such isomorphisms at non-minimal
levels.
 
 We generalize  in \S \ref{Wiles}  the Wiles criterion, which was later refined by Lenstra \cite{lenstra},  for rings $A$ that we do not assume are finite over $\cO$, but assume have  depth at least one (weaker than assuming $A$ is flat over $\cO$).  Our  results,  Propositions \ref{prop:ci} and   \ref{prop:isom},   are refinements of a theorem that is due to Wiles and Lenstra \cite[Theorem 5.27]{DDT}.  The proof of Proposition \ref{prop:ci}  is easily deduced from the work of Wiebe \cite{wiebe} and that of Proposition \ref{prop:isom}  is a variant of arguments of  Wiles and Lenstra.  
 
 Our motivation was to find a sufficient criterion given $(A,\pi_A)$
 with $\Phi_A$ finite, for $A$ being of dimension one.  Taking a cue
 from the Wiles--Lenstra criterion, we guessed this should be implied
 by the numerical equality $\#\Phi_A=\#\Oc/\eta_A$, with
 $\eta_A=\pi_A(\Ann(\ker(\pi_A)))$.
 %In the literature,  $\eta_A$ is often referred to as the congruence ideal and $\Oc/\eta_A$ as the congruence module. In other words, we guessed that the size  of the congruence module $\Oc/\eta_A$ of  the augmentation $\pi_A$  in comparison to the size of the cotangent space $\Phi_A$ of $\pi_A$  would control the size of $A$. 
 Proposition \ref{prop:ci} verifies that the numerical equality
 implies $A$ is a complete intersection of dimension 1 if and only if
 $A$ is of depth one.  We raise the question (without assuming that
 $A$ is of depth one) whether just the numerical equality implies that
 $A$ is of dimension 1.

\subsection{A refinement of Wiles's numerical isomorphism criterion}\label{Wiles}

All rings will be commutative, Noetherian and also local and
complete. Some of the statements make sense without the completeness
condition but they can be immediately reduced to the complete case.

We try to follow as much as possible the notations of \cite{DDT}. So
$\mc{O}$ will always be a discrete valuation ring but we do not
usually assume that it is a finite extension of $\Z_p$ or that the
rings being considered are $\mc{O}$-algebras.

For a complete local ring $(A, m_A)$ as above we consider a surjection
$\pi_A:A \to \mc{O}$ such that $\Phi_A := \ker (\pi_A)/(\ker (\pi_A))^2$
is of finite length.  It follows, as in \cite[Section 5.2]{DDT}, that
for $\eta_A := \pi_A (\Ann (\ker (\pi_A)))$ we have that $\mc{O}/\eta_A$
is also of finite length and
\begin{equation}
\ell(\Phi_A) \geq \ell(\mc{O}/\eta_A) \ .
\end{equation}
Hereafter, we will refer to the data $\pi_A:A \to \mc{O}$
as above simply as an augmented ring.

Our first goal is to prove the following generalisation of the
Wiles--Lenstra criterion for complete intersections.
\begin{prop} \label{prop:ci} Let $\pi_A: A \to \mc{O} $ be an
  augmented ring such that $\ell(\Phi_A)$ is finite and
  $\mr{depth}(A) \geq 1$. Then $A$ is a complete intersection ring iff
  $\ell(\Phi_A) = \ell(\mc{O}/\eta_A)$.
\end{prop}

The main improvement over the original result of Wiles, as extended by
Lenstra, is that we have no other finiteness assumption on $A$. The
condition on the depth is implied by the finite free condition in the
setup of Wiles.

\begin{proof}
  We first prove the forward direction:

  For an ideal $I$ in a ring $R$, we denote by $\Fitt(I)$ the zeroth
  Fitting ideal of $I$ and recall that $\Fitt(I) \subset \Ann(I)$.
  Therefore $\Fitt(\ker (\pi_A)) \subset \Ann(\ker (\pi_A))$. The
  equality $\ell(\Phi_A) = \ell(\mc{O}/\eta_A)$ implies that
  $\pi_A(\Fitt(\ker (\pi_A))) = \pi_A(\Ann(\ker (\pi_A)))$. Since
  $\Phi_A$ is finite, the supports of $\ker(\pi_A)$ and
  $\Ann(\ker(\pi_A))$ intersect only in the closed point of
  $\spec(A)$.  Since the depth of $A$ is at least one, $A$ cannot
  contain any nonzero submodule supported on the closed point of
  $\spec(A)$, hence $\Ann(\ker (\pi_A)) \cap \ker (\pi_A) = \{0\}$ and
  $\pi_A$ is injective when restricted to $\Ann (\ker (\pi_A))$. It
  follows that $\Fitt( \ker (\pi_A)) = \Ann(\ker (\pi_A))$.

  The set of zero-divisors in any ring is the union of the finte set
  of associated primes. Since the maximal ideal of $A$ is not an
  associated prime, it follows from the ``prime avoidance lemma''
  (\cite[Lemma 1.2.2]{bruns-herzog}) that there exists a non-zero
  divisor $x \in A$ such that $\pi_A(x)$ is a uniformizer of $\mc{O}$;
  equivalently, $x$ and $\ker (\pi_A)$ generate $m_A$.
  Let $\ov{A} := A/(x)$ and let $p: A \to \ov{A}$ be the quotient
  map. By construction, we have that $p(\ker (\pi_A)) = m_{\ov{A}}$ and
  so $p(\Ann (\ker (\pi_A))) \subset \Ann (m_{\ov{A}})$. 

  Let $y$ be any element of $\Ann (\ker (\pi_A))$ with $\pi_A(y)$ a
  generator of $\eta_A$. If $p(y) = 0$, then there exists $a \in A$ so
  that $y = ax$. For any $b \in \ker \pi_A$, we have $by = abx = 0$.
  Since $x$ is a non-zero divisor, we must have $ab = 0$ and so
  $a \in \Ann (\ker (\pi_A))$. But since $y = ax$ and $\pi_A(x)$ is not
  a unit, this is a contradiction. Thus,
  $p(\Ann (\ker (\pi_A))) \neq \{0\}$ and so $\Fitt(m_{\ov{A}}) \neq
  0$.
  It then follows from a theorem of Wiebe (\cite[Satz 3]{wiebe},
  \cite[Theorem 2.3.16]{bruns-herzog}) that $\ov{A}$ is a zero
  dimensional complete intersection ring and so, since $x$ is a
  non-zero divisor, that $A$ is also a complete intersection ring (of
  dimension one).

  The converse direction is an easy consequence of a result of Tate as
  in \cite[Satz 2]{wiebe}: this shows that
  $\Fitt(\ker (\pi_A)) = \Ann(\ker (\pi_A))$ and so the desired
  equality follows by applying $\pi_A$. Moreover, since $A$ is a one
  dimensional complete intersection, its depth is equal to one.
\end{proof}

We note that de Smit and Schoof \cite{dSRS} have also used Wiebe's
theorem to give a generalisation of the the Wiles--Lenstra criterion,
but in a different direction.

\smallskip

The depth condition in Proposition \ref{prop:ci} is essential: if $A$
is a complete intersection and $I \subset (\ker (\pi_A))^n$ is any
ideal, then the equality of lengths continues to hold for $A/I$ (with
its induced augmentation) if $n \gg 0$ (cf. \cite[Remark 5.2.5]{DDT}).
However, all examples that we know of are of dimension one. We are
thus led to ask:
\begin{ques}\label{first}
 Let $\pi_A: A \to \mc{O} $ be an augmented ring. If $\ell(\Phi_A)$
  is finite and equal to $\ell(\mc{O}/\eta_A)$,  then is $A$ one
  dimensional?
\end{ques}

We say that an augmented ring $(A,\pi_A)$ is well presented  if $A\simeq \cO[[X_1,\cdots,X_r]]/(f_1,\cdots,f_r)$. An answer to the following related question would also be relevant in applications:
\begin{ques}
 Let $\pi_A: A \to \mc{O} $ be an augmented ring that is well-presented. If $\ell(\Phi_A)$
  is finite and equal to $\ell(\mc{O}/\eta_A)$,  then is $A$  a complete intersection of dimension one?
\end{ques}

A stronger version of Question  \ref{first}, suggested by some computer calculations, is
\begin{ques}
 Let $\pi_A: A \to \mc{O} $ be an augmented ring. Is it always true
 that 
\[
\ell(\Phi_A) \geq \ell(\mc{O}/\eta_A) + \dim(A) -1 \ ?
\]
\end{ques}

  Here is a modest result in support of an affirmative answer to the question above.
  
\begin{lem}
  Let $\pi_A: A \to \mc{O} $ be an augmented ring of dimension $d$. If
  $A$ is a quotient of a regular local ring $R$ with
  $\dim(R) = d + 1$, then
  $\ell(\Phi_A) \geq \ell(\mc{O}/\eta_A) + d - 1$.
\end{lem}

\begin{proof}
  We just sketch the proof since we do not use it later.

  If $\dim(A) =1$, there is nothing to prove so we may assume that
  $\dim(A) > 1$. Choose a surjection $p:R \to A$ as in the statement
  of the lemma and let $f_1,\dots,f_r$ be generators of $J = \ker(p)$.
  Let $D \subset \Spec(A)$ be an irreducible component with
  $\dim(D) = \dim(A)$.  Since $R$ is a UFD, $D$ is the zero set of an
  irreducible element $g \in m_R$, and we have for all $i$,
  $f_i = gf_i'$ for some $f_i'$ in $m_R$. Let $J' = (f_1',\dots,f_r')$
  and set $A' = R/J'$. Since $g \notin K := \ker \pi_A\circ p$,
  $f_i' \in K$ for all $i$, so we have a factorisation of $\pi_A$
  through a map $\pi_{A'}:A' \to \mc{O}$.
 
  By construction, we have $J = gJ'$. This implies that
  $(J:K) = g(J':K)$ and so
  \[
  \pi_A(\Ann (\ker (\pi_A))) = \pi_A \circ p(g) \,\pi_{A'}(\Ann(\ker
  (\pi_{A'}))) \ .
  \]
  On the other hand, we have
  \[
  \pi_A (\Fitt(\ker (\pi_A))) = \pi_A \circ p(g^d) \,\pi_{A'}(\Fitt(\ker
  (\pi_{A'}))) \ ,
  \]
  since $K$ is generated by a subset of a regular system of parameters
  of $R$ of size $d$. We conclude using the trivial inequality
  $\ell(\Phi_{A'}) \geq \ell(\mc{O}/\eta_{A'})$ and the fact that
  $\pi_A \circ p(g)$ is not a unit in $\mc{O}$
\end{proof}

\subsection{The isomorphism criterion}

The isomorphism criterion of Wiles--Lenstra \cite[Theorem 5.28]{DDT}
can also be extended to our setting.

\begin{prop} \label{prop:isom} Let $\phi: A \to B$ be a surjective map
  of augmented rings with $B$ of depth one.  If
  $\ell(\Phi_A) \leq \ell(\mc{O}/\eta_B) < \infty $ then the map is an
  isomorphism and the rings are complete intersections.
\end{prop}
\begin{proof}
The standard inequalities imply that 
\[
\ell(\Phi_A) = \ell(\Phi_B) = \ell(\mc{O}/\eta_B) =
\ell(\mc{O}/\eta_A) \ 
\]
so it follows from Proposition \ref{prop:ci} that $B$ is a complete
intersection. Let $x$ in $B$ be a non-zero divisor mapping to a
uniformizer of $\mc{O}$ and let $x'$ be any element of $A$ lifting
$x$; $x'$ and $\ker \pi_A$ then generate $m_A$. It follows that
$\Fitt(m_{A/x'}) \neq \{0\}$ since it maps onto $\Fitt(m_{B/x})$.  We
conclude that $A/x'$ is also a zero dimensional complete intersection,
so $A$ is one dimensional.

Now by applying Lemma \ref{lem:resolve} we choose a complete
intersection augmented ring $\tilde{A}$ with a surjective map to $A$
inducing an isomorphism $\Phi_{\tilde{A}} \to \Phi_A$. Since
\[
\ell(\Phi_{A}) = \ell(\Phi_{\tilde{A}}) \geq \ell(\mc{O}/\eta_{\tilde{A}}) \geq
  \ell(\mc{O}/\eta_B) = \ell(\Phi_B) ,
\] 
we deduce that $\ell(\mc{O}/\eta_{\tilde{A}}) = \ell(\mc{O}/\eta_B)$.
Finally, by applying Lemma \ref{lem:isom2} we deduce that $\phi$ is an
isomorphism.
\end{proof}

\begin{lem} \label{lem:resolve} Let $A$ be a one dimensional augmented
  ring such that $\Phi_A$ is of finite length. Then there is a map of
  augmented rings $\tilde{A} \to A$ which induces an isomorphism
  $\Phi_{\tilde{A}} \to \Phi_A$ and such that $\tilde{A}$ is a
  complete intersection.
\end{lem}

This is the version of \cite[Theorem 5.26]{DDT} that we shall need.

\begin{proof}
  Let $R$ be a regular local ring of dimension $d+1$ with a surjection
  $\psi: R \to A$. We view $R$ as an augmented ring via the map
  $\pi_R := \pi_A \circ \psi$. Since $\mc{O}$ is a dvr, $\ker \pi_R$
  is generated by $d$ elements, so $\Phi_R$ is a free $\mc{O}$-module
  of rank $d$. Therefore, the kernel $K$ of the map
  $\Phi_R \to \Phi_A$ is also generated by $d$ elements. We let
  $f_1,\dots,f_d$ be elements in $I := \ker \psi$ whose images in
  $\Phi_R$ generate $K$; this is possible since $\psi$ induces a
  surjection $(\ker (\pi_R))^2 \to (\ker (\pi_A))^2$. 

  Let $f_1' = f_1$. Having chosen $f_1',\dots,f_i'$ in $I$ for some
  $i$, $1\leq i<d$, so that the dimension of
  $R_i = R/(f_1',\dots,f_i')$ is $d+1 -i$ and such that
  $f_i' \equiv f_i \mod (\ker (\pi_R))^2$, we apply the prime avoidance
  lemma \cite[Lemma 1.2.2]{bruns-herzog}\footnote{In the notation of \cite{bruns-herzog} we take
    $M= R$, $x_1 = f_{i+1}$ and $x_2,\dots,x_n$ to be generators of
    $I \cap (\ker (\pi_R))^2$} to the primes corresponding
  to the generic points of the irreducible components of
  $\spec(R_i)$ to find $f_{i+1}' \in I$ such that
  $f_{i+1}' \equiv f_{i+1} \mod (\ker (\pi_R))^2$ and
  $R_{i+1} = R_i/(f_{i+1}')$ has dimension equal to $d -i$.

  We let $\tilde{A} :=R_d = R/(f_1',\dots,f_d')$. By construction
  $\tilde{A}$ is one dimensional and since
  $f_i' \equiv f_i \mod (\ker \pi_R)^2$, the map
  $\Phi_{\tilde{A}} \to \Phi_A$ is an isomorphism.
\end{proof}

\begin{lem} \label{lem:isom2} Let $\phi:A \to B$ be a surjective
  morphism of augmented rings. If $A$ is a complete intersection, $B$
  has depth one, and $\eta_A = \eta_B \neq (0)$, then $\phi$ is an
  isomorphism.
\end{lem}
This is essentially \cite[Theorem 5.24]{DDT}; the depth condition
replaces the flatness condition therein. Also, as in
\emph{loc.~cit.}~we only need that $A$ be Gorenstein.

\begin{proof}
  We follow the proof of \cite[Theorem 5.24]{DDT} with appropriate
  modifications. We first note that since $\eta_A \neq (0)$ and $A$ is
  a complete intersection we have that $A$ is one dimensional and then
  since $B$ is a quotient of $A$ and is augmented, it is also one
  dimensional.

  Since $A$ is a one dimensional complete intersection, its depth is one, so as in the
  proof of Proposition \ref{prop:ci}, we have
\begin{equation}
  \ker (\pi_A) \cap \Ann_A (\ker (\pi_A)) = (0)
\end{equation}
and similarly for $B$.  As in the proof of \cite[Theorem 5.24]{DDT},
we get an exact sequence
\begin{equation}
0 \to \ker (\phi) \oplus \Ann_A (\ker (\pi_A)) \to A \to B/(\Ann_B (\ker (\pi_B)))
\to 0
\end{equation}
of $A$-modules. Furthermore, we have a natural injection
\[
B/(\Ann_B (\ker (\pi_B))) \to \End_B(\ker (\pi_B)) \ .
\]
Since $B$ has depth one, $\ker (\pi_B)$ also has depth one since an
element of $B$ which is a non-zero divisor for any $B$-module is
trivially a non-zero divisor of any submodule. This implies that
$\End_B(\ker (\pi_B))$ and so also $B/(\Ann_B \ker (\pi_B))$ have depth
one (use the same element!).

Now we use that $A$ is Gorenstein, so by Lemma \ref{lem:gor} we have
$\Ext_A^1( B/(\Ann_B (\ker (\pi_B))), A) = 0$.  Thus we get a
surjection
\begin{equation} \label{eqn:dual}
A \cong \Hom_A(A,A) \to \Hom_A(\ker (\phi), A) \oplus \Hom_A(\Ann_A (\ker
(\pi_A)), A) \to 0 \ .
\end{equation}
Both summands in \eqref{eqn:dual} are non-zero---consider the
tautological inclusions---if $\ker (\phi)$ is non-zero. This leads to a
contradiction by tensoring the sequence with $A/m_A$ and using
Nakayama's lemma as in \cite{DDT}.

\end{proof}

\subsection{Annihilators and Gorenstein rings}

The statement of Lemma \ref{gor} below (when $B$ is Gorenstein) is
very close to a statement on page 216 of \cite{K}, which is due to the
referee of that paper.
 
We first collect some well-known facts about Gorenstein local rings.
\begin{lem} \label{lem:gor} 
  Let $A$  be a Gorenstein local ring. Then
\begin{enumerate}
\item For any finite $A$-module $M$ with $\depth(M) = \dim(A)$, we
  have $\Ext_A^i(M,A) = 0$ for $i>0$. Consequently, the functor
  $M \mapsto M^* :=\Hom_A(M,A)$ preserves exact sequences of such
  modules; moreover, $(M^*)^*$ is canonically isomorphic to $M$.
\item Let $B$ be a Cohen--Macaulay local ring of the same dimension as
  $A$ and $\phi: A \to B$ a homomorphism which makes $B$ into a finite
  $A$-module. Then $B^* \cong B$ as an $A$-module iff $B$ is
  Gorenstein.
\end{enumerate}
\end{lem}

\begin{proof}
  The first part of (1) follows from \cite[Corollary
  3.5.11]{bruns-herzog} since $A$, being Gorenstein, is a canonical
  module for itself. The preservation of exact sequences then follows
  from the vanishing of $\Ext_A^1$.  The isomorphism of $(M^*)^*$ is a
  consequence of the fact that $A$ (as a complex in degee $0$) is a
  dualising complex for $A$ and then the vanishing of the $\Ext_A^i$ for
  $i>0$ implies that $\mr{RHom}_A(M,A) = M^*$.

  (2) follows from duality for the finite map $\phi$ : since $A$ is
  Gorenstein we have that $\mr{RHom}_A(B,A)$, which is isomorphic to
  $B^*$ under the depth assumption, is a dualising complex for $B$, so
  it is isomorphic to $B$ iff $B$ is Gorenstein.

\end{proof}

% In the following lemma, $R$ denotes $\mc{O}$ or $\mc{O}/\varpi^n$,
% where $\varpi$ is a uniformizer of $\mc{O}$. Note that both these
% rings are Gorenstein. 

\begin{lem}\label{gor}

  Let $A$ be a one dimensional Gorenstein local ring with an
  augmentation $\pi_A:A \to \mc{O}$ with $\ell(\Phi_A) < \infty$. If
   $\pi_A$ factors through a surjection $\phi: A \to B$ with $B$ a
  Cohen--Macaulay ring, then
\[
\pi_A(\Ann_A(\ker(\pi_A))) = \pi_A(\Ann_A(\ker(\phi))) \,
\pi_B(\Ann_B(\ker(\pi_B))) \ 
\]
where $\pi_B:B \to \mc{O}$ is the induced surjection.
\end{lem}

\begin{proof}
  If $\phi$ is an isomorphism the formula is a tautology, so may
  assume this is not the case. Then $\ker(\phi)$, being a nonzero
  submodule of $A$, has depth one. Of course $B$ and $\mc{O}$ also
  have depth one. Applying the functor $\Hom_A(\mc{O}, -)$ to the
  exact sequence of $A$-modules
\begin{equation} \label{eq:ker}
0 \to \ker(\phi) \to A \to B \to 0
\end{equation}
and using Lemma \ref{lem:gor}, we get an exact sequence
\begin{equation*} \label{eq:coker}
0 \to \Hom_A(\mc{O},\ker(\phi)) \to \Hom_A(\mc{O}, A) \to
\Hom_A(\mc{O}, B) \to \Ext^1_A(\mc{O}, \ker(\phi)) \to 0 .
\end{equation*}
Now
$\Hom_A(\mc{O},A) = \Ann_A(\ker(\pi_A)) \cong
\pi_A(\Ann(\ker(\pi_A)))$ where the second isomorphism follows from
the fact that $A$ has depth one and the kernel of the map
$\Ann_A(\ker(\pi_A)) \to \pi_A(\Ann_A(\ker(\pi_A)))$ is supported on
the closed point of $\spec(A)$ by $\ell(\Phi_A) < \infty$. The same
statements hold with $A$ replaced by $B$ since $B$ also has depth one
and $\ell(\Phi_B) \leq \ell(\Phi_A)$. It follows that to prove the
lemma we must show that
\begin{equation} \label{eq:mainformula}%
\ell(\Ext_A^1(\mc{O}, \ker(\phi))) =
\ell(\mc{O}/\pi_A(\Ann_A(\ker(\phi)))) .
\end{equation}

Dualising the sequence \eqref{eq:ker} (and using Lemma \ref{lem:gor})
we get an exact sequence
\begin{equation*} \label{eq:dker}%
  0 \to B^* \to A^* \to \ker(\phi)^* \to 0 .
\end{equation*}
Of course $A^* = A$ and $B^* = \Ann_A(\ker(\phi))$. We now apply the
functor $\Hom_A(-, \mc{O}^*)$ to this sequence to get an exact sequence
\begin{equation*} \label{eq:dual}%
  0 \to \Hom_A(\ker(\phi)^*, \mc{O}^*) \to \Hom_A(A, \mc{O}^*) \stackrel{p}{\to}
  \Hom_A(\Ann_A(\ker(\phi)), \mc{O}^*) \to \Ext_A^1(\ker(\phi)^*, \mc{O}^*) \to 0
  .
\end{equation*}
Since
$\mc{O}^* = \Hom_A(\mc{O},A) = \Ann_A(\ker(\pi_A)) \cong
\pi_A(\Ann_A(\ker(\pi_A))) \subset \mc{O}$ the map $p$ is the same as
the natural map
$\pi_A(\Ann(\ker(\pi_A))) \to \Hom_{\mc{O}}(\Ann_A(\ker(\phi))
\otimes_A \mc{O}, \pi_A(\Ann(\ker(\pi_A))))$ induced by the inclusion
of $\ker(\phi)$ in $A$. Tensoring the exact sequence
\[
  0 \to \Ann_A(\ker(\phi)) \to A \to A/\Ann_A(\ker(\phi)) \to 0
\]
with $\mc{O}$ we see that the kernel of the natural surjective map
$\Ann_A(\ker(\phi)) \otimes_A \mc{O} \to \pi_A(\Ann_A(\ker(\phi)))$ is
torsion: this follows from the condition $\ell(\Phi_A) < \infty$ which
implies that $\ell(\mr{Tor}_1^A(A/\Ann(\ker(\phi)), \mc{O}))$ is
finite. Thus,
\[ \Hom_{\mc{O}}(\Ann_A(\ker(\phi)) \otimes_A \mc{O},
  \pi_A(\Ann(\ker(\pi_A)))) = \Hom_{\mc{O}}(\pi_A(\Ann_A(\ker(\phi))),
  \pi_A(\Ann(\ker(\pi_A)))) ,
\]
so we deduce that
$\ell(\mr{Coker}(p)) = \ell(\mc{O}/\pi_A(\Ann_A(\ker(\phi))))$. By
duality, i.e, Lemma \ref{lem:gor}, it follows that
\eqref{eq:mainformula} holds, so the lemma is proved.
\end{proof}

\begin{rem} \label{rem:cm}
  Although we only use Lemma \ref{gor} when $B$ is Gorenstein, the
  case that $A$ is a complete intersection (so Gorenstein) and $B$
  is Cohen--Macaulay plays a crucial role in other applications \cite{bkm1}.
\end{rem}

\bibliographystyle{alpha}
\bibliography{useful}

\end{document}